\documentclass[aap]{imsart}

\RequirePackage{amsthm,amsmath,amsfonts,amssymb}
\RequirePackage[numbers,sort&compress]{natbib}
\RequirePackage[colorlinks,citecolor=blue,urlcolor=blue]{hyperref}
\RequirePackage{graphicx}
\RequirePackage{tikz}
\usetikzlibrary{patterns}
\RequirePackage{stmaryrd} 
\usepackage{ulem}

\startlocaldefs
\numberwithin{equation}{section}
\numberwithin{figure}{section}

\theoremstyle{plain}
\newtheorem{theorem}{Theorem}[section] 
\newtheorem{thm}[theorem]{Theorem}
\newtheorem{lem}[theorem]{Lemma}
\newtheorem{prop}[theorem]{Proposition} 
\newtheorem{cor}[theorem]{Corollary}
\newtheorem{notation}[theorem]{Notation}

\theoremstyle{remark} 
\newtheorem{assumption}[theorem]{Assumption}
\newtheorem{rem}[theorem]{Remark}
\newtheorem{defn}[theorem]{Definition}
\newtheorem*{notation*}{Notation}

\endlocaldefs

\begin{document}

\begin{frontmatter}
\title{Asymptotic stability and cut-off phenomenon for the underdamped Langevin dynamics}
\runtitle{Stability and cut-off for underdamped Langevin dynamics}

\begin{aug}
\author[A]{\fnms{Seungwoo}~\snm{Lee}\ead[label=e1]{swlee017@snu.ac.kr}},
\author[B]{\fnms{Mouad}~\snm{Ramil}\ead[label=e2]{mouad.ramil@inria.fr}}
\and
\author[C]{\fnms{Insuk}~\snm{Seo}\ead[label=e3]{insuk.seo@snu.ac.kr}}
\address[A]{Department of Mathematical Sciences, Seoul National University\printead[presep={,\ }]{e1}}

\address[B]{IRMAR (Université de Rennes), Inria Rennes and Research Institute of Mathematics, Seoul National University \printead[presep={,\ }]{e2}}

\address[C]{Department of Mathematical Sciences and Research Institute of Mathematics, Seoul National University \printead[presep={,\ }]{e3}}

\end{aug}

\begin{abstract}
In this article, we provide a detailed analysis of the long-time behavior of the underdamped Langevin dynamics.
We first provide a necessary condition guaranteeing that the zero-noise dynamical system converges to its unique attractor. We also observed that this condition is sharp for a large class of linear models. We then prove the so-called cut-off phenomenon in the small-noise regime under this condition. This result provides the precise asymptotics of the mixing time of the process and of the distance between the distribution of the process and its stationary measure.  The main difficulty of this work
relies on the degeneracy of its infinitesimal generator which is not elliptic, thus requiring a new set of methods.
\end{abstract}

 

\end{frontmatter}

\section{Introduction}
\subsection{Underdamped Langevin dynamics}

We start by introducing the underdamped Langevin dynamics which are
the main concern of the current article. In statistical physics the
evolution of a molecular system at a fixed temperature $T$ consisting
of $N$ particles can be modeled by the underdamped Langevin dynamics
$((q_{t},\,p_{t}))_{t\ge 0}$  in $\mathbb{R}^{d}\times\mathbb{R}^{d}$, $d=3N$,
where $q_{t}\in\mathbb{R}^{d}$ and $p_{t}\in\mathbb{R}^{d}$ denote
the vectors of positions and momenta of the particles. The process
$((q_{t},\,p_{t}))_{t\ge 0}$ satisfies a system of stochastic differential equations
of the form, see~\cite[Section 1.2.1]{LelRouSto10},
\begin{equation}
\left\{ \begin{aligned} & \mathrm{d}q_{t}=M^{-1}p_{t}\mathrm{d}t,\\
 & \mathrm{d}p_{t}=F(q_{t})\mathrm{d}t-\gamma M^{-1}p_{t}\mathrm{d}t+\sqrt{2\gamma\beta^{-1}}\mathrm{d}B_{t}\;,
\end{aligned}
\right.\label{eq:Langevin_intro}
\end{equation}
where $M\in\mathbb{R}^{d\times d}$ is the diagonal positive definite mass matrix comprising the respective mass of each particle, $F:\mathbb{R}^{d}\to\mathbb{R}^{d}$
is the force acting on the particles, $\gamma>0$ is the friction
parameter, and $\beta^{-1}=k_{B}T$ with $k_{B}$ the Boltzmann constant
and $T$ the temperature of the system. Such dynamics can be used
to compute thermodynamic quantities of the process. 

Thermodynamic quantities are averages of given observables against
the stationary distribution of~\eqref{eq:Langevin_intro}. They capture
the characteristics of the system at equilibrium and offer numerous
applications in biology, chemistry and materials science. However,
the computation of such quantities poses a certain number of computational
challenges. In fact, direct numerical integration is in general very
difficult when the space dimension is large, or equivalently when the number of particles $N$ is large, when computing integrals against the stationary distribution.  This is why stochastic
approaches are generally preferred and consist in computing the time-average
of the observable. Using the ergodicity of~\eqref{eq:Langevin_intro},
such time averages shall converge to the expected thermodynamic limit. For a recount on the statistical error made by the associated Monte-Carlo method we refer to~\cite[Section 2.3.1]{LelRouSto10} for instance. However, the associated speed of convergence relies heavily on how fast the process~\eqref{eq:Langevin_intro}
converges to its stationary distribution. Therefore, one area of research
is to understand the speed of convergence to equilibrium
in the low temperature regime, i.e. $\beta^{-1}$ goes to $0$. In
the current article, we completely solve this problem by verifying
the so-called cut-off phenomenon for the underdamped Langevin
dynamics under a suitable and reasonable condition on the vector field
$F$.

\subsection{Thermalization of the underdamped Langevin dynamics}

Although the proof is exactly similar for any positive diagonal matrix $M$, to simplify the notations we set $M=\mathbb{I}_{d}$ where $\mathbb{I}_d$ denotes the $d\times d$ identity matrix and $\epsilon=\gamma \beta^{-1}>0$ so that we can rewrite \eqref{eq:Langevin_intro} as
\begin{equation}
\begin{cases}
\mathrm{d}q_{t}^{\epsilon}=p_{t}^{\epsilon}\mathrm{d}t,\\
\mathrm{d}p_{t}^{\epsilon}=-F(q_{t}^{\epsilon})\mathrm{d}t-\gamma p_{t}^{\epsilon}\mathrm{d}t+\sqrt{2\epsilon}\mathrm{d}B_{t}\;.
\end{cases}\label{eq:uld}
\end{equation} 
Then, we are interested in the behavior of the $2d$-dimensional process
$(X_{t}^{\epsilon})_{t\ge0}=((q_{t}^{\epsilon},\,p_{t}^{\epsilon}))_{t\ge0}$ in the low-temperature
regime, i.e., the regime where $\epsilon$ tends to $0$. We first
provide a necessary condition (cf. Assumption  \ref{ass:main})
on the vector field $F$ under which the zero-temperature dynamics,
i.e., the dynamics \eqref{eq:uld} with $\epsilon=0$ converges, as
$t\rightarrow\infty,$ to a unique stable equilibrium of the dynamics.
As we can observe from the discussion below, this condition is very
close to the necessary and sufficient condition in the sense that
for a large-class of linear models our condition is indeed necessary
and sufficient.

Under this condition, we demonstrate in  Proposition \ref{prop:unique stat distrib} that
the process $(X_{t}^{\epsilon})_{t\geq0}$ is positive recurrent and therefore
admits a unique stationary probability measure $\mu_{\epsilon}$.

Then, by the standard ergodicity argument, the distribution of the 
process $(X_{t}^{\epsilon})_{t\geq0}$, starting from any fixed point in $\mathbb{R}^{d}\times\mathbb{R}^{d}$,
converges in total-variation distance to $\mu_{\epsilon}$, as $t\rightarrow\infty$, for any fixed $\epsilon>0$. This phenomenon is commonly referred to as \textit{thermalization}. The precise quantitative analysis of this convergence is the main
problem confronted in the current article. In fact, we prove the so-called 
\textit{cut-off phenomenon} which will be explained in 
the next subsection, which provides a quantitatively precise estimate of this convergence in the low-temperature regime where $\epsilon$ is close to zero.

\subsection{Cut-off phenomenon}

The convergence to equilibrium for stochastic diffusion processes
sometimes exhibits a brutal decrease around a cut-off time where the
total-variation distance between the law of the process and its equilibrium
goes from a value close to $1$ to a value close to $0$. This fact is known as \textit{cut-off phenomenon} in some card shuffling models.

\begin{figure}
\includegraphics[scale=0.28]{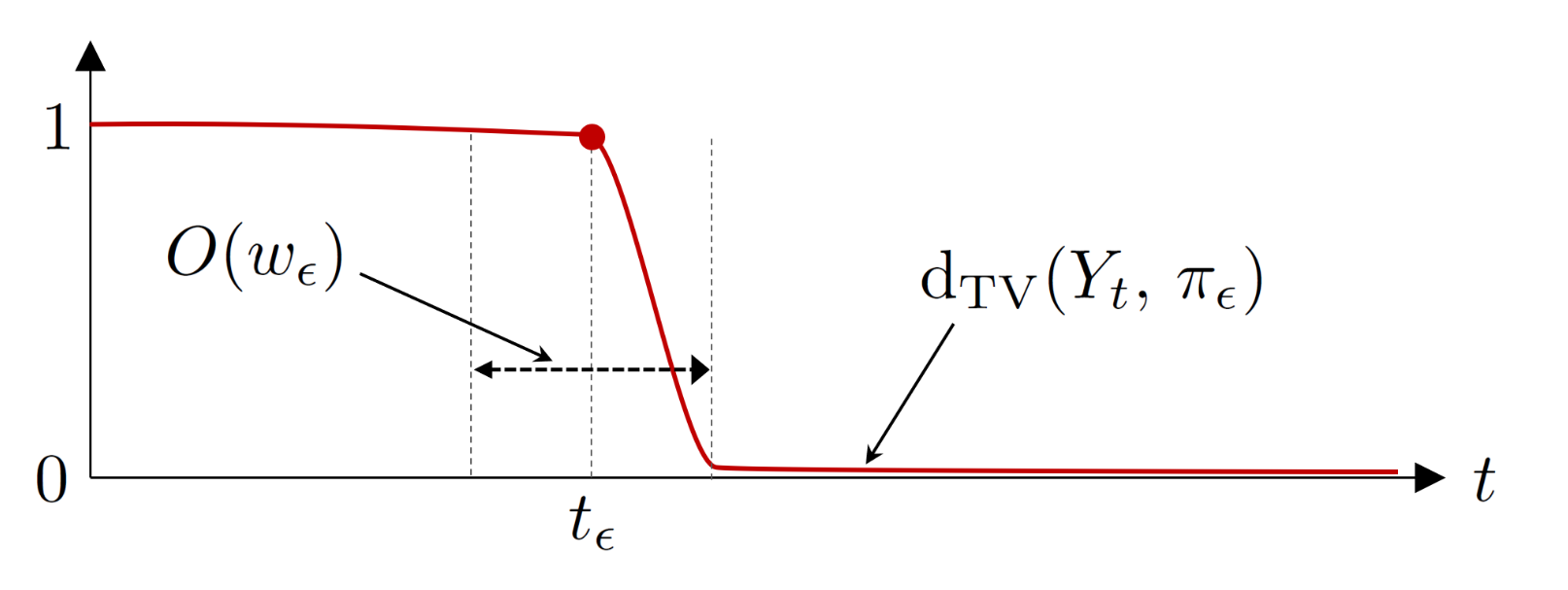} \caption{\label{fig1.1}Illustration of the cut-off phenomenon: the red line is the graph of the function $d(t)=\textrm{d}_\textup{TV} (Y_t^\epsilon, \, \pi_\epsilon)$ which abruptly drops from values close to $1$ to values close to $0$ at $t=t_\epsilon$ in a window of size $w_\epsilon$.}

\end{figure}
\begin{defn}[Total variation distance]\label{tvd}
For two probability measures $\pi_{1}$ and $\pi_{2}$ on $\mathbb{R}^{2d}$
(with the measure space consisting of Lebesgue measurable sets), we
define the total variation distance between $\pi_{1}$ and $\pi_{2}$
as
\[
\textup{d}_{\textup{TV}}(\pi_{1},\,\pi_{2})=\sup_{A}\left|\pi_{1}(A)-\pi_{2}(A)\right|\;,
\]
where the supremum is taken over all Lebesgue measurable sets in $\mathbb{R}^{2d}$. When
the law of a random vector $Z_{i}\in\mathbb{R}^{2d}$ is $\pi_{i}$
($i=1,\,2)$, we can substitute $\pi_{i}$ above with $Z_{i}$. For
instance, $\textup{d}_{\textup{TV}}(Z_{1},\,\pi_{2})$ and $\textup{d}_{\textup{TV}}(Z_{1},\,Z_{2})$
also represent $\textup{d}_{\textup{TV}}(\pi_{1},\,\pi_{2})$.
\end{defn}

\begin{defn}[Cut-off phenomenon]
\label{def:cut-off}For $\epsilon>0$, let $(Y_{t}^{\epsilon})_{t\geq0}$ be a Markov
process with unique stationary probability measure $\pi_{\epsilon}$. 
Then, we say that the process $(Y_{t}^{\epsilon})_{t\geq0}$ exhibits
the cut-off at mixing time $t^{\epsilon}$ on a window $w^{\epsilon}$
satisfying
\[
\lim_{\epsilon\rightarrow0}t^{\epsilon}=\infty\quad\text{and}\quad\lim_{\epsilon\rightarrow0}\frac{w^{\epsilon}}{t^{\epsilon}}=0\;,
\]
if we have
\begin{equation}
\begin{cases}
\lim_{c\rightarrow+\infty}\limsup_{\epsilon\rightarrow0}\mathrm{d}_{\mathrm{TV}}\left(Y_{t^{\epsilon}+cw^{\epsilon}}^{\epsilon},\,\pi_{\epsilon}\right)=0\;,\\
\lim_{c\rightarrow-\infty}\liminf_{\epsilon\rightarrow0}\mathrm{d}_{\mathrm{TV}}\left(Y_{t^{\epsilon}+cw^{\epsilon}}^{\epsilon},\,\pi_{\epsilon}\right)=1\;.
\end{cases}\label{eq:cut-off0}
\end{equation}
We refer to Figure \ref{fig1.1} for an illustration of the cut-off
phenomenon. The cut-off phenomenon implies that, the distribution of the process $(Y_t 
 ^\epsilon)_{t\ge 0 }$ converges to its stationary measure abruptly at the mixing time $t_\epsilon$ within a window of size $w_\epsilon \ll t_\epsilon$.

\begin{rem} 
In some contexts, this phenomenon is referred to as the \textit{window} cut-off to distinguish it from the \textit{profile} cut-off introduced below. However, in this article, we will simply refer to it as the cut-off phenomenon, as there is no risk of confusion within our scope.
\end{rem}
\begin{rem} 
In this definition, referring to $t^\epsilon$ as the mixing time may be considered a slight abuse of notation, since the mixing time $\widehat{t}^\epsilon$ of the process $Y_t^\epsilon$ is more commonly defined as  
\[
\widehat{t}^\epsilon = \inf \{ t : d_{\textrm{TV}}(Y_t^\epsilon, \pi_\epsilon) = 1/4 \}.
\]
However, this abuse of notation causes no issue, as it is immediate to verify that the cutoff phenomenon implies that $\lim_{\epsilon \to 0} \widehat{t}^\epsilon/t^\epsilon = 1$.
\end{rem}

Moreover, we say that there is a \textit{profile cut-off
}\textit{\emph{or}}\textit{ profile thermalization} with the profile
$G:\mathbb{R}\rightarrow[0,\,1]$ which is a decreasing function with
$G(-\infty)=1$ and $G(+\infty)=0$ if we have
\begin{equation}
\lim_{\epsilon\rightarrow0}\mathrm{d}_{\mathrm{TV}}\left(Y_{t^{\epsilon}+cw^{\epsilon}}^{\epsilon},\,\pi_{\epsilon}\right)=G(c)\label{eq:prof-cut-off}
\end{equation}
for all $c\in\mathbb{R}$. Of course, the profile cut-off is a stronger
notion than cut-off since \eqref{eq:prof-cut-off} immediately implies \eqref{eq:cut-off0}.
\end{defn}

\begin{rem}
Although we explain the cut-off phenomenon in the regime $\epsilon\rightarrow0$
to match with the main result of the current article, the other models
can exhibit the cut-off phenomenon in a different regime, e.g., $N\rightarrow\infty$
where $N$ is the number of sites in a spin system.
\end{rem}

In the modern study of the mixing behavior of the statistical phyics
systems, the cut-off phenomenon has been verified for various stochastic
systems. For instance, the cut-off phenomenon has been widely verified
for high-temperature spin systems such as the Glauber dynamics of
the Curie-Weiss type mean-field model, \cite{levin2010glauber,cuff2,yang2023cut-off},
the Glauber dynamics of the Ising/Potts model on lattices \cite{LS_Ising1,LS_Ising2,LS_Ising3,LS_Ising4},
the Swandsen-Wang dynamics of the Ising/Potts model \cite{NamSly},
and the FK-dynamics of the random-cluster model \cite{ganguly2020information}.
The cut-off phenomenon is also widely observed in the interacting
particle systems such as the zero-range processes \cite{merle2019cut-off},
the exclusion processes \cite{lacoin2016cut-off,elboim2022mixing,salez2023universality},
and the east process \cite{ganguly2015cut-off}. It is also demonstrated
that the cut-off phenomenon is universally observed for the random
walks on nicely expanding graphs, see \cite{lubetzky2010cut-off,lubetzky2011explicit,lubetzky2016cut-off}.
We also refer to the monograph \cite{LPW} for the introduction to
this topic.

\subsection{Cut-off phenomenon for the overdamped Langevin dynamics}
In \cite{cut-off_overdamped}, the cut-off phenomenon for the \textit{overdamped}
Langevin dynamics $(\overline{q}_{t}^{\epsilon})_{t\geq0}$ in $\mathbb{R}^{d}$
defined by
\begin{equation}
\mathrm{d}\overline{q}_{t}^{\epsilon}=-F(\overline{q}_{t}^{\epsilon})\mathrm{d}t+\sqrt{2\epsilon}\mathrm{d}B_{t}\;,\label{e:old}
\end{equation}
where $F\in\mathcal{C}^{3}(\mathbb{R}^{d},\,\mathbb{R}^{d})$ and
where $(B_{t})_{t\ge0}$ denotes the $d$-dimensional Brownian motion
has been established. In this work, it has been assumed that $0\in\mathbb{R}^d$ is the unique equilibrium point, thus $F(0)=0$, and that there
exist constants $\delta,C,\eta>0$ such that for all $q,q'\in\mathbb{R}^{d}$,
\begin{equation}\label{condov}
  \langle q,\mathrm{D}F(q')q\rangle \geq\delta|q|^{2}\;\;\;\text{and} \;\;\; |F(q)|\leq C\mathrm{e}^{\eta|q|^{2}}\;.
\end{equation}
where $DF : \mathbb{R}^d \rightarrow \mathbb{R}^{d\times d}$ denotes the Jacobian matrix of $F$ on each point (cf. Notation \ref{not21}).

The first condition implies that for all $q,\,q'\in\mathbb{R}^d$,
\begin{equation}\label{eq:variation F overdamped}
    \langle F(q')-F(q),\,q'-q\rangle\geq\delta|q-q'|^2\;.
\end{equation}
In particular, if we denote by $(\overline{q}_{t})_{t\geq0}$ the zero-noise dynamics associated to~\eqref{e:old} by setting $\epsilon=0$, one deduces from the inequality above that for all $t\geq0$, 
\begin{equation}\label{eq:glob exp stab overd}
    |\overline{q}_{t}|^2\leq|\overline{q}_{0}|^2\mathrm{e}^{-2\delta t}\;,
\end{equation}
which we refer to here as the global exponential stability of the zero-noise dynamics. Furthermore, using Ito's formula along with~\eqref{eq:variation F overdamped}, see~\cite[Section 3.1]{cut-off_overdamped} for the exact computations, one is also able to derive a control on the distance between $\overline{q}_{t}^{\epsilon}$ and $\overline{q}_{t}$. Namely, for all $t\geq0$,
$$\mathbb{E}\left[\left|\overline{q}_{t}^{\epsilon}-\overline{q}_{t}\right|^2\right]\leq\frac{d\epsilon}{2\delta}\,,$$
which can be seen as a zero-order approximation in $\epsilon$ of the perturbed system. A first order approximation of the perturbed system is also derived in~\cite[Section 3.2]{cut-off_overdamped} relying in particular on~\eqref{eq:variation F overdamped} and the construction of an appropriate Gaussian process $(\overline{y_t})_{t\geq0}$, which is independent of $\epsilon$, and such that for all $\theta\in(0,1)$,
\begin{equation}\label{eq:first order approx Langevin}
    \sup_{t\in[0,1/\epsilon^\theta]}\mathbb{E}\left[\left|\overline{q}_{t}^{\epsilon}-\overline{q}_{t}-\sqrt{2\epsilon}\,\overline{y_t}\right|^2\right]\underset{\epsilon\rightarrow0}{\longrightarrow}0\,.
\end{equation}
The above first-order approximation is a crucial element of the proofs as it allows us to focus on the speed of convergence to equilibrium for a Gaussian variable for which the total variation distance can be explicited in order to study the convergence to equilibrium for $(\overline{q}_{t}^{\epsilon})_{t\geq0}$. 

All in all, the estimates above yield the proofs of the cut-off phenomenon in~\cite{cut-off_overdamped} for the process~\eqref{e:old}. Namely, under the conditions~\eqref{condov}, the cut-off for the process $(\overline{q}_{t}^{\epsilon})_{t\geq0}$
starting from any fixed point is proven in \cite{cut-off_overdamped}
where the mixing time is of order $\log\frac{1}{\epsilon}$ and the
window is of order $1$. The mixing time is computed explicitly in
\cite[Lemma 2.1]{cut-off_overdamped}. The authors also provide in~\cite[Corollary 2.9]{cut-off_overdamped}
a necessary and sufficient condition under which the profile thermalization
holds.

Our main purpose in this work is to extend the existence of such cut-off
phenomenon to the underdamped Langevin dynamics~\eqref{eq:uld}. Fundamental difficulties appear when considering the underdamped Langevin dynamics which are due to the non-ellipticity of its infinitesimal generator since the diffusion matrix in~\eqref{eq:uld} is not invertible. We briefly mention below the challenges we faced and the new methods we developed to overcome them.

The first difficulty concerns the asymptotic stability of the zero-noise 
dynamics $(\epsilon=0)$ which is a crucial element of the proofs in~\cite{cut-off_overdamped} for the overdamped setting~\eqref{e:old}. In this case, the asymptotic stability is a trivial consequence of the condition~\eqref{condov} which ensures the global exponential stability stated in~\eqref{eq:glob exp stab overd}. However, to the best of our knowledge, there was no criterion available in the literature to ensure the asymptotic stability of the zero-noise underdamped Langevin dynamics when $F$ is not in gradient form (i.e., $F=\nabla U$). 

In Section~\ref{sec4}, we are able to provide a sharp criterion on the vector field $F$ which ensures the global exponential stability property for the zero-noise dynamics. The proof is based on considering a suitable Lyapunov function, which is defined as a polynomial perturbation of the Hamiltonian function. Some works have also been developed in the literature recently to provide suitable contraction by modifying the metric, see for instance in~\cite{schuh2024global}. However, we illustrate here the sharpness of our criterion by showing that it is a necessary and sufficient condition when $F(q)=\mathbb{M}q$
with $\mathbb{M}$ being a normal matrix (i.e. $\mathbb{M}\mathbb{M}^{\dagger}=\mathbb{M}^{\dagger}\mathbb{M}$, where $\mathbb{M}^{\dagger}$ denotes the transpose matrix of $\mathbb{M}$; cf. Notation \ref{not21}). It is important to notice that, unlike the overdamped setting, the condition~\eqref{condov} does not guarantee the asymptotic stability of the zero-noise underdamped Langevin dynamics, even in the linear case. This illustrates the difference of nature between both processes. Additionally, the criterion on $F$ also allows us to obtain in Section~\ref{sec8} relevant upper-bounds on the tails of the stationary probability distribution $\mu_\epsilon$, when $\epsilon$ tends to zero, which is also essential in the proofs.

Another requirement of the proofs is to obtain a control on the difference between the underdamped Langevin dynamics and its first-order approximation which is given by $Z_t:=X_t+\sqrt{2\epsilon}\,Y_t$ where $(X_t)_{t\geq0}$ is the zero-noise underdamped Langevin dynamics and $(Y_t)_{t\geq0}$ is a suitable centered Gaussian process. The objective is to deduce similar control established in~\eqref{eq:first order approx Langevin} for the overdamped setting. However, the proof developed in~\cite[Section 3.2]{cut-off_overdamped} relies on the condition~\eqref{condov} which the drift coefficient in~\eqref{eq:uld} cannot satisfy. This constitutes the second important difficulty of this work and we propose in Section~\ref{sec5} an alternative scheme of proof. It relies on the global exponential stability property, obtained before under suitable assumptions on $F$, which guarantees that the zero-noise dynamics arrives in a neighborhood of the origin in a finite time. Subsequently, we construct an appropriate metric in a neighborhood of the origin under which the drift coefficient of the first-order approximation $(Z_t)_{t\geq0}$ satisfies~\eqref{condov}. The computations are more tedious in this setting and they require in particular the use of a quadratic Gronwall inequality (Lemma~\ref{lem:quadratic gronwall}).

The third difficulty arises in Section~\ref{sec6} when studying the long-time behavior of the Gaussian process $(Y_t)_{t\geq0}$ given above, which is necessary for the study of the speed of convergence to equilibrium of $(X^\epsilon_t)_{t\geq0}$. This requires to study the long-time convergence of the covariance matrix $\Sigma_t$ of $Y_t$. In the overdamped setting, it is shown that $\Sigma_t\underset{t\rightarrow\infty}{\longrightarrow}\Sigma$ and the limiting matrix $\Sigma$ is defined as the unique solution to the following matrix Lyapunov equation
\begin{equation}
\textup{D}F(0)\Sigma+\Sigma\,\textup{D}F(0)^{\dagger}=\mathbb{J}\;,\label{eq:ov_lya}
\end{equation}
where $\mathbb{J}=\mathbb{I}_{d}$. Given the positive definiteness of $\mathbb{J}$ in this case, it is well-known that there exists a unique solution $\Sigma$ to this matrix equation. Moreover, $\Sigma$ is symmetric and positive definite. However, in the underdamped setting, the matrix $\mathbb{J}$ obtained is not invertible. This first issue requires us to extend the existence and uniqueness of a solution to~\eqref{eq:ov_lya} to the non-invertible underdamped setting by using the inherent dependency between position and velocity coordinates.

Define now $\Delta_t=\Sigma_t-\Sigma$. It remains then to show that $\Delta_t\underset{t\rightarrow\infty}{\longrightarrow}\mathbb{O}_{2d}$, which is done in the overdamped setting by studying the time evolution of its Frobenius norm defined as follows
\begin{equation}\label{eq:frob norm}
    t\geq0\to\sqrt{\textup{tr}(\Delta_t^2)}\;.
\end{equation}
Using the positive definiteness of the diffusion matrix in the overdamped setting, the authors show in~\cite[Appendix C]{cut-off_overdamped} that the function~\eqref{eq:frob norm} converges to zero exponentially fast in time. Given the non-invertibility of the diffusion matrix in the underdamped setting, such proof is not expected to work. However, by a smart modification of the metric through the construction of a symmetric positive definite matrix $\Gamma$ solution to a given matrix Lyapunov equation, we are able to show that the modified Frobenius norm $t\geq0\to\textup{tr}(\Delta_t\Gamma\Delta_t\Gamma)$ converges to zero exponentially fast in time as well. One can then easily conclude the proof using the positive definiteness of $\Gamma$. 

All in all, the challenges faced in this work were fundamental as well as technical. They required a deeper understanding
of the underdamped Langevin dynamics as well as the development of new methods to circumvent these obstacles. We believe they open the door to further understanding of the underdamped Langevin dynamics as well as the study of the cutoff phenomenon for more general non-elliptic diffusion processes.

\section{Model and Main Results}

\subsection{Stability of the underdamped Langevin dynamics}

\subsubsection*{Zero-noise dynamics}
In this article, we investigate the thermalization of the underdamped Langevin dynamics
$(X_{t}^{\epsilon})_{t\ge0}=(q_{t}^{\epsilon},\,p_{t}^{\epsilon}){}_{t\geq0}$
in $\mathbb{R}^{d}\times\mathbb{R}^{d}$ defined in \eqref{eq:uld}. To that end, we first have to understand the behaviour of the associated
zero-noise dynamics by setting $\epsilon=0$ in~\eqref{eq:uld}.
More precisely, the zero-noise dynamics\footnote{There is a collision of this notation with \eqref{eq:Langevin_intro},
but we from now on assume that $(q_{t},\,p_{t})_{t\geq0}$ always represents the
solution to \eqref{eq:uld0}.} $(X_{t})_{t\ge0}=(q_{t},\,p_{t})_{t\geq0}$ then satisfies the following
ordinary differential equation
\begin{equation}
\begin{cases}
\mathrm{d}q_{t}=p_{t}\mathrm{d}t,\\
\mathrm{d}p_{t}=-F(q_{t})\mathrm{d}t-\gamma p_{t}\mathrm{d}t\;.
\end{cases}\label{eq:uld0}
\end{equation}
We assume here that $F(0)=0$ so that $(0,\,0)\in\mathbb{R}^{d}\times\mathbb{R}^{d}$
is an equilibrium point of the deterministic dynamics $(X_{t})_{t\geq0}$.
In particular, we are interested in the situation when the dynamics
$(X_{t})_{t\geq0}$ admits the unique stable equilibrium $(0,\,0)$.
Additionally, we would like to ensure the \emph{global asymptotic
stability} of $(X_{t})_{t\geq0}$ meaning that starting from any initial
point $X_{0}=x\in\mathbb{R}^{2d}$,
\begin{equation}
\lim_{t\rightarrow\infty} X_{t}=(0,0) \in \mathbb{R}^{2d} \;.\label{eq:asymptotic stability}
\end{equation}
In the case when the force field $F$ is linear (i.e., there exists a matrix $\mathbb{M}$ such that $F(q)=\mathbb{M}q$ for all $q\in\mathbb{R}^d$), then several conditions were developed in the literature to ensure global asymptotic stability. We refer for instance to the well-known Kalman rank condition~\cite[Chapter 2, p. 35]{kalman1969topics}. Nonetheless, our objective here is to provide sharp conditions for general force fields $F$, which need not be linear or in gradient form, such that global asymptotic stability holds.

The condition on $F$ established in this work ensures the \emph{global exponential stability} of the dynamical system $(X_{t})_{t\geq0}$
meaning that there exists $\lambda>0$ such that starting from
any point $X_{0}=x\in\mathbb{R}^{2d}$, there exists a constant $C(x)$ such that
\begin{equation}
|X_{t}|^{2}\le C(x)e^{-\lambda t}\;\;\;\text{for all\;}t\ge0\;.\label{eq:stab}
\end{equation}
Additionally, we provide an explicit expression of $C(x)$ ensuring that it is bounded in any compact subset of $\mathbb{R}^{2d}.$

Note that the validity of such contraction is not immediate for dynamics of the form (\ref{eq:uld0}) since the time-derivative of $|X_{t}|^{2}$ is given by
\begin{equation}
\frac{\textup{d}}{\textup{d}t}|X_{t}|^{2}=2\left[q_{t}p_{t}+p_{t}(-F(q_{t})-\gamma p_{t})\right]\label{eq:L^2X_t}
\end{equation}
which has no reason to be negative for all $(q_{t},p_{t})\in\mathbb{R}^{2d}$.
In fact, as one can observe from an example of the flow diagram of
$(X_{t})_{t\geq0}$ given in Figure \ref{fig:flow}, the $L^{2}$-norm $|X_{t}|^{2}$
may increase at certain times $t_{0}$ since the trajectory of $(X_{t})_{t\geq0}$ is an elliptical rotation around the stable equilibria. This feature makes the asymptotic stability
analysis of $(X_{t})_{t\geq0}$ a non-trivial question. In particular, previous
studies such as \cite{cut-off_overdamped} on the overdamped model~\eqref{e:old}
assumed $F$ to satisfy certain conditions (see~\eqref{condov})
ensuring that the zero-noise dynamics satisfies
\begin{equation}
\frac{\textup{d}}{\textup{d}t}|X_{t}|^{2}\le-\lambda|X_{t}|^{2}\;\;\;\text{for some }\lambda>0\;,\label{eq::L^2decay}
\end{equation}
and thus global exponential stability. However, we notice from (\ref{eq:L^2X_t})
that there is no condition on $F$ which could possibly ensure (\ref{eq::L^2decay})
for the underdamped Langevin dynamics as well. 

We finally remark that, in the case where the dynamical system $(X_{t})_{t\geq0}$
admits multiple stable equilibria, the underdamped Langevin dynamics
$(X_{t}^{\epsilon})_{t\geq0}$ exhibits metastability among these
stable equilibria. This question was successfully studied in a consecutive work~\cite{LRS1} by the authors of the current article.

\begin{figure}\label{fig:flow}
\includegraphics[scale=0.27]{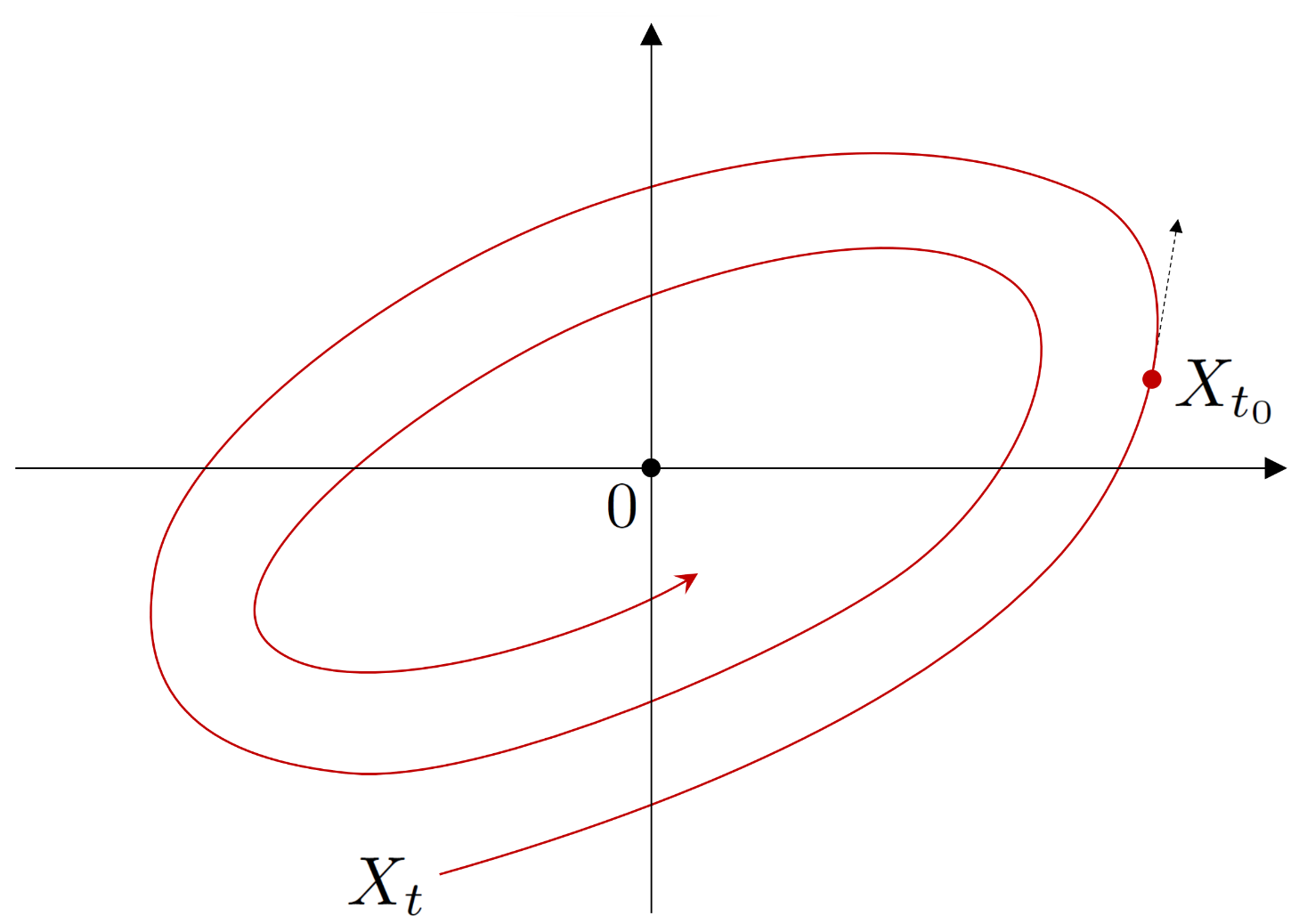} \caption{Trajectory of $(X_{t})_{t\ge0}$:
the $L^{2}$-norm $|X_{t}|^{2}$ may increase at a certain time $t_{0}$ since the trajectory converges to the origin by elliptically rotating.}
\end{figure}

\subsubsection*{Brief introduction to the main results}

Our first main result provides a sufficient condition (Assumption
\ref{ass:main}) under which the zero-noise dynamical system $(X_{t})_{t\geq0}$
is globally exponentially stable. This result is stated in Theorem
\ref{thm:stability} of Section \ref{sec2.2}. We believe that this
is the first non-trivial criterion ensuring global asymptotic stability
of the underdamped system $(X_{t})_{t\geq0}$ when the vector field
$F$ is not linear or cannot be written as the gradient of a given function. Furthermore,
in the linear case ($F(q)=\mathbb{M}q$ for some constant matrix
$\mathbb{M}$), we shall prove in Section~\ref{subsec:stab of linear models} that the condition given in
Assumption \ref{ass:main} is necessary and sufficient for a large class of linear models.

In Section \ref{sec2.3}, we study the low-temperature behavior of
the underdamped Langevin dynamics $(X_{t}^{\epsilon})_{t\geq0}$,
i.e. when $\epsilon$ is close to zero, under the exponential stability
property obtained in Theorem \ref{thm:stability}. In particular,
this allows us to obtain an asymptotic analysis of the invariant distribution
of $(X_{t}^{\epsilon})_{t\geq0}$ (Theorem \ref{thm:asymstat}) and
the sharp characterization of the speed of convergence to equilibrium for the underdamped Langevin dynamics, thus highlighting the cut-off phenomenon in Theorem \ref{thm:cut-off}.

\begin{notation}
\label{not21} We shall use the following notations throughout this
article.
\begin{enumerate}
\item We denote by $\mathbb{M}^{\dagger}$ (resp. $\mathbb{M}^*$) the transpose (resp. conjugate transpose) matrix of a square complex-valued matrix $\mathbb{M}$.
\item For a function $H:\mathbb{R}^n \rightarrow \mathbb{R}^m$, we denote by $DH(h) \in \mathbb{R}^{m\times n}$ its Jacobian matrix at a point $h\in \mathbb{R}^n$.
\item Any appearance of $x$ denotes an element of $\mathbb{R}^{2d}$. We
shall denote implicitly by $q\in\mathbb{R}^{d}$ its position vector
and $p\in\mathbb{\mathbb{R}}^{d}$ its velocity vector such that $x=(q,\,p)$.
\item We denote by $(X_{t}^{\epsilon}(x))_{t\ge0}=((q_{t}^{\epsilon}(x),\,p_{t}^{\epsilon}(x)))_{t\ge0}$
(resp. $(X_{t}(x))_{t\ge0}=((q_{t}(x),\,p_{t}(x)))_{t\ge0}$) the process $(X_{t}^{\epsilon})_{t\geq0}$
satisfying~\eqref{eq:uld} (resp. $(X_{t})_{t\geq0}$ satisfying~\ref{eq:uld0})
at time $t$ and starting from $x\in\mathbb{R}^{2d}$, i.e., $X_{0}^{\epsilon}(x)=x$
(resp. $X_{0}(x)=x$) almost-surely. We will not stress the starting point and merely
write $X_{t}^{\epsilon}$ and $X_{t}$ when we do not need to emphasize
the starting point.
\end{enumerate}
\end{notation}

\subsection{\label{sec2.2}Asymptotic stability analysis of zero-noise dynamics}

To state our first main result, we introduce the following assumption
on $F$ which guarantees the global exponential stability of $(X_{t})_{t\geq0}$
(\ref{eq:stab}). \begin{assumption} \label{ass:main}The vector
field $F\in\mathcal{C}^{3}(\mathbb{R}^{d},\mathbb{R}^{d})$ appearing
in (\ref{eq:uld}) and (\ref{eq:uld0}) can be decomposed into
\begin{equation}
F=\nabla U+\ell\;,\label{e:decF}
\end{equation}
for some non-negative function $U\in\mathcal{C}^{4}(\mathbb{R}^{d},\mathbb{R})$
and vector field $\ell\in\mathcal{C}^{3}(\mathbb{R}^{d},\mathbb{R}^{d})$ satisfying, for some constants $\alpha\in(0,\,\infty)$,
$\beta\in(0,\,\gamma)$,
\begin{equation}
\left\langle F(q),\,q\right\rangle \geq\alpha\left(|q|^{2}+U(q)\right)+\frac{|\ell(q)|^{2}}{\beta^{2}}\;\;\;\text{for all }q\in\mathbb{R}^{d}\;.\label{eq:conF}
\end{equation}
\end{assumption}
Under the assumption above, the following global exponential stability
of the deterministic dynamics $(X_{t})_{t\geq0}$ described by (\ref{eq:uld0})
can be stated.
\begin{thm}
\label{thm:stability} Under Assumption \ref{ass:main}, there exist
constants $\kappa,\lambda>0$ independent of $x\in\mathbb{R}^{2d}$
such that (cf. Notation \ref{not21})
\begin{equation}
|X_{t}(x)|^{2}\leq\kappa\left[|x|^{2}+U(q)\right]\mathrm{e}^{-\lambda t}\;\;\;\text{for all }t\ge0\;.\label{eq:conv_uld0}
\end{equation}
\end{thm}

The proof of this theorem is based on the construction of a suitable
Lyapunov function and shall be carried out in Section \ref{sec4}.
\begin{rem}
\label{rem:stability}The followings are explanations on Assumption
\ref{ass:main} and Theorem \ref{thm:stability}.
\begin{enumerate}
\item \emph{Necessity of Assumption \ref{ass:main}.} In view of Theorem
\ref{thm:linear case} below, Assumption \ref{ass:main} (which might
seem to be technical at first glance) is in fact not only natural
but also necessary in a large class of models to get global asymptotic
stability in the sense of~\eqref{eq:asymptotic stability}.
\item \emph{Meaning of condition \ref{eq:conF}. }The condition (\ref{eq:conF})
implies that $\ell$ should not be too large, and hence in view of
(\ref{e:decF}), the force $F$ should be sufficiently close to a
gradient force attracting towards the stable equilibrium. We note
that, in the overdamped model, as explained in \cite[condition (C) and display (2.2)]{cut-off_overdamped},
it suffices to have $F(q)\cdot q\geq\alpha|q|^{2}$ for all $q\in\mathbb{R}^{d}$
for some positive constant $\alpha$ in order to ensure global asymptotic
stability. However, in the underdamped model, if $|\ell(q)|$ is too
large so that the condition (\ref{eq:conF}) is violated, then such
stability fails as will be shown in Section~\ref{sec3}.
\item \emph{Decomposition \ref{e:decF} in the overdamped model.} The decomposition
of the form (\ref{e:decF}) is standard, e.g. \cite{Freidlin-Wentzell,Lee-Seo1,Lee-Seo2,LPeu-Michel}
in the overdamped model. In this literature, it is also often assumed
that this decomposition satisfies the orthogonality condition $\nabla U\cdot\ell\equiv0$,
which we did not need to assume in this work.
\item \emph{Equilibria of the dynamical system $(X_{t})_{t\geq0}$ under Assumption \ref{ass:main}.
}Note first that, since adding a constant to $U$ does not change
$\nabla U$ in (\ref{e:decF}), the non-negativity assumption on $U$
amounts to assuming that $U$ is bounded from below. Additionally,
taking $q=0$ in (\ref{eq:conF}), we obtain $U(0)=0$ and $\ell(0)=0$.
The former implies that $0$ is a global minimum of $U$ (by the non-negativity
assumption), and hence $\nabla U(0)=0$. Therefore, $(0,0)$ is an
equilibrium point of $(X_{t})_{t\geq0}$. On the other hand, suppose that $x_{0}=(q_{0},\,p_{0})$
is an equilibrium point of the dynamics (\ref{eq:uld0}), we should
have $p_{0}=F(q_{0})=0$. Inserting $q=q_{0}$ into (\ref{eq:conF}),
we get $|q_{0}|^{2}=0$ (since $U$ is non-negative) and thus $q_{0}=0$.
Therefore, condition (\ref{eq:conF}) implies that\emph{ $(0,\,0)$
is the unique equilibrium point of the dynamics (\ref{eq:uld0}).
}
\item \emph{Quadratically growing $U$. }When $U$ grows at most quadratically,
i.e., there exists $C>0$ such that
\begin{equation}
|U(q)|\le C|q|^{2}\;\;\;\text{for all }q\in\mathbb{R}^{d}\;,\label{e:quad}
\end{equation}
we notice that the condition (\ref{eq:conF}) is equivalent to the existence of $\alpha>0$ and $\beta\in(0,\gamma)$ such that
\begin{equation}
\left\langle F(q),\,q\right\rangle \geq\alpha|q|^{2}+\frac{|\ell(q)|^{2}}{\beta^{2}}\;\;\;\text{for all }q\in\mathbb{R}^{d}\;,\label{e:conF-1}
\end{equation}
and moreover we can replace (\ref{eq:conv_uld0}) by
\begin{equation}
|X_{t}(x)|^{2}\leq\kappa|x|^{2}\mathrm{e}^{-\lambda t}\;\;\;\text{for all }t\ge0\;.\label{eq:conv_uld0-1}
\end{equation}
\item \emph{Regularity assumption on $F$. }Although we assumed that $F\in\mathcal{C}^{3}(\mathbb{R}^{d},\,\mathbb{R}^{d})$,
a careful reading of the proofs reveals that the conditions $F\in\mathcal{C}^{2}(\mathbb{R}^{d},\,\mathbb{R}^{d})$
and $F$ is $\mathcal{C}^{3}$ in a neighborhood of the origin are
indeed sufficient. The regularity of $U$ and $\ell$ imposed in Assumption
\ref{ass:main} can be weakened accordingly.
\end{enumerate}
\end{rem}

\subsubsection*{Linear models and necessity of Assumption \ref{ass:main}}

The theorem below argues that Assumption \ref{ass:main} is a necessary
and sufficient condition for a large class of linear forces $F$.
\begin{thm}
\label{thm:linear case}Suppose that $F(q)=\mathbb{M}q$ with $\mathbb{M}$
being a normal matrix, i.e. $\mathbb{M}\mathbb{M}^{\dagger}=\mathbb{M}^{\dagger}\mathbb{M}$.
Then, the dynamics (\ref{eq:uld0}) is asymptotically stable
if and only if $F$ satisfies Assumption \ref{ass:main}.
\end{thm}

The proof of this theorem will be given in Section \ref{sec3}. In
particular, this theorem is a consequence of Theorem \ref{thm:linchar}
which completely characterizes the necessary and sufficient condition
on $\mathbb{M}$ under which the dynamics (\ref{eq:uld0}) is globally
asymptotically stable. Notice that in the linear case global asymptotic
stability is equivalent to asymptotic stability starting from an arbitrary point.

\subsection{\label{sec2.3}Thermalization of the underdamped Langevin dynamics}

Let us assume that $F$ satisfies Assumption \ref{ass:main} so that
the zero-noise dynamics exhibits the global exponential stability
towards the origin stated in our first result Theorem \ref{thm:stability}.
We shall prove in Section \ref{sec4.4} that the underdamped
process $(X_{t}^{\epsilon})_{t\geq0}$ admits in fact a unique stationary
distribution $\mu^{\epsilon}$.

\begin{prop}\label{prop:unique stat distrib}
The process $(X_{t}^{\epsilon})_{t\geq0}$ admits a unique stationary
distribution $\mu^{\epsilon}$.
\end{prop}

Our second result investigates the
thermalization of the process $(X_{t}^{\epsilon})_{t\geq0}$
in the regime $\epsilon\rightarrow0$ towards its stationary measure $\mu_\epsilon$. Our concern is to precisely describe the underlying cut-off phenomenon involved when $\epsilon$ goes to zero.

The convergence to equilibrium is studied under the total variation distance (for the precise definition, we refer to Definition \ref{tvd}).

\subsubsection*{Stationary measure}

An important difficulty arising in the study of $\textup{d}_{\textup{TV}}(X_{t}^{\epsilon},\,\mu_{\epsilon})$
is that the stationary measure $\mu^{\epsilon}$ cannot be given explicitly
when the force $F$ does not write as a gradient function. However,
for the purpose of this work, we are able to compute an asymptotic
behavior of $\mu^{\epsilon}$ when $\epsilon$ goes to zero. In fact,
we show in Theorem \ref{thm:asymstat} that the probability measure
$\mu^{\epsilon}$ behaves in the low-temperature regime ($\epsilon\rightarrow0$)
as the law of a given centered Gaussian variable. Such approximation
is known for the overdamped model from \cite{Barrera_inv,cut-off_overdamped},
but as far as the author's knowledge, this constitutes the first result
for the underdamped model.

To explain this result, we define two $2d\times2d$ matrices $\mathbb{A}$
and $\mathbb{J}$ by
\begin{equation}
\mathbb{A}=\begin{pmatrix}\mathbb{O}_{d} & \mathbb{I}_{d}\\
-\textup{D}F(0) & -\gamma\mathbb{I}_{d}
\end{pmatrix}\;\;\;\;\text{and}\;\;\;\;\mathbb{J}=\begin{pmatrix}\mathbb{O}_{d} & \mathbb{O}_{d}\\
\mathbb{O}_{d} & \mathbb{I}_{d}
\end{pmatrix}\;,\label{eq:J}
\end{equation}
where $\textup{D}F$ represents the Jacobian of $F$ at $(0,0)$, and where $\mathbb{O}_{d}$
and $\mathbb{I}_{d}$ denote the zero matrix of size $d\times d$
and the $d$-dimensional identity matrix, respectively. Then, by Lemma
\ref{lem:riccati} with $(\mathbb{U},\,\mathbb{W})=(\mathbb{A}^{\dagger},\,\mathbb{J})$
(cf. the paragraph after the proof of Lemma \ref{lem:riccati}), there
exists a unique $2d\times2d$ matrix $\Sigma$ satisfying
\begin{equation}
\mathbb{A}\Sigma+\Sigma\mathbb{A}^{\dagger}=-\mathbb{J}\;.\label{eq:Sigma}
\end{equation}
Moreover, by the same lemma the matrix $\Sigma$ is symmetric positive
definite. We show then that we can approximate $\mu^{\epsilon}$ by
the law of the $2d$-dimensional centered Gaussian variable with covariance
matrix $2\epsilon\Sigma$.
\begin{thm}
\label{thm:asymstat}Suppose that Assumption \ref{ass:main} holds.
Then, the matrix $\Sigma$ defined in (\ref{eq:Sigma}) is such that
\[
\lim_{\epsilon\rightarrow0}\mathrm{d}_{\mathrm{TV}}\left(\mathcal{N}(0,2\epsilon\Sigma),\,\mu^{\epsilon}\right)=0\;.
\]
\end{thm}

The proof of this Theorem is given in Section \ref{sec8}. The appearance
of this theorem is similar to \cite[Proposition 3.6]{cut-off_overdamped}
for the overdamped Langevin dynamics. However, in our case, $\Sigma$
is solution to a matrix Lyapunov equation (\ref{eq:Sigma}) \emph{where
the right-hand side matrix $\mathbb{J}$ is not invertible} and therefore~\cite[Theorem 1 p. 443]{lancaster1985theory} which was crucially used in the
proof of \cite[Proposition 3.6]{cut-off_overdamped} is not applicable.
\begin{rem}[Gradient model]
It was shown in \cite[Proposition 2.1]{Monmarche-Ramil} that the
Gibbs measure
\[
\mu_{\epsilon}(\mathrm{d}x)=\mu_{\epsilon}(x)\mathrm{d}x:=\frac{1}{Z_{\epsilon}}\exp\left\{ -\frac{\gamma}{\epsilon}V(x)\right\} \mathrm{d}x \;,
\]
where $V(x)=\frac{1}{2}|p|^{2}+U(q)$ and $Z_{\epsilon}$ is the normalizing
constant, is a stationary probability distribution if and only if $F=\nabla U$ for
some $U\in\mathcal{C}^{3}(\mathbb{R}^{d},\,\mathbb{R})$. In this
case, by the global asymptotic stability (Theorem \ref{thm:stability})
and Remark \ref{rem:stability}-(4), the potential function $U$ achieves
its global minimum at $0$. Let us assume without loss of generality
that $U(0)=0$ (since adding a constant to $U$ does not affect the
dynamical system). Then, by the Laplace asymptotics, we can derive
\[
Z_{\epsilon}\simeq\frac{(2\pi\epsilon/\gamma)^{d}}{\sqrt{\det\mathbb{H}}}\;,
\]
where $\mathbb{H}=\nabla^{2}U(0)$ denotes the Hessian of $U$ at
$0$. Thus, by taking the second order Taylor approximation of $U$
at $0$ (note that $U(0)=0$ and $\nabla U(0)=0$), we obtain
\begin{align*}
\mu_{\epsilon}(x)=\frac{1}{Z_{\epsilon}}\textup{e}^{-\frac{\gamma}{\epsilon}V(x)} & \simeq\;\frac{\sqrt{\det\mathbb{H}}}{(2\pi\epsilon/\gamma)^{d}}\exp\left\{ -\frac{\gamma}{2\epsilon}\left(|p|^{2}+\left\langle q,\,\mathbb{H}q\right\rangle \right)\right\} \\
 & =\;\frac{1}{\sqrt{(2\pi\epsilon)^{2d}\det\Sigma}}\exp\left\{ -\frac{1}{2\epsilon}\left(\left\langle x,\,\widehat{\Sigma}^{-1}x\right\rangle \right)\right\}\;,
\end{align*}
where $\widehat{\Sigma}$ is the $2d\times2d$ matrix defined by
\begin{equation}
\widehat{\Sigma}=\frac{1}{\gamma}\begin{pmatrix}\mathbb{H}^{-1} & \mathbb{O}_{d}\\
\mathbb{O}_{d} & \mathbb{I}_{d}
\end{pmatrix}\;.\label{eq:covgrad}
\end{equation}
Therefore, we can conclude that $\mu_{\epsilon}$ is close to the
centered Gaussian distribution with covariance matrix $\epsilon\widehat{\Sigma}$
when $\epsilon$ is close to zero. By a direct computation (note that
in this case $\mathrm{D}F(0)=\mathbb{H})$, we can check that $\widehat{\Sigma}/2$
is indeed the solution to the equation (\ref{eq:Sigma}), and hence
$\widehat{\Sigma}=2\Sigma$.
\end{rem}

\subsubsection*{Cut-off phenomenon}

To quantify the total variation distance $\textup{d}_{\textup{TV}}(X_{t}^{\epsilon},\,\mu_{\epsilon})$
in a precise manner when $\epsilon\rightarrow0$, we first have to state the technical lemma below.
We remark that this lemma is the underdamped version of \cite[Lemma 2.1]{cut-off_overdamped}.
Although the proof is very similar to that of \cite[Lemma 2.1]{cut-off_overdamped} which is
based on \cite[Lemmata B.1 and B.2]{cut-off_overdamped} and the Hartman-Grobman
theorem, we will provide its proof in Section \ref{sec:lem_asym_uld}
in order to explain the nature of the different constants appearing
in the statement of the lemma in the underdamped setting.
\begin{lem}
\label{lem:asym_uld} Under Assumption \ref{ass:main}, for each $x\in\mathbb{R}^{2d}\setminus\{0\}$,
there exist constants
\[
\begin{cases}
\eta=\eta(x)>0,\;\nu=\nu(x)\ge0,\;\tau=\tau(x)\ge0\;,\\
m=m(x)\in\{1,\ldots,2d\}\;,\\
\theta_{1}=\theta_{1}(x),\ldots,\theta_{m}=\theta_{m}(x)\in[0,2\pi)\;,
\end{cases}
\]
and linearly independent complex vectors $v_{1}=v_{1}(x),\ldots,v_{m}=v_{m}(x)\in\mathbb{C}^{2d}$
such that
\begin{equation}
\lim_{t\rightarrow\infty}\left|\frac{\mathrm{e}^{\eta t}}{t^{\nu}}X_{t+\tau}(x)-\sum_{k=1}^{m}\mathrm{e}^{i\theta_{k}t}v_{k}\right|=0\;.\label{eq:approx1}
\end{equation}
\end{lem}

The constants appearing in the previous lemma are determined by the
matrix $\textup{D}F(0)$ and the starting location $x$ in a complex
manner. Roughly speaking, $\eta$ is the smallest positive real part
of the eigenvalues of the matrix $-\mathbb{A}$ defined in (\ref{eq:J})
and $\nu$ is the size of the Jordan block associated with the eigenvalue
with smallest positive real part minus $1$. These constants depend
on $x$ since some eigenvalues must not be considered when $x$ is
of special direction. For more detail, we refer to the proof in Section
\ref{sec:lem_asym_uld} and Remark \ref{rem:const}.

We are now ready to provide our result on the thermalization of the
underdamped Langevin process $(X_{t}^{\epsilon})_{t\geq0}$~\eqref{eq:uld}.
We shall from now on suppose the following control on the jacobian
$\textup{D}F$.

\begin{assumption} \label{ass:DF}There exist constants
$C,\,\rho>0$ such that
\[
|\textup{D}F(q)|\leq C\mathrm{e}^{\rho|q|^{2}}\;\;\;\text{for all }q\in\mathbb{R}^{d}\;.
\]
\end{assumption}

The second main result is the following theorem.
\begin{thm}
\label{thm:cut-off}Let Assumptions \ref{ass:main} and \ref{ass:DF}
hold. For any $x\in\mathbb{R}^{2d}\setminus\{0\}$, recall the constants
$\eta(x),\,\nu(x)$ and $\tau(x)$ from Lemma \ref{lem:asym_uld},
and define the mixing time as
\begin{equation}
t_{\mathrm{mix}}^{\epsilon}(x)=\frac{1}{2\eta(x)}\log\left(\frac{1}{2\epsilon}\right)+\frac{\nu(x)}{\eta(x)}\log\log\left(\frac{1}{2\epsilon}\right)+\tau(x)\;.\label{eq:mixingtime}
\end{equation}
Then, we have
\begin{align}
 & \lim_{w\rightarrow\infty}\limsup_{\epsilon\rightarrow0}\mathrm{d}_{\mathrm{TV}}\left(X_{t_{\mathrm{mix}}^{\epsilon}(x)+w}^{\epsilon}(x),\,\mu_{\epsilon}\right)=0\;\text{ and}\nonumber \\
 & \lim_{w\rightarrow\infty}\liminf_{\epsilon\rightarrow0}\mathrm{d}_{\mathrm{TV}}\left(X_{t_{\mathrm{mix}}^{\epsilon}(x)-w}^{\epsilon}(x),\,\mu_{\epsilon}\right)=1\;.\label{eq:cut-off}
\end{align}
\end{thm}

The previous theorem asserts that there is a cut-off phenomenon at
the mixing time $t_{\mathrm{mix}}^{\epsilon}(x)$ with window of size
$O(1)$. Under a condition detailed below, we have a stronger cut-off
notion called profile cut-off. The definition of the profile cut-off is mentioned on \eqref{eq:prof-cut-off}.  
\begin{thm}
\label{thm:profilecut-off}Let Assumptions \ref{ass:main} and \ref{ass:DF}
hold and let $x\in\mathbb{R}^{2d}\setminus\{0\}$. Recall the mixing
time $t_{\mathrm{mix}}^{\epsilon}(x)$ from (\ref{eq:mixingtime})
and the constants $\theta_{k}(x)$'s and vectors $v_{k}(x)$'s from Lemma
\ref{lem:asym_uld}. Then, the underdamped Langevin dynamics starting
at $x\in\mathbb{R}^{2d}$ exhibits a profile cut-off at the mixing
time $t_{\mathrm{mix}}^{\epsilon}(x)$ and window $O(1)$ if and only
if the limit
\begin{equation}
r:=\lim_{t\rightarrow\infty}\left|\Sigma^{-1/2}\sum_{k=1}^{m}\mathrm{e}^{i\theta_{k}t}v_{k}\right|\label{eq:limcon}
\end{equation}
exists. Moreover, in this case $r>0$, the profile cut-off function is given by {
\[
\Lambda(w)=2\int_{0}^{\sqrt{2}(1/2\eta)^{\nu}e^{-w\eta}}\frac{1}{\sqrt{2\pi}}e^{-t^{2}/2}dt\;.
\]
}
\end{thm}

We note that the limit (\ref{eq:limcon}) exists if all the eigenvalues
of matrix $-\mathbb{A}$ defined in (\ref{eq:J}) with smallest real part
are real. In particular, if all the eigenvalues of $\mathbb{A}$ are
real, then the limit (\ref{eq:limcon}) exists.

The proof of Theorems \ref{thm:cut-off} and \ref{thm:profilecut-off}
will be given in Section \ref{sec7} based on the analyses carried
out in Sections \ref{sec4}-\ref{sec6}. In particular, these theorems
are direct consequences of Theorem \ref{thm:tv}.

\section{\label{sec3}Linear Models}

In this section, we analyze the case when the force $F$ is linear,
i.e. $F(q)=\mathbb{M}q$. We shall see in this section that Assumption
\ref{ass:main} is indeed a sharp, i.e., necessary and sufficient
condition ensuring asymptotic stability for a large class of
linear models. \begin{notation*} The following notation will be used
in the current section.
\begin{itemize}
\item Writing a complex number as $a+ib$ implicitly assumes that $a,\,b\in\mathbb{R}$.
Similarly, writing a complex vector as $u+iv$ implicitly assumes
that $u$ and $v$ are real vectors.
\item For a square matrix $\mathbb{M}$, we denote by $\mathrm{Sp}(\mathbb{M})$
the set of (possibly complex) eigenvalues of $\mathbb{M}$.
\item For a square matrix $\mathbb{M}$, we denote by $\mathbb{\mathbb{M}^{\textup{s}}}$
and $\mathbb{\mathbb{M}^{\textup{a}}}$ the symmetric, and skew-symmetric
part of matrix $\mathbb{M}$, respectively, i.e.,
\begin{equation}
\mathbb{M}^{\textup{s}}=\frac{1}{2}(\mathbb{M}+\mathbb{M}^{\dagger})\;\;\;\text{and}\;\;\;\mathbb{M}^{\textup{a}}=\frac{1}{2}(\mathbb{M}-\mathbb{M}^{\dagger})\label{eq:sadef}
\end{equation}
so that $\mathbb{M}=\mathbb{M}^{\textup{s}}+\mathbb{M}^{\textup{a}}$.
\end{itemize}
\end{notation*}

\subsection{Elementary lemmata }

We start with some elementary lemmata from linear algebra.
\begin{lem}
\label{lem:ev}For each $d\times d$ matrix $\mathbb{M}$ and $\gamma>0$
(which is the constant appearing in (\ref{eq:uld})), define
\begin{equation}
\mathbb{T}_{\mathbb{M}}=\begin{pmatrix}\mathbb{O}_{d} & -\mathbb{I}_{d}\\
\mathbb{M} & \gamma\mathbb{I}_{d}
\end{pmatrix}\label{eq:T_M}\;.
\end{equation}
Then, all the eigenvalues of $\mathbb{T}_{\mathbb{M}}$ have positive
real part if and only if
\[
\mathrm{Sp}(\mathbb{M})\subset\{a+ib:a>0\text{ and }\gamma^{2}a>b^{2}\}\;.
\]
\end{lem}

\begin{rem}
    This ensures that the spectrum of $\mathbb{T}_{\mathbb{M}}$ shall be included in the following red parabolic region of the complex eigenspace.
    \begin{center}
    \begin{figure}[h]
        \centering
        \begin{tikzpicture}[scale=0.8]
          \draw[->] (-2.5, 0) -- (4.7, 0) node[right] {$\mathrm{Re}$};
  \draw[->] (0, -2) -- (0, 2.2) node[above] {$\mathrm{Im}$};
  \fill[scale=0.5, red, pattern=north east lines, pattern color = red, domain=-3:3, smooth, variable=\y]  plot ({\y*\y}, {\y}); 
  

        \end{tikzpicture}
    \end{figure}
\end{center}
\end{rem}

\begin{proof}
\noindent We note that, for $d\times d$ matrices $\mathbb{B}_{11},\,\mathbb{B}_{12},\,\mathbb{B}_{21}$ and
$\mathbb{B}_{22}$ such that $\mathbb{B}_{11}\mathbb{B}_{21}=\mathbb{B}_{21}\mathbb{B}_{11}$,
we have the following determinant formula for the block matrix:
\[
\det\begin{pmatrix}\mathbb{B}_{11} & \mathbb{B}_{12}\\
\mathbb{B}_{21} & \mathbb{B}_{22}
\end{pmatrix}=\det\left(\mathbb{B}_{11}\mathbb{B}_{22}-\mathbb{B}_{12}\mathbb{B}_{21}\right)\;.
\]
Using this formula, we can deduce that, for any $t\in\mathbb{R},$
\[
\det\left(\mathbb{T}_{\mathbb{M}}-t\mathbb{I}_{2d}\right)=\det\left(\mathbb{M}-t(\gamma-t)\mathbb{I}_{d}\right)\;.
\]
From this expression, we can deduce that $\lambda\in\mathbb{C}$ is
an eigenvalue of $\mathbb{\mathbb{T}_{M}}$ if and only if $h(\lambda)=-\lambda^{2}+\gamma\lambda$
is an eigenvalue of $\mathbb{M}$. One can check $h:\mathbb{C}\rightarrow\mathbb{C}$
is a bijection between $\{z\in\mathbb{C}:\gamma>\text{Re}(z)>0\}$
and $\{z\in\mathbb{C}:\textrm{Im}(z)^{2}<\gamma^{2}\text{\textrm{Re}}(z)\}$.
Since sum of two roots of $h(\lambda)=c$ is $\gamma$, all the eigenvalues
of $\mathbf{\mathbb{T}_{\mathbb{M}}}$ have positive real part if
and only if all of the real parts of eigenvalues of $\mathbf{\mathbb{T}_{\mathbb{M}}}$ are
positive and less then $\gamma$.
\end{proof}
\begin{lem}
\label{lem:evrel}For a $d\times d$ matrix $\mathbb{M}$, define
the symmetric matrix $\mathbb{K}_{\mathbb{M}}$ by
\begin{equation}
\mathbb{K}_{\mathbb{M}}=\gamma^{2}\mathbb{M}^{s}+(\mathbb{M}^{\textup{a}})^{2}+\frac{1}{2}(\mathbb{M}^{\textup{a}}\mathbb{M}^{s}-\mathbb{M}^{s}\mathbb{M}^{\textup{a}})\;.\label{eq:K-1}
\end{equation}
Denote by $w=u+iv$ an unit eigenvector (i.e., $|u|^{2}+|v|^{2}=1$)
of $\mathbb{M}$ associated with the eigenvalue $z=a+ib$. Then,
\begin{enumerate}
\item it holds that
\[
\left\langle \mathbb{K}_{\mathbb{M}}w,\,w\right\rangle _{\mathbb{C}}=\langle\mathbb{K}_{\mathbb{M}}u,u\rangle+\langle\mathbb{K}_{\mathbb{M}}v,v\rangle=\gamma^{2}a-b^{2}\;,
\]
{where $\langle\,,\,\rangle_{\mathbb{C}}$ is the Hermitian scalar
product in $\mathbb{C}^{d}$.}
\item It holds that
\[
\langle(\mathbb{M}^{\textup{a}}\mathbb{M}^{s}-\mathbb{M}^{s}\mathbb{M}^{\textup{a}})w,\,w\rangle_{\mathbb{C}}=\langle(\mathbb{M}^{\textup{a}}\mathbb{M}^{s}-\mathbb{M}^{s}\mathbb{M}^{\textup{a}})u,\,u\rangle+\langle(\mathbb{M}^{\textup{a}}\mathbb{M}^{s}-\mathbb{M}^{s}\mathbb{M}^{\textup{a}})v,\,v\rangle\geq0\;.
\]
\end{enumerate}
\end{lem}

\begin{proof}
By looking at the real and imaginary parts of the equation $\mathbb{M}w=zw$
with $\mathbb{M}=\mathbb{M}^{\textup{s}}+\mathbb{M}^{\textup{a}}$,
$w=u+iv$, and $z=a+ib$, respectively, we get
\begin{align}
 & (\mathbb{M}^{\textup{s}}+\mathbb{M}^{\textup{a}})u=au-bv\;,\label{eq:ide1}\\
 & (\mathbb{M}^{\textup{s}}+\mathbb{M}^{\textup{a}})v=av+bu\;.\label{eq:ide2}
\end{align}
Since $\mathbb{M}^{\textup{a}}$ is skew-symmetric, $[u\cdot$(\ref{eq:ide1})$+v\cdot$(\ref{eq:ide2})$]$
and $[v\cdot$(\ref{eq:ide1})$-u\cdot$(\ref{eq:ide2})$]$ respectively
yield that
\begin{align}
 & \langle\mathbb{M}^{\textup{s}}u,\,u\rangle+\langle\mathbb{M}^{\textup{s}}v,\,v\rangle=a\;\;\;\text{and\;\;\;}2\langle\mathbb{\mathbb{M}^{\textup{a}}}u,\,v\rangle=-b\;,\label{eq:ide3}
\end{align}
where we used the fact that $|u|^{2}+|v|^{2}=1$. Next, $[(\mathbb{\mathbb{M}^{\textup{a}}}$$\times$(\ref{eq:ide1})$)\cdot u+(\mathbb{M}^{\textup{a}}\times$(\ref{eq:ide2})$)\cdot v]$
yields that
\begin{align*}
\langle\mathbb{M}^{\textup{a}}(\mathbb{M}^{\textup{s}}+\mathbb{M}^{\textup{a}})u,\,u\rangle+\langle\mathbb{\mathbb{M}^{\textup{a}}}(\mathbb{M}^{\textup{s}}+\mathbb{\mathbb{M}^{\textup{a}}})v,\,v\rangle & =-b\langle\mathbb{M}^{\textup{a}}v,\,u\rangle+b\langle\mathbb{M}^{\textup{a}}u,\,v\rangle=-b^{2}\;,
\end{align*}
where the last equality follows from the second identity of (\ref{eq:ide3})
and the skew-symmetry of $\mathbb{M}^{\textup{a}}$. Combining this
with the first identity of (\ref{eq:ide3}), we get
\[
\gamma^{2}a-b^{2}=\gamma^{2}(\langle\mathbb{M}^{\textup{s}}u,\,u\rangle+\langle\mathbb{M}^{\textup{s}}v,\,v\rangle)+\langle\mathbb{M}^{\textup{a}}(\mathbb{M}^{\textup{s}}+\mathbb{\mathbb{M}^{\textup{a}}})u,\,u\rangle+\langle\mathbb{M}^{\textup{a}}(\mathbb{M}^{\textup{s}}+\mathbb{M}^{\textup{a}})v,\,v\rangle\;.
\]
Recalling the definition of $\mathbb{K}_{\mathbb{M}}$ completes the
proof of part (1).

Next we turn to part (2). By the symmetry of $\mathbb{M}^{\textup{s}}$
and skew-symmetry of $\mathbb{M}^{\textup{a}}$, it suffices to prove
that
\begin{equation}
\langle\mathbb{\mathbb{M}^{\textup{s}}\mathbb{M}^{\textup{a}}}u,\,u\rangle+\langle\mathbb{\mathbb{M}^{\textup{s}}}\mathbb{M}^{\textup{a}}v,\,v\rangle\leq0\;.\label{eq:assa1.5}
\end{equation}
Observe that $[(\mathbb{M}^{\textup{s}}$$\times$(\ref{eq:ide1})$)\cdot u+(\mathbb{M}^{\textup{s}}\times$(\ref{eq:ide2})$)\cdot v]$
yields that
\begin{align}
\langle\mathbb{M}^{\textup{s}}(\mathbb{M}^{\textup{s}}+\mathbb{M}^{\textup{a}})u,\,u\rangle+\langle\mathbb{M}^{\textup{s}}(\mathbb{M}^{\textup{s}}+\mathbb{M}^{\textup{a}})v,\,v\rangle & =a\langle\mathbb{M}^{\textup{s}}u,\,u\rangle+a\langle\mathbb{M}^{\textup{s}}v,\,v\rangle\;.\label{eq:assa2}
\end{align}
By the first identity of (\ref{eq:ide3}) and the Cauchy-Schwarz inequality
(with the fact that $|u|^{2}+|v|^{2}=1$), we get
\begin{align*}
a\langle\mathbb{M}^{\textup{s}}u,\,u\rangle+a\langle\mathbb{M}^{\textup{s}}v,\,v\rangle 
&=\left(\langle\mathbb{M}^{\textup{s}}u,\,u\rangle+\langle\mathbb{M}^{\textup{s}}v,\,v\rangle\right)^{2} \\
&\le|\mathbb{\mathbb{M}^{\textup{s}}}u|^{2}+|\mathbb{M}^{\textup{s}}v|^{2}=\langle(\mathbb{M}^{\textup{s}})^{2}u,\,u\rangle+\langle(\mathbb{M}^{\textup{s}})^{2}v,\,v\rangle\;.
\end{align*}
Injecting this to (\ref{eq:assa2}) completes the proof of (\ref{eq:assa1.5}).
\end{proof}

\subsection{Asymptotic stability of linear models}\label{subsec:stab of linear models}

In this subsection, we suppose that the force $F$ appearing in (\ref{eq:uld0})
is given by
\begin{equation}
F(q)=\mathbb{M}q\;\;\;;\;q\in\mathbb{R}^{d}\label{eq:linearforce}
\end{equation}
for some constant $d\times d$ matrix $\mathbb{M}$. Here, we completely
characterize the asymptotic stability of the linear dynamical system
(\ref{eq:uld0}).
\begin{thm}[Asymptotic stability: linear case]
\label{thm:linchar}For the linear model with a force $F$ given
by (\ref{eq:linearforce}), the process $(X_{t})_{t\geq0}$ is asymptotically
stable if and only if one of the following conditions holds:
\begin{enumerate} 
\item $\mathrm{Sp}(\mathbb{M})\subset\{a+ib:a>0\text{ and }\gamma^{2}a>b^{2}\}$.
\item  $\mathbb{K}_{\mathbb{M}}$ is positive definite on the
subspace of $\mathbb{C}^{d}$ spanned by eigenvectors of $\mathbb{M}$.
\end{enumerate}
In particular, if $\mathbb{K}_{\mathbb{M}}$ is positive definite,
then the process $(X_{t})_{t\geq0}$ is asymptotically stable.
\end{thm}

\begin{proof}
When the force $F$ is given by (\ref{eq:linearforce}), the ODE (\ref{eq:uld0})
describing the dynamics $(X_{t})_{t\ge0}$ can be written as
\begin{equation}
\frac{\mathrm{d}X_{t}}{\mathrm{d}t}=-\mathbb{T}_{\mathbb{M}}X_{t}\label{eqX_t}\;,
\end{equation}
where $\mathbb{T}_{\mathbb{M}}$ is the matrix given in (\ref{eq:T_M}).
It is well-known that the dynamical
system of the form (\ref{eqX_t}) is asymptotically stable if and
only if $\mathbb{T}_{\mathbb{M}}$ admits only eigenvalues with positive
real part. Thus, by Lemma \ref{lem:ev}, the dynamics $(X_{t})_{t\geq0}$
is asymptotically stable if and only if
\[
\mathrm{Sp}(\mathbb{M})\subset\{a+ib:a>0\text{ and }\gamma^{2}a>b^{2}\}\;.
\]
By Lemma \ref{lem:evrel}-(1), this is equivalent to the positive
definiteness of $\mathbb{K}_{\mathbb{M}}$ on the subspace of $\mathbb{C}^{d}$
spanned by eigenvectors of $\mathbb{M}$ and the proof is completed.
\end{proof}
\begin{rem}
It is well-known that in the linear model (\ref{eqX_t}), the process
$(X_{t})_{t\geq0}$ is even exponentially stable in the sense that
for all $t\geq0$,
\[
|X_{t}(x)|^{2}\le C|x|^{2}e^{-\lambda t}
\]
for some constants $C,\,\lambda>0$.
\end{rem}

We shall see below that the asymptotic stability of ($X_{t})_{t\geq0}$
can be obtained by looking only at the eigenvalues of $\gamma^{2}\mathbb{M}^{\textup{s}}+(\mathbb{M}^{\textup{a}})^{2}$
instead of $\mathbb{K}_{\mathbb{M}}$.
\begin{prop}
\label{prop:suff1} For the linear model with force $F$ given by
(\ref{eq:linearforce}), the process $(X_{t})_{t\geq0}$ is asymptotically
stable if the $d\times d$ symmetric matrix $\gamma^{2}\mathbb{M}^{\textup{s}}+(\mathbb{M}^{\textup{a}})^{2}$
is positive definite.
\end{prop}

\begin{proof}
Suppose that $\gamma^{2}\mathbb{M}^{\textup{s}}+(\mathbb{M}^{\textup{a}})^{2}$
is positive definite. Then, by Lemma \ref{lem:evrel}-(2), the matrix
$\mathbb{K}_{\mathbb{M}}$ is positive definite on the subspace of
$\mathbb{C}^{d}$ spanned by eigenvectors of $\mathbb{M}$. Hence,
the proof is completed by Theorem \ref{thm:linchar}.
\end{proof}
We show in the next proposition that the condition obtained in the
proposition above is actually a necessary and sufficient condition
when $\mathbb{M}$ is a normal matrix, i.e., satisfies $\mathbb{M}\mathbb{M}^{\dagger}=\mathbb{M}^{\dagger}\mathbb{M}$.
\begin{prop}
\label{prop:suff1_normal}Assume that the force $F$ is given by (\ref{eq:linearforce})
where $\mathbb{M}$ is a normal matrix. Then, the process $(X_{t})_{t\geq0}$
is asymptotically stable if and only if $\gamma^{2}\mathbb{M}^{\textup{s}}+(\mathbb{M}^{\textup{a}})^{2}$
is positive definite.
\end{prop}

\begin{proof}
By Proposition \ref{prop:suff1}, it suffices to prove that when $(X_{t})_{t\geq0}$
is asymptotically stable, the matrix $\gamma^{2}\mathbb{M}^{\textup{s}}+(\mathbb{M}^{\textup{a}})^{2}$
is positive definite. Suppose that $(X_{t})_{t\geq0}$ is asymptotically
stable. Since a normal matrix is unitarily similar to a diagonal matrix,
we can write $\mathbb{M}=\mathbb{U}\mathbb{D}\mathbb{U}^{*}$ where
$\mathbb{U}$ is a unitary matrix and
\[
\mathbb{D}=\textup{diag}\left(a_{1}+ib_{1},\,\,\dots,\,a_{d}+ib_{d}\right)
\]
is a diagonal matrix. Then, we can write
\begin{equation}
\gamma^{2}\mathbb{M}^{\textup{s}}+(\mathbb{M}^{\textup{a}})^{2}=\mathbb{U}\widetilde{\mathbb{D}}\mathbb{U^{*}}\label{eq:repmsma}\;,
\end{equation}
where the diagonal matrix $\widetilde{\mathbb{D}}$ is given by
\[
\widetilde{\mathbb{D}}=\textup{diag}\left(\gamma^{2}a_{1}-b_{1}^{2},\,\,\dots,\,\gamma^{2}a_{d}-b_{d}^{2}\right)\;.
\]
Hence, since we assumed that $(X_{t})_{t\geq0}$ is asymptotically stable,
by Theorem \ref{thm:linchar}, the matrix $\widetilde{\mathbb{D}}$
is positive definite matrix. By (\ref{eq:repmsma}), the proof is
completed.
\end{proof}
\begin{rem}
The previous proposition also implies that, for the linear model with
force $F$ given by (\ref{eq:linearforce}) with a normal matrix $\mathbb{M}$,
the process $(X_{t})_{t\geq0}$ is asymptotically stable if and only
if the matrix $\mathbb{K}_{\mathbb{M}}$ is positive definite, since
we have $\mathbb{\mathbb{M}^{\textup{s}}\mathbb{M}^{\textup{a}}}=\mathbb{M}^{\textup{a}}\mathbb{M}^{\textup{s}}$
when $\mathbb{M}$ is normal.
\end{rem}

We are now able to prove Theorem \ref{thm:linear case}. Note that
we shall prove Theorem \ref{thm:stability} in the next section and
we shall assume below that Theorem \ref{thm:stability} holds.
\begin{proof}[Proof of Theorem \ref{thm:linear case}]
Assume that $F(q)=\mathbb{M}q$ with a normal matrix $\mathbb{M}$,
and that the dynamics (\ref{eq:uld0}) is asymptotically stable. Define
$U:\mathbb{R}^{d}\rightarrow\mathbb{R}$ and $\ell:\mathbb{R}^{d}\rightarrow\mathbb{R}^{d}$
as
\[
U(q)=\frac{1}{2}\left\langle \mathbb{M}^{\textup{s}}q,\,q\right\rangle \;\;\;\text{and}\;\;\;\ell(q)=\mathbb{M}^{\textup{a}}q\;.
\]
Then, since $\nabla U(q)=\mathbb{M}^{\textup{s}}q$, we have $F(q)=\nabla U(q)+\ell(q)$.
Moreover, by Proposition \ref{prop:suff1_normal}, the matrix $\mathbb{M}^{\textup{s}}+\frac{1}{\gamma^{2}}(\mathbb{M}^{\textup{a}})^{2}$
is positive definite. Hence, we can find $\beta\in(0,\,\gamma)$ which
is close enough to $\gamma$ so that
\[
\mathbb{M}^{\textup{s}}+\frac{1}{\beta^{2}}(\mathbb{M}^{\textup{a}})^{2}\ge\alpha\mathbb{I}_{d}
\]
for some $\alpha>0$. Then, we have
\[
F(q)\cdot q-\frac{|\ell(q)|^{2}}{\beta^{2}}=\left\langle \mathbb{M}q,\,q\right\rangle -\frac{|\mathbb{M}^{\textup{a}}q|^{2}}{\beta^{2}}=\left\langle \left(\mathbb{M}^{\textup{s}}+\frac{1}{\beta^{2}}(\mathbb{M}^{\textup{a}})^{2}\right)q,\,q\right\rangle \ge\alpha|q|^{2}\;.
\]
Since in this case $U$ is at most of quadratic growth in the sense
of (\ref{e:quad}) one can easily deduce by Remark \ref{rem:stability}-(5)
that Assumption \ref{ass:main} is satisfied. Since the other direction
concerning the implication of global asymptotic stability from Assumption
\ref{ass:main}, is the content of Theorem \ref{thm:stability} we
can conclude this proof.
\end{proof}

\section{\label{sec4}Lyapunov Function }

In the remainder of this article, we shall always assume that Assumption
\ref{ass:main} holds (while we do not assume Assumption \ref{ass:DF}
until further notice). In this section, we construct a Lyapunov function
which decays along the trajectory of $(X_{t})_{t\geq0}$. In particular,
we are able to prove the global exponential stability of $(X_{t})_{t\geq0}$
stated in Theorem \ref{thm:stability} and also provide moment estimates
for the underdamped Langevin dynamics $(X_{t}^{\epsilon})_{t\geq0}$
in Sections \ref{sec42} and \ref{sec43}.

\subsection{Lyapunov function }

In this section, we provide a Lyaponov function $H:\mathbb{R}^{2d}\rightarrow\mathbb{R}$
with respect to the dynamical system (\ref{eq:uld0}). In the underdamped
model, finding such a function is not a trivial task since there is
no global contraction in general even for the energy function
\begin{equation}
\widehat{H}(x)=\frac{1}{2}|p|^{2}+U(q)\label{eq:defhatH}
\end{equation}
as the time-derivative
\[
\frac{\textrm{d}\widehat{H}}{\mathrm{d}t}(X_{t})=-\gamma|p_{t}|^{2}-\left\langle p_{t},\,\ell(q_{t})\right\rangle
\]
has no reason to be negative for all $(q_{t},p_{t})\in\mathbb{R}^{2d}$.
Our idea is to cleverly modify the function $\widehat{H}$ so as to
obtain a Lyapunov function.
\begin{defn}[Lyapunov function $H$]
\label{def:H}Recall constants $\alpha,\,\beta>0$ from Assumption
\ref{ass:main}. Let us take a constant $\lambda>0$ small enough
so that these three conditions
\begin{equation}
\frac{\lambda(\gamma-\lambda)}{2}\leq\alpha\;,\qquad\frac{2\lambda}{\gamma-\lambda}\leq\alpha\;,\;\;\;\;\text{and}\qquad\beta^{2}\leq\gamma(\gamma-\lambda)\label{eq:cond_lambda}
\end{equation}
are simultaneously satisfied, where the last condition can be satisfied
since $\beta\in(0,\,\gamma)$. We define the Lyapunov function $H:\mathbb{R}^{2d}\rightarrow\mathbb{R}$
as follows
\begin{equation}
H(x):=\frac{1}{2}|p|^{2}+\frac{\gamma-\lambda}{2}\left\langle q,\,p\right\rangle +\frac{(\gamma-\lambda)^{2}}{4}|q|^{2}+U(q)\;.\label{eq:defH}
\end{equation}
Notice that the function $H$ is obtained by modifying the quadratic
part of $\widehat{H}$ defined in (\ref{eq:defhatH}).

From now on the constant $\lambda$ shall always refer to the constant
defined in~\eqref{eq:cond_lambda}. The next result follows immediately
from the definition of $H$.
\end{defn}

\begin{lem}
\label{lem_Hquad}There exist constants $\kappa_{0}>0$ such that
\[
\frac{1}{\kappa_{0}}|x|^{2}\le H(x)\le\kappa_{0}|x|^{2}+U(q)\;\;\;\text{for all }x=(q,p)\in\mathbb{R}^{2d}\;.
\]
In addition, if $U$ is at most of quadratic growth in the sense of
(\ref{e:quad}), we can remove $U(q)$ term at the right-hand side
and thus the function $H$ is comparable to the quadratic function.
\end{lem}

\begin{proof}
Since we assumed that $U\ge0$, it suffices to observe that
\[
\frac{1}{6}|p|^{2}+\frac{(\gamma-\lambda)^{2}}{16}|q|^{2}\le\frac{1}{2}|p|^{2}+\frac{\gamma-\lambda}{2}\left\langle q,\,p\right\rangle +\frac{(\gamma-\lambda)^{2}}{4}|q|^{2}\le\frac{3}{4}|p|^{2}+\frac{\gamma^{2}}{2}|q|^{2}\;.
\]
\end{proof}
Next we shall show that $H$ is a Lyapunov function.
\begin{lem}
\label{lem_dH}It holds that
\begin{equation}
\frac{\mathrm{d}}{\mathrm{d}t}H(X_{t})\le-\lambda H(X_{t})\label{eq:dH_uld0}
\end{equation}
for all $t\ge0$, where $\lambda$ is the constant given in Definition
\ref{def:H}.
\end{lem}

\begin{proof}
By a direct computation, we can deduce that $-\frac{\mathrm{d}}{\mathrm{d}t}H(X_{t})-\lambda H(X_{t})$
is equal to
\begin{equation}
\left\langle p_{t},\,\ell(q_{t})\right\rangle +\frac{\gamma}{2}|p_{t}|^{2}+\frac{\gamma-\lambda}{2}\left\langle F(q_{t}),\,q_{t}\right\rangle -\lambda\frac{(\gamma-\lambda)^{2}}{4}|q|^{2}-\lambda U(q)\;.\label{dH1}
\end{equation}
By Assumption \ref{ass:main} and an elementary inequality
\[
\left\langle p_{t},\,\ell(q_{t})\right\rangle +\frac{\gamma}{2}|p_{t}|^{2}+\frac{1}{2\gamma}|\ell(q_{t})|^{2}\ge0\;,
\]
we can notice that (\ref{dH1}) is bounded from below by
\[
\frac{\gamma-\lambda}{2}\left[\left(\alpha-\frac{\lambda(\gamma-\lambda)}{2}\right)|q_{t}|^{2}+\left(\alpha-\frac{2\lambda}{\gamma-\lambda}\right)U(q_{t})+\left(\frac{1}{\beta^{2}}-\frac{1}{\gamma(\gamma-\lambda)}\right)|\ell(q_{t})|^{2}\right]
\]
which is non-negative by (\ref{eq:cond_lambda}). This completes the
proof.
\end{proof}

\subsection{\label{sec42}Global exponential stability of zero-noise dynamics}

The proof of the global exponential stability of the process $(X_{t})_{t\ge0}$
now follows immediately from the construction of the above Lyapunov
function $H$.
\begin{proof}[Proof of Theorem \ref{thm:stability}]
By Lemma \ref{lem_dH}, we have
\[
\frac{\mathrm{d}(\mathrm{e}^{\lambda t}H(X_{t}))}{\mathrm{d}t}=\mathrm{e}^{\lambda t}\left(\frac{\mathrm{d}H(X_{t})}{\mathrm{d}t}+\lambda H(X_{t})\right)\le0
\]
for all $t\ge0$. Thus, we get
\begin{equation}
\mathrm{e}^{\lambda t}H(X_{t}(x))\le H(x)\;.\label{eq:HXt}
\end{equation}
Hence, the proof is completed by recalling Lemma \ref{lem_Hquad}.
\end{proof}
We next discuss two consequences of Theorem \ref{thm:stability} which
will be used later. Define $G:\mathbb{R}^{2d}\rightarrow\mathbb{R}^{2d}$ as
\begin{equation}
G(x)=(-p,\,F(q)+\gamma p)\label{eq:G(x)}\;,
\end{equation}
so that we can rewrite the ODE (\ref{eq:uld0}) satisfied by $(X_{t})_{t\geq0}$
as follows:
\begin{equation}
\frac{\textup{d}X_{t}}{\mathrm{d}t}=-G(X_{t})\;.\label{eq:uld0_2}
\end{equation}
The following is the first direct consequence of Theorem \ref{thm:stability}
\begin{cor}
\label{cor:positiveev} All the complex eigenvalues of the matrix
$\textup{D}F(0)$ have positive real part. In particular, $\textup{D}F(0)$
is invertible.
\end{cor}

\begin{proof}
It is well-known that local asymptotic
stability around the equilibrium point $(0,0)$ of $(X_{t})_{t\geq0}$
implies that all the complex eigenvalues of the Jacobian matrix $\textup{D}G(0,\,0)$
have positive real part, where $G$ is the function defined in (\ref{eq:G(x)}).
Since $\textup{D}G(0,\,0)=\mathbb{T}_{\textup{D}F(0)}$ where the
matrix $\mathbb{T}_{\mathbb{M}}$ is the one defined in (\ref{eq:T_M}),
the proof follows from Lemma \ref{lem:ev}.
\end{proof}
\begin{notation*} \label{not_kappa}From now on, $\kappa$ and $\kappa_{0}$
always refer to the constants appearing in Theorem \ref{thm:stability}
and Lemma \ref{lem_Hquad}, respectively. \end{notation*}

For $r>0$, let us define $\mathcal{D}_{r},\,\mathcal{H}_{r}\subset\mathbb{R}^{2d}$
as follows
\begin{align}
\mathcal{D}_{r} & =\{x\in\mathbb{R}^{2d}:|x|\le r\}\;,\label{eq:def_DR}\\
\mathcal{H}_{r} & =\left\{ x\in\mathbb{R}^{2d}:|x|^{2}\le\kappa\left(r^{2}+\max_{y\in\mathcal{D}_{r}}U(y)\right)\right\} \;.\label{eq:def_HR}
\end{align}
Note that both $\mathcal{D}_{r}$ and $\mathcal{H}_{r}$ are $2d$-dimensional
closed balls. The following remark explains why we define $\mathcal{D}_{r}$
and $\mathcal{H}_{r}$.
\begin{rem}
\label{rem:Thmstat}By Theorem \ref{thm:stability}, if $x\in\mathcal{D}_{r}$,
then we have $X_{t}(x)\in\mathcal{H}_{r}$ for all $t\ge0$. In other
words, the trajectory $(X_{t}(x))_{t\ge0}$ is contained in the closed
ball $\mathcal{H}_{r}$.
\end{rem}

\begin{notation*} \label{not:constant}From this moment on, different
appearances of a constant may stand for different constants unless
otherwise mentioned as in Notation \ref{not_kappa}. For instance,
the constant $C(r)$ represents a constant depending on the parameter
$r$, but different appearances of $C(r)$ may refer to different
quantities. In addition, since we are interested in the regime $\epsilon\rightarrow0$,
all the statements below implicitly assume that $\epsilon$ is sufficiently
small compared to the other parameters. \end{notation*}

Let us now state the following Corollary which is a consequence of
Theorem \ref{thm:stability}.
\begin{cor}
\label{cor:contraction}For all $r>0$, there exists a constant $C(r)>0$
such that
\[
\left|X_{t}(u)-X_{t}(u')\right|\le e^{C(r)t}|u-u'|
\]
for all $u,\,u'\in\mathcal{D}_{r}$ and for all $t\ge0$.
\end{cor}

\begin{proof}
By (\ref{eq:uld0_2}), we have
\begin{equation}
\frac{\textup{d}}{\textup{d}t}\left|X_{t}(u)-X_{t}(u')\right|^{2}=2\left\langle X_{t}(u)-X_{t}(u'),\,G(X_{t}(u'))-G(X_{t}(u))\right\rangle \;.\label{eq:l2nom}
\end{equation}
By Remark \ref{rem:Thmstat} and the mean-value theorem, we have that
\[
\left|G(X_{t}(u))-G(X_{t}(u'))\right|\le C(r)\left|X_{t}(u)-X_{t}(u')\right|
\]
for all $u,\,u'\in\mathcal{D}_{r}$ where $C(r)=\sup_{x\in\mathcal{H}_{r}}\left\Vert \textup{D}G(x)\right\Vert $.
Hence, by (\ref{eq:l2nom}) and the Cauchy Schwarz inequality, we
get
\[
\frac{\textup{d}}{\textup{d}t}\left|X_{t}(u)-X_{t}(u')\right|^{2}\le C(r)\left|X_{t}(u)-X_{t}(u')\right|^{2}
\]
for all $u,\,u'\in\mathcal{D}_{r}$ which concludes the proof.
\end{proof}

\subsection{\label{sec43}Moment estimates}

Another consequence of the existence of the Lyapunov function $H$
is the control of the moments for the process $(X_{t}^{\epsilon})_{t\geq0}$.
\begin{prop}
\label{prop:mom_est}For all $n\geq0$, $x\in\mathbb{R}^{2d}$, and
$t\geq0$, it holds that
\begin{equation}
\mathbb{E}\left[|X_{t}^{\epsilon}(x)|^{2n}\right]\leq\kappa_{0}^{n}\omega_{n}\left(H(x)\mathrm{e}^{-\lambda t}+\frac{d\epsilon}{\lambda}\right)^{n},\label{eq:control moment n H(q,p)}
\end{equation}
where $\kappa_{0}$ is the constant appearing in Lemma \ref{lem_Hquad}
(cf. Notation \ref{not_kappa}), and the sequence $(\omega_{n})_{n\ge0}$
is defined by $\omega_{0}=\omega_{1}=1$ and
\[
\omega_{k}=\prod_{j=2}^{k}(d+2(j-1))\;\;\;\text{for }k\ge2\;.
\]
\end{prop}

First we need to analyze the following recurrence relation which allows
us to introduce the sequence $(\omega_{k})_{k=0}^{\infty}$ appearing
in Proposition \ref{prop:mom_est}.
\begin{lem}
\label{lem_seq}Let $(a_{n}(\cdot))_{n\in\mathbb{Z}^{+}}$ be a sequence
of continuous functions on $[0,\,\infty)$ such that $a_{0}\equiv1$
and
\begin{equation}
a_{n}'(t)\leq-n\lambda a_{n}(t)+n\epsilon(d+2(n-1))a_{n-1}(t)\label{eq:seq}
\end{equation}
for all $n\ge1$ and $t\ge0$. Then, for the sequence $(\omega_{n})_{n=0}^{\infty}$
defined in Proposition \ref{prop:mom_est}, we have that
\begin{equation}
a_{n}(t)\leq\omega_{n}\left[\sum_{k=1}^{n}\binom{n}{k}\frac{a_{k}(0)\mathrm{e}^{-k\lambda t}}{\omega_{k}}\left(\frac{\epsilon}{\lambda}\right)^{n-k}+d\left(\frac{\epsilon}{\lambda}\right)^{n}\right]\label{eq:seqbd}
\end{equation}
for all $n\in\mathbb{Z}^{+}$ and $t\ge0$.
\end{lem}

\begin{proof}
We prove this lemma by induction on $n$. Note that (\ref{eq:seqbd})
is immediate for $n=0$. Now let us take $n=1$ in (\ref{eq:seq})
so that we obtain $a_{1}'(t)\leq-\lambda a_{1}(t)+\epsilon d$. By
considering the differential of $\textup{e}^{\lambda t}a_{1}(t)$,
we can readily get $a_{1}(t)\leq\textup{e}^{-\lambda t}a_{1}(0)+\frac{\epsilon d}{\lambda}$
and thus we obtain (\ref{eq:seqbd}) for $n=1$.

Now assume that $n\ge2$ and that $\eqref{eq:seqbd}$ is satisfied
at $n-1$. By considering the differential of $\textup{e}^{n\lambda t}a_{n}(t)$
and applying (\ref{eq:seq}), we get
\[
a_{n}(t)\leq\mathrm{e}^{-n\lambda t}\left[a_{n}(0)+n\epsilon(d+2(n-1))\int_{0}^{t}\mathrm{e}^{n\lambda s}a_{n-1}(s)\mathrm{d}s\right]\;.
\]
By the induction hypothesis and the definition of $\omega_{n}$, we
can bound the right-hand side from above by
\begin{align*}
 & \mathrm{e}^{-n\lambda t}a_{n}(0)+\mathrm{e}^{-n\lambda t}n\epsilon\omega_{n}\left[\sum_{k=1}^{n-1}\binom{n-1}{k}\frac{a_{k}(0)}{\omega_{k}}\left(\frac{\epsilon}{\lambda}\right)^{n-1-k}\frac{\textup{e}^{(n-k)\lambda t}-1}{(n-k)\lambda}+d\left(\frac{\epsilon}{\lambda}\right)^{n-1}\frac{\textup{e}^{n\lambda t}-1}{n\lambda}\right]\\
 & \qquad\leq\omega_{n}\left[\sum_{k=1}^{n}\binom{n}{k}\frac{a_{k}(0)\mathrm{e}^{-k\lambda t}}{\omega_{k}}\left(\frac{\epsilon}{\lambda}\right)^{n-k}+d\left(\frac{\epsilon}{\lambda}\right)^{n}\right]
\end{align*}
which concludes the proof.
\end{proof}
We are now able to prove Proposition \ref{prop:mom_est}.
\begin{proof}[Proof of Proposition \ref{prop:mom_est}]
By Ito's formula and a similar computation to Lemma \ref{lem_dH},
we have that, for all $t\ge0$,
\begin{equation}
\mathrm{d}H(X_{t}^{\epsilon})\le-\big[\lambda H(X_{t}^{\epsilon})-\epsilon d\big]dt+\sqrt{2\epsilon}\left\langle p_{t}^{\epsilon}+\frac{\gamma-\lambda}{2}q_{t}^{\epsilon},\,\mathrm{d}B_{t}\right\rangle \;.\label{eq:dH_uld}
\end{equation}
Therefore, again by Ito's formula, for $n\ge1$, we get
\begin{align}
\mathrm{d}H(X_{t}^{\epsilon})^{n}= \;& nH(X_{t}^{\epsilon})^{n-1}\mathrm{d}H(X_{t}^{\epsilon})+\frac{1}{2}\times n(n-1)H(X_{t}^{\epsilon})^{n-2}\times2\epsilon\left|p_{t}^{\epsilon}+\frac{\gamma-\lambda}{2}q_{t}^{\epsilon}\right|^{2}\mathrm{d}t\nonumber \\
\leq \; & H(X_{t}^{\epsilon})^{n-1}\left[\left(-n\lambda H(X_{t}^{\epsilon})+n\epsilon(d+2(n-1))\right)\textup{d}t+n\sqrt{2\epsilon}\left\langle p_{t}^{\epsilon}+\frac{\gamma-\lambda}{2}q_{t}^{\epsilon},\,\mathrm{d}B_{t}\right\rangle \right]\label{eq:bdH^n}\;,
\end{align}
where the last line used (\ref{eq:dH_uld}) and the elementary bound
\[
\frac{1}{2}\left|p+\frac{\gamma-\lambda}{2}q\right|^{2}\le H(x)\;.
\]
Now let us fix $x\in\mathbb{R}^{2d}$ and let
\[
a_{n}(t)=\mathbb{E}\left[H(X_{t}^{\epsilon}(x))^{n}\right]\;\;\;\text{for }n\in\mathbb{Z}^{+}\text{ and }t\ge0\;.
\]
Then, by the bound (\ref{eq:bdH^n}), the sequence $(a_{n}(\cdot))_{n\in\mathbb{Z}^{+}}$
satisfies the relation (\ref{eq:seq}), and therefore by Lemma \ref{lem_seq},
we get (since $a_{k}(0)=H(x)^{k}$)
\[
\mathbb{E}\left[H(X_{t}^{\epsilon}(x))^{n}\right]\leq\omega_{n}\left[\sum_{k=1}^{n}\binom{n}{k}\frac{H(x)^{k}\mathrm{e}^{-k\lambda t}}{\omega_{k}}\left(\frac{\epsilon}{\lambda}\right)^{n-k}+d\left(\frac{\epsilon}{\lambda}\right)^{n}\right]\;.
\]
Since $\omega_{k}\ge1$ for all $k\in\mathbb{Z}^{+}$, we have a rough
bound
\[
\sum_{k=1}^{n}\binom{n}{k}\frac{H(x)^{k}\mathrm{e}^{-k\lambda t}}{\omega_{k}}\left(\frac{\epsilon}{\lambda}\right)^{n-k}+d\left(\frac{\epsilon}{\lambda}\right)^{n}\le\left(H(x)\mathrm{e}^{-\lambda t}+\frac{d\epsilon}{\lambda}\right)^{n}
\]
which is sufficient to conclude the proof along with Lemma \ref{lem_Hquad}.
\end{proof}
The following exponential moment estimate is a direct consequence
of Proposition \ref{prop:mom_est} and will be used later.
\begin{cor}
\label{cor_expmom}For all $x\in\mathbb{R}^{2d}$, $t\geq0$, and
$a\in\left(0,\,\frac{1}{2(d+2)\kappa_{0}\left(H(x)\mathrm{e}^{-\lambda t}+\frac{d\epsilon}{\lambda}\right)}\right)$,
we have that
\[
\mathbb{E}\left[\mathrm{exp}\left(a|X_{t}^{\epsilon}(x)|^{2}\right)\right]<2\;.
\]
\end{cor}

\begin{proof}
By Proposition \ref{prop:mom_est}, we get
\begin{align}
\mathbb{E}\left[\frac{a^{n}|X_{t}^{\epsilon}(x)|^{2n}}{n!}\right] & \leq a^{n}\frac{\kappa_{0}^{n}\omega_{n}}{n!}\left(H(x)\mathrm{e}^{-\lambda t}+\frac{d\epsilon}{\lambda}\right)^{n}\;.\label{eq:exp1}
\end{align}
Note that for $n\geq2$ we have the bound
\begin{equation}
\omega_{n}=\prod_{j=2}^{n}(d+2(j-1))\le\prod_{j=2}^{n}(d+2)(j-1)<(d+2)^{n}n!\label{eq:exp2}
\end{equation}
which also trivially holds for $n=0,\,1$. Combining (\ref{eq:exp1})
and (\ref{eq:exp2}), we can deduce that, for $0<a<\frac{1}{2(d+2)\kappa_{0}\left(H(x)\mathrm{e}^{-\lambda t}+\frac{d\epsilon}{\lambda}\right)}$,
\begin{align*}
\mathbb{E}\left[\mathrm{exp}\left(a|X_{t}^{\epsilon}(x)|^{2}\right)\right] & \le\sum_{n=0}^{\infty}a^{n}\frac{\kappa_{0}^{n}\omega_{n}}{n!}\left(H(x)\mathrm{e}^{-\lambda t}+\frac{d\epsilon}{\lambda}\right)^{n}\\
 & \le\sum_{n=0}^{\infty}\left\{ a\kappa_{0}(d+2)\left(H(x)\mathrm{e}^{-\lambda t}+\frac{d\epsilon}{\lambda}\right)\right\} ^{n}<\sum_{n=0}^{\infty}\left(\frac{1}{2}\right)^{n}=2\;.
\end{align*}
\end{proof}

\subsection{Uniqueness of stationary distribution\label{sec4.4}}
Let us conclude this section by showing that indeed the underdamped
process $(X_{t}^{\epsilon})_{t\geq0}$ admits a unique
stationary distribution $\mu_\epsilon$.
In order to do that let us first show the irreducibility of $(X_{t}^{\epsilon})_{t\geq0}$.
\begin{lem}\label{lem:irreducibility}
The process $(X_{t}^{\epsilon})_{t\geq0}$ is irreducible.
\end{lem}
\begin{proof}
It suffices to prove that for all $T>0$, $x\in\mathbb{R}^{2d}$ and any non-empty open set $\mathcal{A}\subset\mathbb{R}^{2d}$,
    $$\mathbb{P}(X^\epsilon_T(x)\in\mathcal{A})>0\;.$$
The proof is similar to~\cite[Lemma 3.4]{mattingly2002ergodicity} and   relies on classical tools from control theory. Let us fix $T>0$. It is enough to show the above property for any arbitrary $x\in\mathbb{R}^{2d}$ and any open ball $\mathcal{A}=\mathrm B(y,\delta)$ for arbitrary $y\in\mathbb{R}^{2d}$ and $\delta>0$.

First we shall define a path function $\phi\in\mathcal{C}^{2}([0,T],\mathbb{R}^d)$ such that
$$\left(\phi(0),\frac{\mathrm{d}\phi}{\mathrm{d}t}(0)\right)=x\;,\qquad\left(\phi(T),\frac{\mathrm{d}\phi}{\mathrm{d}t}(T)\right)=y\;,$$
which existence can be guaranteed by~\cite[Lemma 4.1]{kFP} for instance. Now define a function $f\in\mathcal{C}^1([0,T],\mathbb{R}^d)$ by $f(0)=0$ and for all $t\in[0,T]$,
$$\frac{\mathrm{d}^2\phi}{\mathrm{d}t^2}(t)+\gamma\frac{\mathrm{d}\phi}{\mathrm{d}t}(t)+F(\phi(t))=\sqrt{2\epsilon}\frac{\mathrm{d}f}{\mathrm{d}t}(t)\;.$$
Let $\Phi(t)=(\phi(t),\mathrm{d}\phi(t)/\mathrm{d}t)$. On the event $\{\sup_{t\in[0,T]}|B_t-f(t)|\leq\alpha\}$, by Gronwall's lemma, there exists a constant $C>0$ increasing with $\alpha$ such that
$$\sup_{t\in[0,T]}|X^\epsilon_t-\Phi(t)|\leq\sqrt{2\epsilon}\alpha\mathrm{e}^{CT}.$$
As a result, taking $\alpha$ small enough with respect to $\delta,T$ ensures that $\sup_{t\in[0,T]}|X^\epsilon_t-\Phi(t)|\leq \frac{\delta}{2}$, thus $X^\epsilon_T\in\mathcal{A}$. Therefore,
$$\mathbb{P}(X^\epsilon_T(x)\in\mathcal{A})\geq\mathbb{P}\big(\sup_{t\in[0,T]}|B_t-f(t)|\leq\alpha\big)>0$$
by the support theorem for the Brownian motions see~\cite[Theorem 4.20]{stroock1982lectures}.
\end{proof}

We are now able to prove Proposition~\ref{prop:unique stat distrib}.

\begin{proof}[Proof of Proposition \ref{prop:unique stat distrib}]
The proof relies on the criterias developed in~\cite{zhang2013new} ensuring the existence of a unique stationary distribution. Such criterias immediately follow from the existence of the Lyapunov function $H$ in Lemma~\ref{lem_dH}, the moment estimates in Proposition~\ref{prop:mom_est} in addition with the irreducibility of the process shown in Lemma~\ref{lem:irreducibility}. This concludes the proof.
\end{proof}

\section{\label{sec5}Approximation by a Gaussian Process}

In this section, we provide crucial estimates allowing us to analyze
the thermalization of the underdamped Langevin dynamics $(X_{t}^{\epsilon})_{t\geq0}$
in the regime $\epsilon\rightarrow0$. More precisely, we approximate
the process $(X_{t}^{\epsilon})_{t\geq0}$ by a Gaussian perturbation
of the deterministic process $(X_{t})_{t\geq0}$ (\ref{eq:uld0}).
This approximation will be carried out in terms of the $L^{p}$-distance
between the processes, and the total variation distance between the associated
laws. Namely Proposition \ref{prop:tvm1} is the key challenge in
the analysis of the total variation distance to the equilibrium measure
$\textup{d}_{\textup{TV}}(X_{t}^{\epsilon},\,\mu_{\epsilon})$.

We note that a similar approximation is also crucially used in the
investigation of the thermalization of the overdamped Langevin dynamics
in \cite[Section 3.2]{cut-off_overdamped}, but it turns out that the
proof of this approximation for the underdamped dynamics is much more
complicated since the dynamics is degenerate, and there is no global
contraction for the process $(X_{t}^{\epsilon})_{t\geq0}$. In \cite{cut-off_overdamped},
the strong coercivity assumption (H) was crucially used which implied
the strong local contraction in the zero-noise dynamics (cf. \cite[the second display of Section 3]{cut-off_overdamped}),
while we are not able to expect similar property in the underdamped
case.

\subsection{Heuristic explanation}

We start by explaining the heuristic argument leading to the main
Gaussian approximation, which idea is similar to the study of the
overdamped case \cite[Section 3.2]{cut-off_overdamped}. For $\epsilon>0$,
define the process $(Y_{t}^{\epsilon})_{t\ge0}$ in $\mathbb{R}^{d}$
as
\[
Y_{t}^{\epsilon}=\frac{X_{t}^{\epsilon}-X_{t}}{\sqrt{2\epsilon}}\;.
\]
Write $(\widetilde{B}_{t})_{t\ge0}$ a $2d$-dimensional process defined
by
\begin{equation}
\widetilde{B}_{t}=(0,\,B_{t})\in\mathbb{R}^{d}\times\mathbb{R}^{d}\label{eq:B_t_tilde}\;,
\end{equation}
where $(B_{t})_{t\geq0}$ is standard $d$-dimensional Brownian motion appearing
in (\ref{eq:uld}) so that we can write the SDE satisfied by $(X_{t}^{\epsilon})_{t\geq0}$
as (cf. (\ref{eq:G(x)}))
\begin{equation}
\textup{d}X_{t}^{\epsilon}=-G(X_{t}^{\epsilon})\textup{d}t+\sqrt{2\epsilon}\textup{d}\widetilde{B}_{t}\;.\label{eq:uld_2}
\end{equation}
Then, by (\ref{eq:uld0_2}) and (\ref{eq:uld_2}), we can write
\[
\textup{d}Y_{t}^{\epsilon}=-\frac{1}{\sqrt{2\epsilon}}\left[G(X_{t}^{\epsilon})-G(X_{t})\right]\text{d}t+\textup{d}\widetilde{B}_{t}\simeq-\textup{D}G(X_{t})Y_{t}^{\epsilon}\textup{d}t+\textup{d}\widetilde{B}_{t}\;.
\]
Note that the Jacobian $\textup{D}G(x)$ depends only on the position
vector $q$ of $x=(q,\,p)$ since
\begin{equation}
\mathbb{A}(q):=-\textup{D}G(x)=\begin{pmatrix}\mathbb{O}_{d} & \mathbb{I}_{d}\\
-\textup{D}F(q) & -\gamma\mathbb{I}_{d}
\end{pmatrix}\;.\label{eq:def A(q)}
\end{equation}
Additionally we expect that $(Y_{t}^{\epsilon})_{t\ge0}=(\frac{X_{t}^{\epsilon}-X_{t}}{\sqrt{2\epsilon}})_{t\ge0}$
is close to the Gaussian process $(Y_{t})_{t\ge0}$ defined by the
SDE
\begin{equation}
\textup{d}Y_{t}=\mathbb{A}(q_{t})Y_{t}\textup{d}t+\textup{d}\widetilde{B}_{t}\;,\label{eq:Y_t}
\end{equation}
with initial condition $Y_{0}=0$, when $\epsilon$ goes to zero.
With this observation in mind, we shall approximate the underdamped
Langevin dynamics $(X_{t}^{\epsilon})_{t\ge0}$ by the Gaussian process
$(Z_{t}^{\epsilon})_{t\ge0}$ defined by
\begin{equation}
Z_{t}^{\epsilon}=X_{t}+\sqrt{2\epsilon}Y_{t}\;.\label{eq:Z_t^eps}
\end{equation}
This approximation is formalized in more details in Proposition \ref{prop:main_apprx}
of the next subsection.
\begin{rem}
\label{rem:approx}The followings are remarks on the matrix $\mathbb{A}(q)$
and processes $(Y_{t})_{t\ge0},(Z_{t}^{\epsilon})_{t\ge0}$ defined
above.
\begin{enumerate}
\item We shall always couple the processes $(X_{t})_{t\ge0}$, $(X_{t}^{\epsilon})_{t\ge0},$$(Y_{t})_{t\ge0}$,
and $(Z_{t}^{\epsilon})_{t\ge0}$ naturally; i.e., we assume that
the processes $(X_{t})_{t\geq0},\,(X_{t}^{\epsilon})_{t\geq0}$ and
$(Z_{t}^{\epsilon})_{t\geq0}$ start together at some $x\in\mathbb{R}^{2d}$,
and that the processes $(X_{t}^{\epsilon})_{t\geq0}$ and $(Y_{t})_{t\geq0}$
share the same Brownian motion $(B_{t})_{t\geq0}$. When we would
like to stress that the starting point of these processes are $x$,
we denote them by $(X_{t}(x))_{t\geq0}$, $(X_{t}^{\epsilon}(x))_{t\geq0}$
and $(Z_{t}^{\epsilon}(x))_{t\geq0}$. Note that the process $(Y_{t}(x))_{t\geq0}$
does not start at $x$; instead it always starts at $0$. However,
the process depends on $x$ through the process $q_{t}=q_{t}(x)$
appearing in the equation (\ref{eq:Y_t}).
\item For the simplicity of notation, we use the abbreviation
\[
\mathbb{A}_{t}(x)=\mathbb{A}(q_{t}(x))
\]
and we shall even say $\mathbb{A}_{t}=\mathbb{A}_{t}(x)$ when there
is no risk of confusion regarding the starting point $x$.
\item We write $\mathbb{A}:=\mathbb{A}(0)$ for the convenience of notation.
Then, since $\mathbb{A}=-\textup{D}G(0)$, by the global exponential
stability established in Theorem \ref{thm:stability}, all the eigenvalues
of $\mathbb{A}$ have negative real part.
\item Since $\mathbb{A}(q)\rightarrow\mathbb{A}$ as $q\rightarrow0$ and
since $q_{t}(x)\rightarrow0$ as $t\rightarrow\infty$ by Theorem
\ref{thm:stability}, we have
\begin{equation}
\lim_{t\rightarrow\infty}\mathbb{A}_{t}(x)=\mathbb{A}\;\;\;\;\text{for all }x\in\mathbb{R}^{2d}\;.\label{eq:conv_A_t}
\end{equation}
\end{enumerate}
\end{rem}

\subsection{Gaussian approximation of the underdamped Langevin dynamics}

\begin{notation*} \label{not:T(x)}Recall the constant $\lambda$
and function $H(\cdot)$ from Definition \ref{def:H}, and the constant
$\kappa$ from Theorem \ref{thm:stability}. Then, define
\begin{equation}
T(x)=T^{\delta}(x)=\max\left\{ \frac{1}{\lambda}\log\frac{\kappa\left(|x|^{2}+U(q)\right)}{\delta^{2}},\,0\right\} \;\;\;\;;\;x=(q,p)\in\mathbb{R}^{2d}\label{eq:t(x)}\;,
\end{equation}
where $\delta>0$ is a sufficiently small constant which will be specified
later in Lemma \ref{lem:drift}. The definition of $T(x)$ will also
be explained therein. We also from now on fix a parameter $\theta\in(0,\,1/3)$,
and we shall always suppose that $\epsilon>0$ is small enough so
that $\frac{1}{\epsilon^{\theta}}>T(x)$, since we will focus on the
Gaussian approximation in the time interval $[T(x),\,1/\epsilon^{\theta}]$
in the next proposition. \end{notation*}

The following proposition is the underdamped version of \cite[Lemma 3.1]{cut-off_overdamped}
which provides a quantitative estimate on the error appearing in the
approximation of $(X_{t}^{\epsilon})_{t\ge0}$ by $(Z_{t}^{\epsilon})_{t\ge0}$ on the
time interval $[T(x),\,1/\epsilon^{\theta}]$.
\begin{prop}[Gaussian approximation]
\label{prop:main_apprx}For all $r>0$, $n\ge1$ and $\nu\in(0,1)$, there exists a
constant $\alpha>0$ (independent of $r,n$ and $\nu$)
and a constant $C(n,\,r)>0$ such that
\[
\sup_{x\in\mathcal{D}_{r}}\sup_{t\in[T(x),\,1/\epsilon^{\theta}]}\mathbb{E}\left[|X_{t}^{\epsilon}(x)-Z_{t}^{\epsilon}(x)|^{n}\right]\leq C(n,\,r)\epsilon^{n/2}\left[\frac{e^{-\alpha nt}}{\epsilon^{n\nu}}+\epsilon^{n(1-2\theta)}\right]
\]
for all small enough $\epsilon>0$.
\end{prop}

As we have mentioned earlier, the proof of this proposition is more
involved than its analogous in the overdamped case \cite[Lemma 3.1]{cut-off_overdamped}
since the underdamped Langevin dynamics does not have a strong global
contraction which is a crucial property assumed in \cite{cut-off_overdamped}.
We introduced $T(x)$ to resolve this issue. The proof of this proposition
will be given at the end of the current section.

\subsection{Matrix equations}

In this subsection, we provide a lemma allowing us to solve a certain
class of Lyapunov matrix equations which shall appear frequently throughout
this work. Note that in contrary to the overdamped case studied in~\cite{cut-off_overdamped},
the next lemma requires more work since we do not assume that the
matrix $\mathbb{W}$ appearing in~\eqref{eq:ricatti eqn} is positive
definite. This is related to the fact that the underdamped Langvin
dynamics is degenerate.
\begin{lem}
\label{lem:riccati}Let $\mathbb{W}$ be a $2d\times2d$ symmetric
non-negative matrix such that, for some $a>0$,
\begin{equation}
x\cdot\mathbb{W}x\geq a|p|^{2}\text{ \;\;\;\;for all }x=(q,\,p)\in\mathbb{R}^{d}\;.\label{eq:conric}
\end{equation}
Let $\mathbb{U}$ be a $2d\times2d$ matrix such that all the eigenvalues
of $\mathbb{U}$ have negative real part and moreover, when we write
$\mathbb{U}$ as block-form
\begin{equation}
\mathbb{U}=\begin{pmatrix}\mathbb{U}_{11} & \mathbb{U}_{12}\\
\mathbb{U}_{21} & \mathbb{U}_{22}
\end{pmatrix}\label{eq:blockU}\;,
\end{equation}
where all $\mathbb{U}_{ij}$'s are $d\times d$ matrices, the matrix
$\mathbb{U}_{21}$ is invertible. Then, there exists a unique $2d\times2d$
real matrix $\mathbb{X}$ such that
\begin{equation}
\mathbb{U}^{\dagger}\mathbb{X}+\mathbb{X}\mathbb{U}=-\mathbb{W}\;.\label{eq:ricatti eqn}
\end{equation}
Moreover, $\mathbb{\mathbb{X}}$ is symmetric positive definite.
\end{lem}

\begin{proof}
Since $\mathbb{U}$ is stable i.e., all the eigenvalues of $\mathbb{U}$
have negative real part, by the assumption of the lemma, by \cite[Theorem 2 in p.414]{lancaster1985theory}
there exists a unique solution $\mathbb{X}$ to the matrix equation
(\ref{eq:ricatti eqn}) and moreover by \cite[Theorem 3 in p.414]{lancaster1985theory},
this unique solution $\mathbb{X}$ admits the representation
\begin{equation}
\mathbb{X}=\int_{0}^{\infty}\mathrm{e}^{\mathbb{U}^{\dagger}t}\mathbb{W}\mathrm{e}^{\mathbb{U}t}\,\mathrm{d}t\label{eq:repsol}\;.
\end{equation}
From this representation, the symmetry and non-negative definiteness
of $\mathbb{X}$ is immediate. It now remains to prove the positive
definiteness of $\mathbb{X}$ which follows from the conditions on
$\mathbb{W}$ and $\mathbb{U}$ given in the lemma. By the representation
(\ref{eq:repsol}), we can write for $x\in\mathbb{R}^{2d}$,
\begin{align}
\left\langle x,\,\mathbb{X}x\right\rangle  & =\int_{0}^{\infty}\left\langle \mathrm{e}^{\mathbb{U}t}x,\,\mathbb{W}\mathrm{e}^{\mathbb{U}t}x\right\rangle \mathrm{d}t\;.\label{eq:xXx}
\end{align}
Writing $\mathrm{e}^{\mathbb{U}t}x=(\widehat{q}_{t},\,\widehat{p}_{t})$
so that by (\ref{eq:blockU}), we have the following system of ODE
satisfied by $(\widehat{q}_{t},\,\widehat{p}_{t})_{t\geq0}$:
\begin{equation}
\begin{cases}
\frac{\textup{d}}{\textup{d}t}\widehat{q}_{t}=\mathbb{U}_{11}\widehat{q}_{t}+\mathbb{U}_{12}\widehat{p}_{t}\;,\\
\frac{\textup{d}}{\textup{d}t}\widehat{p}_{t}=\mathbb{U}_{21}\widehat{q}_{t}+\mathbb{U}_{22}\widehat{p}_{t}\;.
\end{cases}\label{eq:ODEx}
\end{equation}
Now we suppose that $\left\langle x,\,\mathbb{X}x\right\rangle =0$,
then in view of (\ref{eq:conric}) and (\ref{eq:xXx}), we have $\widehat{p}_{t}\equiv0$ and hence  by the second equation of (\ref{eq:ODEx}) and by the   invertibility of $\mathbb{U}_{21}$, we also 
have $\widehat{q}_{t}\equiv0$. 
Thus, we can conclude that $x=(\widehat{q}_{0},\,\widehat{p}_{0})=0$
the positive definiteness of $\mathbb{X}$ follows.
\end{proof}
By Theorem~\ref{thm:stability}, Lemma~\ref{lem:ev} and Corollary~\ref{cor:positiveev}; the
global exponential stability ensures necessarily that both $\mathbb{A}$
and $\mathbb{A^{\dagger}}$ satisfy all the requirements imposed to
$\mathbb{U}$ in the previous lemma. Hence, from now on, we denote
by $\Gamma$ the solution to equation (\ref{eq:ricatti eqn}) with
$(\mathbb{U},\,\mathbb{W}):=(\mathbb{A},\,\mathbb{I}_{2d})$. Then,
by Lemma \ref{lem:riccati}, $\Gamma$ is a positive definite symmetric
matrix satisfying
\begin{equation}
\mathbb{A}^{\dagger}\Gamma+\Gamma\mathbb{A}=-\mathbb{I}_{2d}\;.\label{eq:ricattiGamma}
\end{equation}
Let us take $\xi>0$ small enough so that
\begin{equation}
\xi|x|^{2}\le\left\langle x,\,\Gamma x\right\rangle \le\frac{1}{\xi}|x|^{2}\;\;\;\text{for all }x\in\mathbb{R}^{2d}\;.\label{eq:normGamma}
\end{equation}

\subsection{Preliminary moment estimates on $(Y_{t})_{t\ge0}$ }

In this section, we provide a control on the moments of $(Y_{t})_{t\ge0}$.
We start from observing that, by Ito's formula (cf. (\ref{eq:Y_t})),
we can write
\begin{equation}
\mathrm{d}\langle Y_{t},\Gamma Y_{t}\rangle=2\langle Y_{t},\,\Gamma\mathbb{A}(q_{t})Y_{t}\rangle\mathrm{d}t+2\langle\Gamma Y_{t},\mathrm{d}\widetilde{B}_{t}\rangle+M\mathrm{d}t\label{eq:ito Y_t-1}\;,
\end{equation}
where $\Gamma$ is the matrix defined through (\ref{eq:ricattiGamma})
and where $M=\sum_{i=d+1}^{2d}\Gamma_{ii}$. The following lemma provides
a control on the drift appearing in the right-hand side of the equation
above. In particular, the next lemma fully characterizes the constant
$\delta>0$ appearing in Notation \ref{not:T(x)} in the definition
of $T(x)=T^{\delta}(x)$ in (\ref{eq:t(x)}).
\begin{lem}
\label{lem:drift}There exists $\delta>0$ such that, for all $x\in\mathbb{R}^{2d}$
and $t\ge T(x)=T^{\delta}(x)$, we have
\[
\langle y,\,\Gamma\mathbb{A}_{t}(x)y\rangle\le-\frac{\xi}{4}\langle y,\,\Gamma y\rangle\;\;\;\text{for all }y\in\mathbb{R}^{2d}\;,
\]
where the constant $\xi$ is the one appearing in (\ref{eq:normGamma}).
\end{lem}

\begin{proof}
Since $\mathbb{A}(q)\underset{q\rightarrow0}{\longrightarrow}\mathbb{A}$,
there exists $\delta>0$ small enough such that, for all $|q|\leq\delta$,
\begin{equation}
\left\Vert (\mathbb{A}(q)-\mathbb{A})^{\dagger}\Gamma+\Gamma(\mathbb{A}(q)-\mathbb{A})\right\Vert \leq\frac{1}{2}\;.\label{eq:A(q)-A-1}
\end{equation}
By Theorem \ref{thm:stability} and definition (\ref{eq:t(x)}) of
$T(x)$, if $t\ge T(x)$, we have
\begin{equation}
|q_{t}(x)|^{2}\le|X_{t}(x)|^{2}\le\kappa\left(|x|^{2}+U(q)\right)\textup{e}^{-\lambda t}\le\delta^{2}\label{eq:qtbd}
\end{equation}
and thus we have, by (\ref{eq:A(q)-A-1}), 
\begin{equation}
\left\Vert (\mathbb{A}_{t}(x)-\mathbb{A})^{\dagger}\Gamma+\Gamma(\mathbb{A}_{t}(x)-\mathbb{A})\right\Vert =\left\Vert (\mathbb{A}(q_{t}(x))-\mathbb{A})^{\dagger}\Gamma+\Gamma(\mathbb{A}(q_{t}(x))-\mathbb{A})\right\Vert \leq\frac{1}{2}\;.\label{eq:A(qT(x))}
\end{equation}
Hence, for $t\geq T(x)$, by (\ref{eq:A(qT(x))})
and (\ref{eq:ricattiGamma}), we get
\begin{align*}
2\langle y,\,\Gamma\mathbb{A}_{t}(x)y\rangle & =\left\langle y,\,\left[\mathbb{A}_{t}(x)^{\dagger}\Gamma+\Gamma\mathbb{A}_{t}(x)\right]y\right\rangle \\
 & \le\left\langle y,\,\left[\mathbb{A}^{\dagger}\Gamma+\Gamma\mathbb{A}\right]y\right\rangle +\frac{1}{2}|y|^{2}=-\frac{1}{2}|y|^{2}\le-\frac{\xi}{2}\langle y,\,\Gamma y\rangle\;,
\end{align*}
where the last inequality comes from (\ref{eq:normGamma}).
\end{proof}
The next lemma provides a control on the moments of $(Y_{t}(x))_{t\ge0}$ in
the time-interval $[0,T(x)]$.
\begin{lem}
\label{lem:mom_YT(x)}For all $n\ge1$ and $r>0$, there exists a
constant $C(n,\,r)>0$ such that
\[
\sup_{x\in\mathcal{D}_{r}}\sup_{t\in[0,\,T(x)]}\mathbb{E}\left[\langle Y_{t}(x),\,\Gamma Y_{t}(x)\rangle^{n}\right]\le C(n,\,r)\;.
\]
\end{lem}

\begin{proof}
By Remark \ref{rem:Thmstat}, for $x\in\mathcal{D}_{r}$, the entire
path $(X_{t}(x))_{t\ge0}$ is contained in the compact set $\mathcal{H}_{r}$,
we get
\[
\sup_{x\in\mathcal{D}_{r}}\sup_{t\ge0}\Vert\mathbb{A}(q_{t}(x))\Vert\le\sup_{x=(q,\,p)\in\mathcal{H}_{r}}\Vert\mathbb{A}(q)\Vert=C(r)\;.
\]
Thus, there exists a constant $C(r)>0$ such that, for all $x\in\mathcal{D}_{r}$
and $t\ge0$,
\begin{equation}
\langle\mathbb{A}(q_{t}(x))Y_{t}(x),\,\Gamma Y_{t}(x)\rangle\le C(r)\langle Y_{t}(x),\,\Gamma Y_{t}(x)\rangle\;.\label{eq:lbd}
\end{equation}

Now we prove the statement by induction. First consider
the case $n=1$. Fix $x\in\mathcal{D}_{r}$. Then, by (\ref{eq:ito Y_t-1})
and (\ref{eq:lbd}), we have
\[
\mathrm{d}\langle Y_{t}(x),\,\Gamma Y_{t}(x)\rangle\le C(r)\langle Y_{t}(x),\,\Gamma Y_{t}(x)\rangle\mathrm{d}t+2\langle\Gamma Y_{t}(x),\,\mathrm{d}\widetilde{B}_{t}\rangle+M\mathrm{d}t\;.
\]
Therefore, by the Gronwall inequality we get (since $Y_{0}(x)=0$)
\[
\mathbb{E}\left[\langle Y_{t}(x),\,\Gamma Y_{t}(x)\rangle\right]\leq Mte^{C(r)t}\;\;\;\;\text{for all }t\ge0\;.
\]
This proves the initial case of the induction since
\begin{equation}
\sup_{x\in\mathcal{D}_{r}}T(x)<\infty\;.\label{eq:bddT}
\end{equation}

Next we suppose that the statement holds for $1,\,\dots,\,n-1$ and
look at the statement for $n\ge2$. By Ito's formula (cf. (\ref{eq:ito Y_t-1}))
and (\ref{eq:lbd}), there exists a constant $C(r)>0$ such that
\begin{align}
\mathrm{d}\langle Y_{t}(x),\,\Gamma Y_{t}(x)\rangle^{n}\le & \; n\langle Y_{t}(x),\,\Gamma Y_{t}(x)\rangle^{n-1}\left[\left(C(r)\langle Y_{t}(x),\,\Gamma Y_{t}(x)\rangle+M\right)\mathrm{d}t+2\langle\Gamma Y_{t}(x),\,\mathrm{d}\widetilde{B}_{t}\rangle\right]\nonumber \\
 & +\frac{M}{2}n(n-1)\langle Y_{t}(x),\,\Gamma Y_{t}(x)\rangle^{n-2}\mathrm{d}t\;,\label{eq:y^n ito}
\end{align}
and therefore we get
\begin{align*}
 \mathbb{E}\left[\langle Y_{t}(x),\,\Gamma Y_{t}(x)\rangle^{n}\right]
   \le\; &  \int_{0}^{t}\bigg( nC(r)\mathbb{E}\left[\langle Y_{s}(x),\,\Gamma Y_{s}(x)\rangle^{n}\right]+nM\mathbb{E}\left[\langle Y_{s}(x),\,\Gamma Y_{s}(x)\rangle^{n-1}\right] \\
  &+\frac{n(n-1)}{2}M\mathbb{E}\left[\langle Y_{s}(x),\,\Gamma Y_{s}(x)\rangle^{n-2}\right]\bigg)\mathrm{d}s\;.
\end{align*}
By the induction hypothesis, there exist constants $C_{1}(n,\,r),\,C_{2}(n,\,r)>0$
such that for all $t\in[0,T(x)]$ and $x\in\mathcal{D}_{r}$,
$$\mathbb{E}\left[\langle Y_{t}(x),\,\Gamma Y_{t}(x)\rangle^{n}\right]\leq C_{1}(n,\,r)\int_{0}^{t}\mathbb{E}\left[\langle Y_{s}(x),\,\Gamma Y_{s}(x)\rangle^{n}\right]\mathrm{d}s+C_{2}(n,\,r)t\;.$$
Thus, the proof is again completed by the Gronwall inequality and
(\ref{eq:bddT}).
\end{proof}
\begin{lem}
\label{lem:mom_Yt}For all $n\ge1$ and $r>0$, there exists a constant
$C(n,\,r)>0$ such that
\[
\sup_{x\in\mathcal{D}_{r}}\sup_{t:t\ge T(x)}\mathbb{E}\left[|Y_{t}(x)|^{n}\right]\leq C(n,\,r)\;.
\]
\end{lem}

\begin{proof}
By the positive definiteness of the matrix $\Gamma$, it suffices
to prove that
\begin{equation}
\sup_{x\in\mathcal{D}_{r}}\sup_{t:t\ge T(x)}\mathbb{E}\left[\langle Y_{t}(x),\,\Gamma Y_{t}(x)\rangle^{n}\right]\leq C(n,\,r)\;.\label{eq:pup0}
\end{equation}
We prove this bound by mathematical induction as well. First consider
the case $n=1$. By (\ref{eq:ito Y_t-1}) and Lemma \ref{lem:drift},
we have, for all $t\geq T(x)$,
\begin{align*}
\frac{\mathrm{d}\mathbb{E}\left[\langle Y_{t}(x),\,\Gamma Y_{t}(x)\rangle\right]}{\mathrm{d}t} & \leq-\frac{\xi}{2}\mathbb{E}\left[\langle Y_{t}(x),\,\Gamma Y_{t}(x)\rangle\right]+M\;.
\end{align*}
Consequently, for all $t\geq T(x)$, by the Gronwall inequality
\[
\mathbb{E}\left[\langle Y_{t}(x),\,\Gamma Y_{t}(x)\rangle\right]\leq e^{-\frac{\xi(t-T(x))}{2}}\mathbb{E}\left[\langle Y_{T(x)}(x),\,\Gamma Y_{T(x)}(x)\rangle\right]+\frac{2M}{\xi}
\]
and hence the proof of the case $n=1$ of (\ref{eq:pup0}) is completed
by Lemma \ref{lem:mom_YT(x)}.

Next we suppose that the statement holds for $1,\,\dots,\,n-1$ and
prove the statement for $n$. By (\ref{eq:y^n ito}) and Lemma \ref{lem:drift},
we have, for all $t\geq T(x)$,
\begin{align*}
\mathrm{d}\langle Y_{t}(x),\,\Gamma Y_{t}(x)\rangle^{n}\le\; & n\langle Y_{t}(x),\,\Gamma Y_{t}(x)\rangle^{n-1}\left[\left(-\frac{\xi}{2}\langle Y_{t}(x),\,\Gamma Y_{t}(x)\rangle+M\right)\mathrm{d}t+2\langle\Gamma Y_{t},\mathrm{d}\widetilde{B}_{t}\rangle\right]\\
 & +\frac{M}{2}n(n-1)\langle Y_{t}(x),\,\Gamma Y_{t}(x)\rangle^{n-2}\mathrm{d}t\;.
\end{align*}
Hence, by the induction hypothesis, we get for all $x\in\mathcal{D}_{r}$
and $t\geq T(x)$,
\[
\frac{\textup{d}\mathbb{E}\left[\langle Y_{t}(x),\,\Gamma Y_{t}(x)\rangle^{n}\right]}{\textup{d}t}\leq-\frac{\xi n}{2}\mathbb{E}\left[\langle Y_{t}(x),\,\Gamma Y_{t}(x)\rangle^{n}\right]+C(n,\,r)\;.
\]
Therefore, the proof is completed by the Gronwall
inequality and Lemma \ref{lem:mom_YT(x)}.
\end{proof}
We conclude this subsection with the control of the supremum of the
moments of $(Y_{t}(x))_{t\ge0}$ on the time interval $[T(x),1/\epsilon^{\theta}]$.
\begin{lem}
\label{lem:mom_Yt_sup}For all $r>0$ and $n\ge1$, there exists a
constant $C(n,\,r)>0$ such that
\[
\sup_{x\in\mathcal{D}_{r}}\mathbb{E}\left[\sup_{t\in[T(x),\,1/\epsilon^{\theta}]}|Y_{t}(x)|^{n}\right]\leq\frac{C(n,\,r)}{\epsilon^{n\theta/2}}
\]
for all small enough $\epsilon>0$, where $\theta$ is a constant
introduced in Notation \ref{not:T(x)}.
\end{lem}

\begin{proof}
By (\ref{eq:ito Y_t-1}), for all $t\geq T(x)$, we have
\[
\langle Y_{t}(x),\,\Gamma Y_{t}(x)\rangle\leq\langle Y_{T(x)}(x),\,\Gamma Y_{T(x)}(x)\rangle+2\int_{T(x)}^{t}\langle\Gamma Y_{s}(x),\,\textup{d}\widetilde{B}_{s}\rangle+Mt
\]
and therefore by Lemma \ref{lem:mom_YT(x)}, we get 
\begin{align}
&\mathbb{E}\left[\sup_{t\in[T(x),\,1/\epsilon^{\theta}]}\langle Y_{t}(x),\,\Gamma Y_{t}(x)\rangle^{n}\right] \nonumber \\
&\leq3^{n-1}\left(C(n,r)+2^{n}\mathbb{E}\left[\sup_{t\in[T(x),\,1/\epsilon^{\theta}]}\left|\int_{T(x)}^{t}\langle\Gamma Y_{r}(x),\mathrm{\,d}\widetilde{B}_{r}\rangle\right|^{n}\right]+\left(\frac{M}{\epsilon^{\theta}}\right)^{n}\right)\;.\label{eq:bdx}
\end{align}
By the Burkholder-Davis-Gundy inequality, it holds that
\begin{equation}
\mathbb{E}\left[\sup_{t\in[T(x),\,1/\epsilon^{\theta}]}\left|\int_{T(x)}^{t}\langle\Gamma Y_{r}(x),\,\mathrm{d}\widetilde{B}_{r}\rangle\right|^{n}\right]\le C(n)\,\mathbb{E}\left[\left(\int_{T(x)}^{1/\epsilon^{\theta}}|\Gamma Y_{r}(x)|^{2}\mathrm{d}r\right)^{n/2}\right]\;.\label{bdx1}
\end{equation}
By the trivial bound $|\Gamma Y_{r}|^{2}\le ||\Gamma||^2|Y_{r}|^{2}$ and H\"older's inequality, the right-hand
side of the previous display is bounded from above by
\begin{equation}
C(n)\left(\frac{1}{\epsilon^{\theta}}\right)^{n/2-1}\mathbb{E}\left[\left(\int_{T(x)}^{1/\epsilon^{\theta}}|Y_{r}(x)|^{n}\mathrm{d}r\right)\right]\le C(n)\left(\frac{1}{\epsilon^{\theta}}\right)^{n/2}\label{bdx2}\;,
\end{equation}
where the last inequality follows from Lemma \ref{lem:mom_Yt}. Hence,
by (\ref{bdx1}) and (\ref{bdx2}), we have

\[
\mathbb{E}\left[\sup_{t\in[T(x),\,1/\epsilon^{\theta}]}\left|\int_{T(x)}^{t}\langle\Gamma Y_{r}(x),\,\mathrm{d}\widetilde{B}_{r}\rangle\right|^{n}\right]\le C(n,r)\left(\frac{1}{\epsilon^{\theta}}\right)^{n/2}\;.
\]
Inserting this to (\ref{eq:bdx}) and recalling (\ref{eq:normGamma}),
we get
\[
\mathbb{E}\left[\sup_{t\in[T(x),\,1/\epsilon^{\theta}]}|Y_{t}(x)|^{2n}\right]\leq\frac{C(n,r)}{\epsilon^{n\theta}}\;,
\]
which concludes the proof.
\end{proof}

\subsection{Quadratic Gronwall inequality}

In this section we provide a proof of a quadratic Gronwall inequality which shall be crucial in the proof of Proposition~\ref{prop:main_apprx}.

 \begin{lem}\label{lem:quadratic gronwall}
Suppose that a  non-negative function $u\in\mathcal{C}^1(\mathbb{R}_+,\mathbb{R}_+)$ satisfies
\begin{equation}\label{eq:edo u}
    \frac{\mathrm{d}u(t)}{\mathrm{d}t}\leq a-bu(t)+cu(t)^2
\end{equation}
for all $t\geq0$, where $a,b,c > 0$ are constants satisfying  $\delta:=b^2-4ac>0$. Denote by
$$\alpha=\frac{b-\sqrt{\delta}}{2c}\;,\quad\text{}\quad\beta=\frac{b+\sqrt{\delta}}{2c}$$
two solutions of quadratic equation $a-bx+cx^2=0$. Let $M>1$ and suppose that
\begin{equation}\label{eq:assumption u_0}
    u(0)\leq\frac{M\alpha+\beta}{M+1}\;.
\end{equation}
Then, for all $t\geq0$, we have
$$u(t)\leq\alpha+\frac{\sqrt{\delta}}{Mc}\mathrm{e}^{-\sqrt{\delta} t}\;.$$
\end{lem}
\begin{proof} The proof is divided into several cases.
\smallskip

\noindent \textbf{{[}Case 1: $u(0)<\alpha${]} } In this case, we claim that $u(t)<\alpha$ for all $t\ge 0$. Suppose not so that we can define  $$T_\alpha=\inf\{t>0:u(t)=\alpha\}<\infty\;.$$
Observe that 
$$
 a-bu(t)+cu(t)^2 = c(u(t)-\alpha)(u(t)-\beta)>0$$
for all $t\in[0,T_\alpha)$. Hence, 
It follows from~\eqref{eq:edo u} that, for all $t\in[0,T_\alpha)$,
$$\frac{1}{c(u(t)-\alpha)(u(t)-\beta)}\, \frac{\mathrm{d}u(t)}{\mathrm{d}t} \leq1\;.$$
 Integrating over $[0,t]\subset [0,T_\alpha)$ ensures that
$$\log\left(\frac{\beta-u(t)}{\alpha-u(t)}\right)\leq\log\left(\frac{\beta-u(0)}{\alpha-u(0)}\right)+\sqrt{\delta} t\;.$$
As a result, for all $t\in[0,T_\alpha)$,

$$u(t)\leq\alpha-\frac{\sqrt{\delta}}{c}\frac{1}{\left(\frac{\beta-u(0)}{\alpha-u(0)}\right)\mathrm{e}^{\sqrt{\delta} t}-1}\;.$$
Now we get a contradiction by letting $t\nearrow T_\alpha<\infty$ which yields
$$\alpha\leq\alpha-\frac{\sqrt{\delta}}{c}\frac{1}{\left(\frac{\beta-u(0)}{\alpha-u(0)}\right)\mathrm{e}^{\sqrt{\delta} T_\alpha}-1}<\alpha\;.$$ 

\smallskip
\noindent \textbf{{[}Case 2: $u(0)=\alpha${]} }In this case, let us take $\eta>0$  small enough so that $\delta_\eta:=b^2-4(a+\eta)c>0$. By~\eqref{eq:edo u}, for all $t\geq0$,
$$\frac{\mathrm{d}u(t)}{\mathrm{d}t}\leq a+\eta-bu(t)+cu(t)^2\;.$$
Then
$$u(0)=\alpha=\frac{b-\sqrt{\delta}}{2c}<\frac{b-\sqrt{\delta_\eta}}{2c}\;.$$
As a result, the same reasoning as in Case 1 ensures that for all $t\geq0$ and $\eta>0$,
$$u(t)\leq\frac{b-\sqrt{\delta_\eta}}{2c}\;.$$
Taking $\eta\rightarrow0$ yields that
$u(t)\leq\alpha$ for all $t\geq0$.

\smallskip
\noindent \textbf{{[}Case 3: $u(0)>\alpha${]} }Let us now finally consider the case
$\alpha<u(0)\leq \frac{M\alpha+\beta}{M+1}$ for a constant $M>1$. Note here that, since $\beta>\alpha$, we have
$$\alpha<\frac{\beta+M\alpha}{M+1}<\beta$$
and hence we have $u(0)\in(\alpha,\beta)$. Let us now define
$$T_\beta=\inf\{t>0:u(t)=\beta\}\;.$$
Then, for all $t\in[0,T_\alpha\land T_\beta)$, we have 
$$
 a-bu(t)+cu(t)^2 = c(u(t)-\alpha)(u(t)-\beta)<0$$
 and therefore, integrating the equation as in Case 1, we get 
$$\log\left(\frac{\beta-u(t)}{u(t)-\alpha}\right)\geq\log\left(\frac{\beta-u(0)}{u(0)-\alpha}\right)+\sqrt{\delta} t\;.$$
As a result, for all $t\in[0,T_\alpha\land T_\beta)$,
\begin{equation}\label{qgw1}
u(t)\leq\alpha+\frac{\sqrt{\delta}}{c}\frac{1}{\left(\frac{\beta-u(0)}{u(0)-\alpha}\right)\mathrm{e}^{\sqrt{\delta} t}+1}
\leq\alpha+\frac{\sqrt{\delta}}{Mc}\mathrm{e}^{-\sqrt{\delta} t}\;,
\end{equation}
where the second inequality follows from \eqref{eq:assumption u_0}. Since $\alpha+\frac{\sqrt{\delta}}{Mc}=\alpha+\frac{\beta-\alpha}{M}<\beta$
the previous bound itself implies that $T_\beta=\infty$, and hence \eqref{qgw1} holds for $t\in [0,\,T_\alpha)$. Since we have $u(t)\le \alpha$ for all $t\in [T_\alpha,\,\infty)$ by Case 2, we are done. 
\end{proof}

\begin{rem}
Unlike to the ordinary Gronwall's inequality, the function $u(t)$ may blow-up when we start from $u(0)>\beta$. 
\end{rem}

\subsection{Proof of Proposition \ref{prop:main_apprx}}

We are now ready to prove Proposition \ref{prop:main_apprx} which
is the most computational part of the current article.
\begin{proof}[Proof of Proposition \ref{prop:main_apprx}]
The proof is divided into several steps. Let us fix $r>0$ and $x\in\mathcal{D}_{r}$.
We shall omit the dependency on $x$ in this proof; for instance,
we write $X_{t}^{\epsilon}$ and $\mathbb{A}_{t}$ instead of $X_{t}^{\epsilon}(x)$
and $\mathbb{A}_{t}(x)$, respectively.

\smallskip{}
\noindent \textbf{{[}Step 1{]} }Preliminary analysis

\noindent Define the process $(W_{t}^{\epsilon})_{t\ge0}$ as
\begin{equation}
W_{t}^{\epsilon}:=\frac{X_{t}^{\epsilon}-Z_{t}^{\epsilon}}{\sqrt{2\epsilon}}=\frac{X_{t}^{\epsilon}-X_{t}}{\sqrt{2\epsilon}}-Y_{t}\;.\label{eq:wte}
\end{equation}
In view of (\ref{eq:normGamma}), it suffices to show that, for all
$t\in[T(x),\,\frac{1}{\epsilon^{\theta}}]$,
\begin{equation}
\mathbb{E}\left[\langle W_{t}^{\epsilon},\,\Gamma W_{t}^{\epsilon}\rangle^{n}\right]\leq C(n,\,r)\left[\frac{e^{-c_{1}nt}}{\epsilon^{n\nu}}+\epsilon^{n(1-2\theta)}\right]\;.\label{eq:obj0}
\end{equation}

By (\ref{eq:uld}), (\ref{eq:uld0}), and (\ref{eq:Y_t}) (since $(X_{t}^{\epsilon})_{t\geq0}$
and $(Y_{t})_{t\geq0}$ share the same Brownian motion), we can write
\[
\frac{\mathrm{d}\langle W_{t}^{\epsilon},\,\Gamma W_{t}^{\epsilon}\rangle}{\mathrm{d}t}=2\langle W_{t}^{\epsilon},\,\Gamma\mathbb{A}_{t}W_{t}^{\epsilon}\rangle-2\langle W_{t}^{\epsilon},\,\Gamma V_{t}^{\epsilon}\rangle\;,
\]
where
\begin{equation}
V_{t}^{\epsilon}=\frac{1}{\sqrt{2\epsilon}}\begin{pmatrix}0\\
F(q_{t}^{\epsilon})-F(q_{t})-\textup{D}F(q_{t})(q_{t}^{\epsilon}-q_{t})
\end{pmatrix}\;.\label{eq:V_t^eps}
\end{equation}
Hence, by Lemma \ref{lem:drift} and the positive-definiteness of
$\Gamma$, we have, for $t\ge T(x)$,
\begin{align}
\frac{\mathrm{d}\langle W_{t}^{\epsilon},\,\Gamma W_{t}^{\epsilon}\rangle}{\mathrm{d}t} & \leq-\frac{\xi}{2}\langle W_{t}^{\epsilon},\,\Gamma W_{t}^{\epsilon}\rangle-2\langle W_{t}^{\epsilon},\,\Gamma V_{t}^{\epsilon}\rangle\le-c\langle W_{t}^{\epsilon},\,\Gamma W_{t}^{\epsilon}\rangle+C|V_{t}^{\epsilon}|^{2}\label{eq:dW1}
\end{align}
for small enough constant $c>0$ and large enough constant $C>0$.\smallskip{}

\noindent \textbf{{[}Step 2{]} }Control of $|V_{t}^{\epsilon}|^{2}$

\noindent In view of (\ref{eq:dW1}), the next step is to control
$|V_{t}^{\epsilon}|^{2}$ . To that end, define the following event
(cf. Notation \ref{not_kappa})
\begin{equation}
\mathcal{A}_{\epsilon}=\mathcal{A}_{\epsilon}(x):=\left\{ \sup_{s\in[0,\,1/\epsilon^{\theta}]}|X_{s}^{\epsilon}|^{2}\leq\kappa_{0}H(x)+1\right\} \;.\label{eq:A_eps}
\end{equation}
Write
\[
I_{\epsilon}:=[T(x),\,1/\epsilon^{\theta}]
\]
and fix $t\in I_{\epsilon}$. As observed in (\ref{eq:qtbd}), we
have $|q_{t}|\leq\delta$. On the other hand, under the event $\mathcal{A}_{\epsilon}$,
we have $|q_{t}^{\epsilon}|\leq\sqrt{\kappa_{0}H(x)+1}$. Therefore,
by the mean-value theorem (since $F\in\mathcal{C}^{2}(\mathbb{R}^{d},\mathbb{R}^{d})$)
and by (\ref{eq:V_t^eps}), we have
\begin{equation}
|V_{t}^{\epsilon}|\leq\frac{C(r)}{\epsilon^{1/2}}|q_{t}^{\epsilon}-q_{t}|^{2}\label{eq:bdV}
\end{equation}
under the event $\mathcal{A}_{\epsilon}$, where
\[
C(r)=~\sup_{x\in\mathcal{D}_{r}}\,\sup_{z\in\mathbb{R}^{d}:\,|z|\leq\delta+\sqrt{\kappa_{0}H(x)+1}}|D^{2}F(z)|\;.
\]
By definition of $W_{t}^{\epsilon}$ (cf (\ref{eq:wte})), we have
\[
|q_{t}^{\epsilon}-q_{t}|^{2}\le|X_{t}^{\epsilon}-X_{t}|^{2}=2\epsilon|Y_{t}+W_{t}^{\epsilon}|^{2}\le4\epsilon\left[|Y_{t}|^{2}+|W_{t}^{\epsilon}|^{2}\right]\;.
\]
Combining this with (\ref{eq:bdV}) and applying (\ref{eq:normGamma}),
we get
\begin{align}
|V_{t}^{\epsilon}| & \leq C(r)\epsilon^{1/2}\left(\langle W_{t}^{\epsilon},\,\Gamma W_{t}^{\epsilon}\rangle+|Y_{t}|^{2}\right)\;.\label{eq:bdstep2}
\end{align}
We remind that this inequality holds under the event $\mathcal{A}_{\epsilon}$
and for $t\in I_{\epsilon}$.

\smallskip{}
\noindent \textbf{{[}Step 3{]} }Control of $\langle W_{t}^{\epsilon},\,\Gamma W_{t}^{\epsilon}\rangle$

\noindent In this step, we always put ourself under the event $\mathcal{A}_{\epsilon}$
and our aim is to provide a control on $\langle W_{t}^{\epsilon},\,\Gamma W_{t}^{\epsilon}\rangle$.
By (\ref{eq:dW1}) and (\ref{eq:bdstep2}), we get, for $t\in I_{\epsilon}$,
\begin{align*}
\frac{\mathrm{d}\langle W_{t}^{\epsilon},\,\Gamma W_{t}^{\epsilon}\rangle}{\mathrm{d}t} & \leq-c_0\langle W_{t}^{\epsilon},\,\Gamma W_{t}^{\epsilon}\rangle+C_0(r)\epsilon\left[\langle W_{t}^{\epsilon},\,\Gamma W_{t}^{\epsilon}\rangle^{2}+|Y_{t}|^{4}\right]\;
\end{align*}
for some constants $c_0$, $C_0(r)>0$. In the remainder of the proof, the constant $c_0$ and $C_0(r)$ always refer to the constants appearing in this bound. 

Let us now apply the quadratic Gronwall inequality stated in Lemma~\ref{lem:quadratic gronwall} to the function $t\geq0\mapsto\langle W_{t+T_x}^{\epsilon},\,\Gamma W_{t+T_x}^{\epsilon}\rangle$. In order to do that, write
\begin{equation}\label{del_epsc}
\delta_\epsilon:=c_0^2-4C_0(r)^2\epsilon^2\sup_{s\in I_{\epsilon}}|Y_{s}|^{4}\;.
\end{equation}
Then, we define as in Lemma~\ref{lem:quadratic gronwall}, 
$$\alpha_\epsilon=\frac{c_0-\sqrt{\delta_\epsilon}}{2C_0(r)\epsilon}\;,\quad\text{}\quad\beta_\epsilon=\frac{c_0+\sqrt{\delta_\epsilon}}{2C_0(r)\epsilon}\;.$$
Note that we shall later condition on the event $\delta_\epsilon>0$ and therefore we can think of $\alpha_\epsilon$ and $\beta_\epsilon$ as real numbers. 

Let us now take $\nu\in(0,1)$ and define, for $\epsilon$ small enough, 
$$M_\epsilon:=\epsilon^{\nu-1}>1\;.$$ 
It follows then from Lemma~\ref{lem:quadratic gronwall} that under the
event $\mathcal{B}_{\epsilon}$ defined by
\begin{equation}
\mathcal{B}_{\epsilon}=\mathcal{B}_{\epsilon}(x):=\left\{\delta_\epsilon>\frac{c_0^2}{2},\langle W_{T(x)}^{\epsilon},\,\Gamma W_{T(x)}^{\epsilon}\rangle\leq\frac{\beta_\epsilon+M_\epsilon\alpha_\epsilon}{M_\epsilon+1}\right\} \;,\label{eq:perov_ass}
\end{equation}
one has that, for all $t\in I_{\epsilon}$,
\begin{align*}
\langle W_{t}^{\epsilon},\,\Gamma W_{t}^{\epsilon}\rangle & \leq\alpha_\epsilon+\frac{\sqrt{\delta_\epsilon}}{M_\epsilon C_0(r)\epsilon}\mathrm{e}^{-\sqrt{\delta_\epsilon} (t-T_x)} \leq\alpha_\epsilon+\frac{c_0\mathrm{e}^{c_0 T_x/\sqrt{2}}}{\epsilon^\nu C_0(r)}\mathrm{e}^{-c_0t/\sqrt{2}}\;,
\end{align*}
where the last inequality follows from the inequality $c_0^2/2\leq\delta_\epsilon\leq c_0^2$. Therefore, we get 
\begin{align*}
\langle W_{t}^{\epsilon},\,\Gamma W_{t}^{\epsilon}\rangle^{n} & \leq2^{n-1}\left(\alpha_\epsilon^n+C(n,r)\frac{\mathrm{e}^{-nc_0t/\sqrt{2}}}{\epsilon^{n\nu}}\right)\;.
\end{align*}
Moreover, by an elementary inequality $1-\sqrt{1-x}\le x$ for $x\in [0,\,1]$, 
\begin{align*}
\alpha_\epsilon^n&=\frac{c^n}{(2C_0(r)\epsilon)^n}\left(1-\sqrt{1-\frac{4C_0(r)^2\epsilon^2\sup_{s\in I_{\epsilon}}|Y_{s}|^{4}}{c_0^2}}\right)^n\\
&\leq\frac{c_0^n}{(2C_0(r)\epsilon)^n}\left(\frac{4C_0(r)^2\epsilon^2\sup_{s\in I_{\epsilon}}|Y_{s}|^{4}}{c_0^2}\right)^n
\leq C(n,r)\epsilon^n\sup_{s\in I_{\epsilon}}|Y_{s}|^{4n}\;.
\end{align*}
Consequently, for $t\in I_{\epsilon}$, by Lemma \ref{lem:mom_Yt_sup},
we can conclude that
\begin{equation}
\mathbb{E}\left[\langle W_{t}^{\epsilon},\,\Gamma W_{t}^{\epsilon}\rangle^{n}\mathbf{1}_{\mathcal{A}_{\epsilon}\cap\mathcal{B}_{\epsilon}}\right]\leq C(n,r)\left(\epsilon^{n(1-2\theta)}+\frac{\mathrm{e}^{-nc_0t/\sqrt{2}}}{\epsilon^{n\nu}}\right)\;.\label{eq:bdWt_AB}
\end{equation}

This allows us to control $\langle W_{t}^{\epsilon},\,\Gamma W_{t}^{\epsilon}\rangle^{n}$
in (\ref{eq:obj0}) on the event $\mathcal{A}_{\epsilon}\cap\mathcal{B}_{\epsilon}$.
It remains now to provide a control on the event $\mathcal{A}_{\epsilon}^{c}\cup\mathcal{B}_{\epsilon}^{c}$.

\smallskip{}
\noindent \textbf{{[}Step 4{]} }Control of $\langle W_{t}^{\epsilon},\,\Gamma W_{t}^{\epsilon}\rangle$
on the event $\mathcal{A}_{\epsilon}^{c}\cup\mathcal{B}_{\epsilon}^{c}$.

\noindent By the Cauchy-Schwarz inequality and (\ref{eq:normGamma}),
we have
\begin{equation}
\mathbb{E}\left[\langle W_{t}^{\epsilon},\,\Gamma W_{t}^{\epsilon}\rangle^{n}\mathbf{1}_{\mathcal{A}_{\epsilon}^{c}\cup\mathcal{B}_{\epsilon}^{c}}\right]^{2}\leq C(n)\mathbb{E}\left[|W_{t}^{\epsilon}|^{4n}\right]\mathbb{P}(\mathcal{A}_{\epsilon}^{c}\cup\mathcal{B}_{\epsilon}^{c})\;.\label{eq:bdWt_AB-c}
\end{equation}
By Proposition \ref{prop:mom_est}, Theorem \ref{thm:stability},
and Lemma \ref{lem:mom_Yt}, for $t\in I_{\epsilon}$ and $x\in\mathcal{D}_{r}$,
\[
\mathbb{E}\left[|W_{t}^{\epsilon}|^{4n}\right]\leq\frac{3^{4n-1}}{(2\epsilon)^{2n}}\left(\mathbb{E}\left[|X_{t}^{\epsilon}|^{4n}\right]+|X_{t}|^{4n}+(2\epsilon)^{2n}\mathbb{E}\left[|Y_{t}|^{4n}\right]\right)\le\frac{C(n,\,r)}{\epsilon^{2n}}\;,
\]
and inserting this to (\ref{eq:bdWt_AB-c}) yields
\[
\mathbb{E}\left[\langle W_{t}^{\epsilon},\,\Gamma W_{t}^{\epsilon}\rangle^{n}\mathbf{1}_{\mathcal{A}_{\epsilon}^{c}\cup\mathcal{B}_{\epsilon}^{c}}\right]\leq\frac{C(n,\,r)}{\epsilon^{n}}\mathbb{P}(\mathcal{A}_{\epsilon}^{c}\cup\mathcal{B}_{\epsilon}^{c})^{1/2}\;.
\]
Therefore, it suffices to prove that 
\begin{equation}
\mathbb{P}(\mathcal{A}_{\epsilon}^{c})\le C(n,\,r)\epsilon^{4n(1-\theta)}\;\;\;\text{and}\;\;\;\mathbb{P}(\mathcal{B}_{\epsilon}^{c})\le C(n,\,r)\epsilon^{4n(1-\theta)}\label{eq:obj1}
\end{equation}
to complete the proof of (\ref{eq:obj0}). 

\smallskip{}
\noindent \textbf{{[}Step 5{]} }Estimate of $\mathbb{P}(\mathcal{A}_{\epsilon}^{c})$

\noindent By Ito's formula and similar computation   carried
out in the proof of Proposition \ref{prop:mom_est}, we get
\[
H(X_{t}^{\epsilon})\leq H(x)+\sqrt{2\epsilon}\int_{0}^{t}\left\langle p_{s}^{\epsilon}+\frac{\gamma-\lambda}{2}q_{s}^{\epsilon},\,\mathrm{d}B_{s}\right\rangle +d\epsilon t\;.
\]
Thus, by Lemma \ref{lem_Hquad}, for $t\in[0,1/\epsilon^{\theta}]$,
we have
\[
|X_{t}^{\epsilon}|^{2}\leq\kappa_{0}H(x)+\kappa_{0}\sqrt{2\epsilon}\int_{0}^{t}\left\langle p_{s}^{\epsilon}+\frac{\gamma-\lambda}{2}q_{s}^{\epsilon},\,\mathrm{d}B_{s}\right\rangle +\kappa_{0}d\epsilon^{1-\theta}\;.
\]
Therefore, by definition (\ref{eq:A_eps}), for $\epsilon$ small
enough such that $\kappa_{0}d\epsilon^{1-\theta}\leq\frac{1}{2}$,
we have
\[
\mathbb{P}(\mathcal{A}_{\epsilon}^{c})\leq\mathbb{P}\left(\kappa_{0}\sqrt{2\epsilon}\sup_{t\in[0,1/\epsilon^{\theta}]}\left|\int_{0}^{t}\left\langle p_{s}^{\epsilon}+\frac{\gamma-\lambda}{2}q_{s}^{\epsilon},\,\mathrm{d}B_{s}\right\rangle \right|>\frac{1}{2}\right)\;.
\]
By the Markov inequality, the Burkholder-Davis-Gundy inequality and
the H\"older inequality, we can bound the probability at the right-hand
side by
\begin{align*}
 & C(n)\epsilon^{4n}\,\mathbb{E}\left[\sup_{t\in[0,1/\epsilon^{\theta}]}\left|\int_{0}^{t}\left\langle p_{s}^{\epsilon}+\frac{\gamma-\lambda}{2}q_{s}^{\epsilon},\,\mathrm{d}B_{s}\right\rangle \right|^{8n}\right]\\
 & \leq C(n)\epsilon^{4n}\,\mathbb{E}\left[\left(\int_{0}^{1/\epsilon^{\theta}}\left|p_{s}^{\epsilon}+\frac{\gamma-\lambda}{2}q_{s}^{\epsilon}\right|^{2}\mathrm{d}s\right)^{4n}\right]\\
 & \leq C(n)\epsilon^{4n}\,\mathbb{E}\left[\left(\int_{0}^{1/\epsilon^{\theta}}\left|p_{s}^{\epsilon}+\frac{\gamma-\lambda}{2}q_{s}^{\epsilon}\right|^{8n}\mathrm{d}s\right)\left(\frac{1}{\epsilon^{\theta}}\right)^{4n-1}\right]\;.
\end{align*}
Since $|p_{s}^{\epsilon}+\frac{\gamma-\lambda}{2}q_{s}^{\epsilon}|^{2}\leq C|X_{s}^{\epsilon}|^{2}$
for some large enough constant $C>0$, we can conclude from the previous
bound and Proposition \ref{prop:mom_est} that
\begin{equation}
\mathbb{P}(\mathcal{A}_{\epsilon}^{c})\le C(n,\,r)\epsilon^{4n(1-\theta)}\;.\label{eq:obj1-4}
\end{equation}
This completes the proof of the first inequality in (\ref{eq:obj1}).

\smallskip{}
\noindent \textbf{{[}Step 6{]} }Estimate of $\mathbb{P}(\mathcal{B}_{\epsilon}^{c})$

\noindent In view of the definition (\ref{eq:perov_ass}) of $\mathcal{B}_{\epsilon}$
and (\ref{eq:normGamma}), it suffices to prove that, for all $c(r)>0$,
there exists a constant $C(n,\,r)>0$ such that
\begin{equation}
\mathbb{P}\left[\delta_\epsilon\leq\frac{c_0^2}{2}\right]\le C(n,\,r)\epsilon^{4n(1-\theta)}\;\;\;\;\text{and\;\;\;\;}\mathbb{P}\left[\langle W_{T(x)}^{\epsilon},\,\Gamma W_{T(x)}^{\epsilon}\rangle>\frac{\beta_\epsilon+M_\epsilon\alpha_\epsilon}{M_\epsilon+1}\right]\le C(n,\,r)\epsilon^{4n(1-\theta)}\;\label{eq:obj2}
\end{equation}
By definition \eqref{del_epsc} of $\delta_\epsilon$ and by  Lemma~\ref{lem:mom_Yt_sup}, we have
\begin{align*}
\mathbb{P}\left[\delta_\epsilon\leq\frac{c_0^2}{2}\right]&\leq\mathbb{P}\left[\epsilon^2\sup_{s\in I_{\epsilon}}|Y_{s}|^{4}\geq\frac{c_0^2}{8C_0(r)^2}\right]\\
&\leq C(n,r)\epsilon^{4n}\mathbb{E}\left[\sup_{s\in I_{\epsilon}}|Y_{s}|^{8n}\right]\\
&\leq C(n,r)\epsilon^{4n(1-\theta)}.
\end{align*}

Regarding the second estimate in \eqref{eq:obj2}, by (\ref{eq:obj1-4}) it is enough to control
$$\mathbb{P}\left[\left\{\langle W_{T(x)}^{\epsilon},\,\Gamma W_{T(x)}^{\epsilon}\rangle>\frac{\beta_\epsilon+M_\epsilon\alpha_\epsilon}{M_\epsilon+1}\right\}\cap \mathcal{A}_\epsilon\right].$$
Moreover, since $\beta_\epsilon\geq\frac{c}{2C(r)\epsilon}$ and using the Markov inequality,
\begin{align}
&\mathbb{P}\left[\left\{\langle W_{T(x)}^{\epsilon},\,\Gamma W_{T(x)}^{\epsilon}\rangle>\frac{\beta_\epsilon+M_\epsilon\alpha_\epsilon}{M_\epsilon+1}\right\}\cap \mathcal{A}_\epsilon\right]\nonumber\\
&\le\mathbb{P}\left[\left\{\langle W_{T(x)}^{\epsilon},\,\Gamma W_{T(x)}^{\epsilon}\rangle>\frac{c_0}{2C_0(r)\epsilon(M_\epsilon+1)}\right\}\cap \mathcal{A}_\epsilon\right]\nonumber\\
&\le C(N,r)(\epsilon M_\epsilon+\epsilon)^N\mathbb{E}\left[\langle W_{T(x)}^{\epsilon},\,\Gamma W_{T(x)}^{\epsilon}\rangle^N\mathbf{1}_{\mathcal{A}_\epsilon}\right],\label{eq:upperbound markov ineq}
\end{align}
where we are free to take $N\geq1$ large enough such that $\nu N\geq 4n(1-\theta)$.

Next step is to control the expectation appearing in the right-hand
side of the above inequality. By (\ref{eq:normGamma}) and by the
definition (\ref{eq:wte}) of $W_{t}^{\epsilon}$, we have
\begin{equation}
\langle W_{T(x)}^{\epsilon},\,\Gamma W_{T(x)}^{\epsilon}\rangle^{N}\leq C(N)\left[\frac{1}{\epsilon^{N}}|X_{T(x)}^{\epsilon}-X_{T(x)}|^{2N}+|Y_{T(x)}|^{2N}\right]\;.\label{eq:bdWXW}
\end{equation}
By definition (\ref{eq:A_eps}) of $\mathcal{A}_{\epsilon}$ and Theorem
\ref{thm:stability}, we can find a large enough $R(r)>0$ such that,
for all $x\in\mathcal{D}_{r}$ we have
\begin{equation}
|X_{t}^{\epsilon}|,\,|X_{t}|\le R(r)\;\;\;\text{for all }t\in[0,\,T(x)]\label{eq:bdXX}
\end{equation}
under $\mathcal{A}_{\epsilon}$. Hence, by (\ref{eq:uld0_2}), (\ref{eq:uld_2}) and the standard argument
based on Gronwall's lemma along with (\ref{eq:bdXX}), we obtain
\begin{equation}
|X_{T(x)}^{\epsilon}-X_{T(x)}|^{2N}\le C(N,\,r)\epsilon^{N}\sup_{s\in[0,\,T(x)]}|B_{s}|^{2N}\label{eq:estA}
\end{equation}
under $\mathcal{A}^{\epsilon}$. Inserting this to (\ref{eq:bdWXW})
and applying Lemma \ref{lem:mom_Yt}, we get
\begin{equation}
\mathbb{E}\left[\langle W_{T(x)}^{\epsilon},\,\Gamma W_{T(x)}^{\epsilon}\rangle^{N}\mathbf{1}_{\mathcal{A}_{\epsilon}}\right]\le C(N,\,r)\;.\label{eq:eWZW}
\end{equation}
Reinjecting into~\eqref{eq:upperbound markov ineq} ensures that
$$\mathbb{P}\left[\left\{\langle W_{T(x)}^{\epsilon},\,\Gamma W_{T(x)}^{\epsilon}\rangle>\frac{\beta_\epsilon+M_\epsilon\alpha_\epsilon}{M_\epsilon+1}\right\}\cap \mathcal{A}_\epsilon\right]\leq C(N,r)\epsilon^{\nu N}\;.$$
This completes the proof since we can take $N$ so that $\nu N\geq 4n(1-\theta)$.
\end{proof}
\section{\label{sec6}Covariance Matrix of the Gaussian approximation }

Denote by $\Sigma_{t}(x)$ the covariance matrix of the Gaussian random
vector $(Y_{t}(x))_{t\ge0}$ (cf. (\ref{eq:Y_t})). In this section, we investigate
the behavior of $\Sigma_{t}(x)$ in the regime $t\rightarrow\infty$
and $t\rightarrow0$. In the regime $t\rightarrow\infty$, we prove
that $\Sigma_{t}(x)$ converges to $\Sigma$ defined in (\ref{eq:Sigma}).
On the other hand, it is clear that $\Sigma_{t}(x)\rightarrow\mathbb{O}_{2d}$
as $t\rightarrow0$. We investigate the speed of these two convergences.
Using this convergence analysis we will be able to provide estimates
on the total variation distance $\textup{d}_{\textup{TV}}(Z_{t}^{\epsilon}(x),\,Z_{t}^{\epsilon}(y))$
for $x,\,y\in\mathbb{R}^{2d}$ when $t\rightarrow0$ or $t\rightarrow\infty$.

\subsection{Covariance matrix of $(Y_{t})_{t\ge0}$}

We first show that the covariance matrix $\Sigma_{t}(x)$ satisfies
an ODE. Using this ODE we will then provide short and long-time asymptotics
of $\Sigma_{t}(x)$.
\begin{lem}
\label{lem:Sigma_t}The covariance matrix $\Sigma_{t}(x)$ satisfies
\[
\frac{\mathrm{d}\Sigma_{t}(x)}{\mathrm{d}t}=\mathbb{J}+\mathbb{A}_{t}(x)\Sigma_{t}(x)+\Sigma_{t}(x)\mathbb{A}_{t}(x)^{\dagger}
\]
where $\mathbb{J}$ is the matrix defined in (\ref{eq:J}).
\end{lem}

\begin{proof}
Let us denote by $(\Pi_{t})_{t\geq0}=(\Pi_{t}(x))_{t\geq0}$ the solution
of the matrix differential equation
\begin{equation}
\dot{\Pi}_{t}=-\Pi_{t}\mathbb{A}_{t}(x)\text{\;\;\;\;for }t\ge0\label{eq:Pi_eq}
\end{equation}
with initial condition $\Pi_{0}=\mathbb{I}_{2d}$ so that we can write
the process $(Y_{t}(x))_{t\geq0}$ as follows
\[
Y_{t}(x)=\Pi_{t}^{-1}\int_{0}^{t}\Pi_{s}\mathrm{d}\widetilde{B}_{s}\;\;\;;\;t\ge0\;,
\]
where $(\widetilde{B}_{t})_{t\geq0}$ is the process defined in (\ref{eq:B_t_tilde}).
In particular the covariance matrix $\Sigma_{t}(x)$ is given by:
\begin{equation}
\Sigma_{t}(x)=\Pi_{t}^{-1}\left(\int_{0}^{t}\Pi_{s}\mathbb{J}\Pi_{s}^{\dagger}\mathrm{d}s\right)(\Pi_{t}^{-1})^{\dagger}\;.\label{eq:Sig_t}
\end{equation}
The remaining computation is routine. By (\ref{eq:Pi_eq}), we have
\[
\frac{\mathrm{d}\Pi_{t}^{-1}}{\mathrm{d}t}=-\Pi_{t}^{-1}\frac{\mathrm{d}\Pi_{t}}{\mathrm{d}t}\Pi_{t}^{-1}=\mathbb{A}_{t}(x)\Pi_{t}^{-1}
\]
and thus a direct computation of (\ref{eq:Sig_t}) yields
\begin{align*}
\frac{\mathrm{d}\Sigma_{t}(x)}{\mathrm{d}t}=\; & \mathbb{A}_{t}(x)\Pi_{t}^{-1}\left(\int_{0}^{t}\Pi_{s}\mathbb{J}\Pi_{s}^{\dagger}\mathrm{d}s\right)(\Pi_{t}^{-1})^{\dagger}
+\Pi_{t}^{-1}\Pi_{t}\mathbb{J}\Pi_{t}^{\dagger}(\Pi_{t}^{-1})^{\dagger}\\
&+\Pi_{t}^{-1}\left(\int_{0}^{t}\Pi_{s}\mathbb{J}\Pi_{s}^{\dagger}\mathrm{d}s\right)
(\mathbb{A}_{t}(x)\Pi_{t}^{-1})^{\dagger}\\
 =\;& \mathbb{A}_{t}(x)\Sigma_{t}(x)+\mathbb{J}+\Sigma_{t}(x)\mathbb{A}_{t}^{\dagger}(x)\;.
\end{align*}
\end{proof}

\subsection{Long-time behavior of the covariance matrix}

If the matrix $\Sigma_{t}(x)$ converges to some matrix $\Sigma_{\infty}$,
then, by Lemma \ref{lem:Sigma_t} and (\ref{eq:conv_A_t}), we expect
that the matrix $\Sigma_{\infty}$ satisfies $\mathbb{A}\Sigma_{\infty}+\Sigma_{\infty}\mathbb{A}^{\dagger}=-\mathbb{J}$.
Thus, by (\ref{eq:Sigma}) we should have $\Sigma_{\infty}=\Sigma$.
This heuristic argument will be rigorously and quantitatively justfied
in Lemma \ref{lem:long_Sigma} below. Before stating and proving the
lemma, we recall several well-known facts from linear algebra.
\begin{enumerate}
\item \emph{Frobenius norm.} The Frobenius norm of a square matrix $\mathbb{M}$
is defined by
\[
\Vert\mathbb{M}\Vert_{\textup{F}}=\textup{tr}(\mathbb{M}\mathbb{M}^{\dagger})^{1/2}\;.
\]
It is well-known that $\Vert\mathbb{M}\Vert\le\Vert\mathbb{M}\Vert_{\textup{F}}$
and the Frobenius norm is submultiplicative, i.e., for all square
matrices $\mathbb{M}$ and $\mathbb{M}'$ of same size, it holds that
\begin{equation}
\Vert\mathbb{M}\mathbb{M}'\Vert_{\textup{F}}\le\Vert\mathbb{M}\Vert_{\textup{F}}\Vert\mathbb{M}'\Vert_{\textup{F}}\;.\label{eq:submul}
\end{equation}
\item We can  infer from \emph{Von Neumann's trace inequality} that,
if $\mathbb{S}$ is a positive semi-definite symmetric matrix, and
$\mathbb{T}$ is a diagonalizable matrix of same size, then we have
\begin{equation}
\textup{tr}(\mathbb{S}\mathbb{T})\le\textup{tr}(\mathbb{S})\Vert\mathbb{T}\Vert\;.\label{eq:VNT1}
\end{equation}
If in addition $\mathbb{T}$ is a positive definite symmetric matrix
whose smallest eigenvalue is $c>0$, then we have
\begin{equation}
\textup{tr}(\mathbb{S}\mathbb{T})\ge c\,\textup{tr}(\mathbb{S})\label{eq:VNT2}\;.
\end{equation}
\end{enumerate}
Now we state the main result for the long-time behavior of the covariance
matrix $\Sigma_{t}(x)$.
\begin{lem}
\label{lem:long_Sigma}For all $r>0$, there exist constants $C(r)>0$ and $\alpha(r)>0$
such that
\begin{equation}
\sup_{x\in\mathcal{D}_{r}}\left\Vert \Sigma_{t}(x)-\Sigma\right\Vert _{\textup{F}}\le C(r)e^{-\alpha(r) t}\;\;\;\;\text{for all }t\ge0\;.\label{eq:convSigma}
\end{equation}
\end{lem}

\begin{proof}
Fix $x\in\mathcal{D}_{r}$ and we ignore the dependency of the matrices
to $x$ throughout the proof. We define the symmetric matrix $\Delta_{t}=\Delta(x)$ and the matrix $\mathbb{D}_{t}$ by:
\[
\Delta_{t}:=\Sigma_{t}-\Sigma \;,\qquad \mathbb{D}_{t}=\mathbb{A}_{t}-\mathbb{A}\;.
\]
Then, by (\ref{eq:Sigma}) and Lemma \ref{lem:Sigma_t}, we have
\begin{align*}
\frac{\mathrm{d}\Delta_{t}}{\mathrm{d}t}=\frac{\mathrm{d}\Sigma_{t}}{\mathrm{d}t} &=\mathbb{J}+\mathbb{A}_{t}\Sigma_{t}+\Sigma_{t}\mathbb{A}_{t}^{\dagger}\\
&=\mathbb{A}_{t}\Delta_{t}+\Delta_{t}\mathbb{A}_{t}^{\dagger}+\mathbb{D}_{t}\Sigma+\Sigma\mathbb{D}_{t}^{\dagger}\;.
\end{align*}
Recall the positive definite symmetric matrix $\Gamma$ defined in
(\ref{eq:ricattiGamma}). By the previous computation on $\frac{\mathrm{d}\Delta_{t}}{\mathrm{d}t}$,
we get
\begin{align}
\frac{\mathrm{d}}{\mathrm{d}t}\mathrm{tr}(\Delta_{t}\Gamma\Delta_{t}\Gamma) & =\mathrm{tr}\left((\mathbb{A}_{t}\Delta_{t}+\Delta_{t}\mathbb{A}_{t}^{\dagger})\Gamma\Delta_{t}\Gamma\right)+\mathrm{tr}\left(\Delta_{t}\Gamma(\mathbb{A}_{t}\Delta_{t}+\Delta_{t}\mathbb{A}_{t}^{\dagger})\Gamma\right)\\
&\quad+\mathrm{tr}\left((\mathbb{D}_{t}\Sigma+\Sigma\mathbb{D}_{t}^{\dagger})\Gamma\Delta_t\Gamma\right)+\mathrm{tr}\left(\Delta_t\Gamma(\mathbb{D}_{t}\Sigma+\Sigma\mathbb{D}_{t}^{\dagger})\Gamma\right)\nonumber\\
 & =2\mathrm{tr}\left(\Delta_{t}\Gamma\Delta_{t}(\Gamma\mathbb{A}_{t}+\mathbb{A}_{t}^{\dagger}\Gamma)\right)+2\mathrm{tr}\left(\Delta_t\Gamma(\mathbb{D}_{t}\Sigma+\Sigma\mathbb{D}_{t}^\dagger)\Gamma\right)\;.\label{eq:dzdz1}
\end{align}
Moreover, by Cauchy-Schwarz inequality,
\begin{equation}\label{eq:cauchy schwarz frobenius}
\mathrm{tr}\left(\Delta_t\Gamma(\mathbb{D}_{t}\Sigma+\Sigma\mathbb{D}_{t}^\dagger)\Gamma\right)\leq||\Delta_t||_F||\Gamma(\mathbb{D}_{t}\Sigma
+\Sigma\mathbb{D}_{t}^\dagger)\Gamma||_F\;.
\end{equation}
Denote
by $c_{0}>0$ the smallest eigenvalue of the positive definite symmetric
matrix $\Gamma$. Then since $\Delta_{t}^{2}$ and $\Delta_{t}\Gamma\Delta_{t}$
are positive semi-definite matrices and $\Gamma$ is symmetric positive
definite we obtain from (\ref{eq:VNT2}) that
\begin{equation}
\mathrm{tr}(\Delta_{t}\Gamma\Delta_{t}\Gamma)\geq c_{0}\mathrm{tr}(\Delta_{t}\Gamma\Delta_{t})=c_{0}\mathrm{tr}(\Gamma\Delta_{t}^{2})\geq c_{0}^{2}\mathrm{tr}(\Delta_{t}^{2})=c_0^2||\Delta_t||_F^2\;.\label{eq:dzdz4}
\end{equation}
Furthermore, using the submultiplicativity in~\eqref{eq:submul} and Theorem~\ref{thm:stability} along with the fact that $F$ is $\mathcal{C}^2$ on the set $\mathcal{H}_{r}$, we deduce by Remark~\ref{rem:Thmstat} that there exists a constant $C(r)>0$ such that $$||\Gamma(\mathbb{D}_{t}\Sigma+\Sigma\mathbb{D}_{t}^\dagger)\Gamma||_F\leq C(r)\mathrm{e}^{-\lambda t}.$$
As a result, reinjecting into~\eqref{eq:cauchy schwarz frobenius},
\begin{equation}\label{eq:upperbound remaining term}
\mathrm{tr}\left(\Delta_t\Gamma(\mathbb{D}_{t}\Sigma+\Sigma\mathbb{D}_{t}^\dagger)\Gamma\right)\leq C(r)\mathrm{e}^{-\lambda t}\sqrt{\mathrm{tr}\left(\Delta_{t}\Gamma\Delta_{t}\Gamma\right)}\;.
\end{equation}

Additionally, by~\eqref{eq:ricattiGamma}
\begin{equation}
\mathrm{tr}\left(\Delta_{t}\Gamma\Delta_{t}(\Gamma\mathbb{A}_{t}+\mathbb{A}_{t}^{\dagger}\Gamma)\right)=\mathrm{tr}\left(\Delta_{t}\Gamma\Delta_{t}(\Gamma\mathbb{D}_{t}+\mathbb{D}_{t}^{\dagger}\Gamma)\right)-\mathrm{tr}(\Delta_{t}\Gamma\Delta_{t})\;.\label{eq:dzdz2}
\end{equation}
Since $\Gamma\mathbb{D}_{t}+\mathbb{D}_{t}^{\dagger}\Gamma$ is a
symmetric matrix and since $\Delta_{t}\Gamma\Delta_{t}$ is a positive
semi-definite symmetric matrix, by (\ref{eq:VNT1}), we have
\[
\mathrm{tr}\left(\Delta_{t}\Gamma\Delta_{t}(\Gamma\mathbb{D}_{t}+\mathbb{D}_{t}^{\dagger}\Gamma)\right)\le\mathrm{tr}(\Delta_{t}\Gamma\Delta_{t})\left\Vert \Gamma\mathbb{D}_{t}+\mathbb{D}_{t}^{\dagger}\Gamma\right\Vert \;.
\]
Note that we verified in (\ref{eq:A(qT(x))}) that $\left\Vert \Gamma\mathbb{D}_{t}+\mathbb{D}_{t}^{\dagger}\Gamma\right\Vert \le\frac{1}{2}$
for $t\ge T(x)$ and therefore by (\ref{eq:dzdz1}),~\eqref{eq:upperbound remaining term}, (\ref{eq:dzdz2}),
and applying (\ref{eq:VNT1}) we get
\begin{align*}
\frac{\mathrm{d}}{\mathrm{d}t}\mathrm{tr}(\Delta_{t}\Gamma\Delta_{t}\Gamma)&\le-\frac{1}{2}\mathrm{tr}(\Delta_{t}\Gamma\Delta_{t})+C(r)\mathrm{e}^{-\lambda t}\sqrt{\mathrm{tr}\left(\Delta_{t}\Gamma\Delta_{t}\Gamma\right)}\\
&\le-\frac{1}{2\Vert\Gamma\Vert}\mathrm{tr}(\Delta_{t}\Gamma\Delta_{t}\Gamma)+C(r)\mathrm{e}^{-\lambda t}\sqrt{\mathrm{tr}\left(\Delta_{t}\Gamma\Delta_{t}\Gamma\right)}\\
&\leq -c_1\mathrm{tr}(\Delta_{t}\Gamma\Delta_{t}\Gamma)+c_2\mathrm{e}^{-\lambda t}
\end{align*}
for some constants $c_1=c_1(r),c_2=c_2(r)>0$ and for $t\ge T(x)$. Accordingly, by integrating $\mathrm{tr}(\Delta_{t}\Gamma\Delta_{t}\Gamma)\mathrm{e}^{c_1t}$ for $t\ge T(x)$ we get
$$\mathrm{tr}(\Delta_{t}\Gamma\Delta_{t}\Gamma)\mathrm{e}^{c_1t}\leq\mathrm{tr}(\Delta_{T(x)}\Gamma\Delta_{T(x)}
\Gamma)\mathrm{e}^{c_1T(x)}+c_2\int_{T(x)}^t\mathrm{e}^{(c_1-\lambda)s}\mathrm{d}s\;.$$
Therefore,
\begin{equation}\label{eq:dzdz3}
    \mathrm{tr}(\Delta_{t}\Gamma\Delta_{t}\Gamma)\leq\mathrm{tr}(\Delta_{T(x)}\Gamma\Delta_{T(x)}\Gamma)\mathrm{e}^{-c_1(t-T(x))}+\frac{c_2}{c_1-\lambda}\mathrm{e}^{-\lambda t}\;,
\end{equation}
where the second inequality uses (\ref{eq:VNT1}). Next
Combining (\ref{eq:dzdz3}) and (\ref{eq:dzdz4}) proves that
\begin{equation}
\mathrm{tr}(\Delta_{t}^{2})\le C(r)e^{-\alpha t}\label{eq:dzdz5}
\end{equation}
for all $t\ge T(x)$ for some $\alpha=\alpha(r)>0$. On the other hand, by Lemma \ref{lem:mom_YT(x)},
there exists a constant $C(r)>0$ such that
\begin{equation}
\sup_{x\in\mathcal{D}_{r}}\sup_{t\in[0,\,T(x)]}\mathrm{tr}(\Delta_{t}^{2})\le C(r)\;.\label{eq:dzdz6}
\end{equation}
Thus, the proof is completed by (\ref{eq:dzdz5}), (\ref{eq:dzdz6}),
and the definition of $\Delta_{t}$.
\end{proof}

\subsection{Small-time behavior of the covariance matrix}

In order to get the small-time asymptotics of the inverse $\Sigma_{t}(x)^{-1}$
we investigate the behavior of $\Sigma_{t}(x)$ in the regime $t\rightarrow0$.
This part of the proof is interestingly much more involved than its
overdamped analogous studied in~\cite{cut-off_overdamped}, since
$\left.\frac{\mathrm{d}\Sigma_{t}(x)}{\mathrm{d}t}\right|_{t=0}=\mathbb{J}$
is not invertible. To resolve this issue, we go up to the third order
expansion of $\Sigma_{t}$ around $t\simeq0$.
\begin{lem}
\label{lem_shor_Sigma0}For all $r>0$, there exists $C(r)>0$ such
that
\begin{equation}
\sup_{x\in\mathcal{D}_{r}}\left\Vert \Sigma_{t}(x)-\widehat{\Sigma}_{t}(x)\right\Vert _{\textup{F}}\le C(r)t^{4}\label{eq:Sigmaappr}
\end{equation}
for all small enough $t>0$, where
\begin{equation}
\widehat{\Sigma}_{t}(x)=\begin{pmatrix}\frac{t^{3}}{3}\mathbb{I}_{d} & \left(\frac{t^{2}}{2}-\gamma\frac{t^{3}}{2}\right)\mathbb{I}_{d}\\
\left(\frac{t^{2}}{2}-\gamma\frac{t^{3}}{2}\right)\mathbb{I}_{d} & \left(t-\gamma t^{2}+\frac{2}{3}\gamma^{2}t^{3}\right)\mathbb{I}_{d}-\frac{t^{3}}{6}(\textup{D}F(x)+\textup{D}F(x)^{\dagger})
\end{pmatrix}\;.\label{eq:Sigma_hat}
\end{equation}
\end{lem}

\begin{proof}
We fix $x\in\mathcal{D}_{r}$ and ignore the dependency on $x$ in
the notation throughout this proof. By Lemma \ref{lem:Sigma_t},
\[
\frac{\mathrm{d}^{2}\Sigma_{t}}{\mathrm{d}t^{2}}=\frac{\mathrm{d}\mathbb{A}_{t}}{\mathrm{d}t}\Sigma_{t}+\Sigma_{t}\frac{\mathrm{d}\mathbb{A}_{t}^{\dagger}}{\mathrm{d}t}+\mathbb{A}_{t}(\mathbb{J}+\mathbb{A}_{t}\Sigma_{t}+\Sigma_{t}\mathbb{A}_{t}^{\dagger})+(\mathbb{J}+\mathbb{A}_{t}\Sigma_{t}+\Sigma_{t}\mathbb{A}_{t}^{\dagger})\mathbb{A}_{t}^{\dagger}\;,
\]
and consequently (since $\Sigma_{0}=\mathbb{O}_{2d}$), we get
\[
\left.\frac{\mathrm{d}^{2}\Sigma_{t}}{\mathrm{d}t^{2}}\right|_{t=0}=\mathbb{A}\mathbb{J}+\mathbb{J}\mathbb{A}^{\dagger}=\begin{pmatrix}0 & \mathbb{I}_{d}\\
\mathbb{I}_{d} & -2\gamma\mathbb{I}_{d}
\end{pmatrix}\;.
\]
Continuing the computation, we can also obtain
\[
\left.\frac{\mathrm{d}^{3}\Sigma_{t}}{\mathrm{d}t^{3}}\right|_{t=0}=\begin{pmatrix}2\mathbb{I}_{d} & -3\gamma\mathbb{I}_{d}\\
-3\gamma\mathbb{I}_{d} & -(\textup{D}F(x)+\textup{D}F(x)^{\dagger})+4\gamma^{2}\mathbb{I}_{d}
\end{pmatrix}\;.
\]
Then, note that the matrix $\widehat{\Sigma}_{t}=\widehat{\Sigma}_{t}(x)$
defined in (\ref{eq:Sigma_hat}) stands for the third order Taylor
expansion of $\Sigma_{t}$, i.e.,
\[
\widehat{\Sigma}_{t}=t\left.\frac{\mathrm{d}\Sigma_{t}}{\mathrm{d}t}\right|_{t=0}+\frac{t^{2}}{2}\left.\frac{\mathrm{d}^{2}\Sigma_{t}}{\mathrm{d}t^{2}}\right|_{t=0}+\frac{t^{3}}{6}\left.\frac{\mathrm{d}^{3}\Sigma_{t}}{\mathrm{d}t^{3}}\right|_{t=0}\;.
\]
From now on we denote by $\mathbb{O}(t,\,x)$ a $2d\times2d$ matrix
which is continuous in $(t,\,x)\in[0,\,\infty)\times\mathbb{R}^{2d}$.
Then, by the Taylor theorem, we can write
\begin{equation}
\Sigma_{t}(x)-\widehat{\Sigma}_{t}(x)=t^{4}\mathbb{O}(t,\,x)\label{eq:difSig}
\end{equation}
and we can conclude that (\ref{eq:Sigmaappr}) holds since $F\in\mathcal{C}^{3}(\mathbb{R}^{d},\,\mathbb{R}^{d})$
by Assumption~\ref{ass:main}.
\end{proof}
Next we get a bound on the norm of $\Sigma_{t}^{-1/2}$ for small
times.
\begin{lem}
\label{lem:short_Sigma}For all $r>0$, there exists $C(r)>0$ such
that
\[
\sup_{x\in\mathcal{D}_{r}}\left\Vert \Sigma_{t}(x)^{-1/2}\right\Vert _{\textup{F}}\le\frac{C(r)}{t^{3/2}}
\]
for all small enough $t>0$.
\end{lem}

\begin{proof}
For $t>0$, define a $2d\times2d$ matrix $\Xi_{t}(x)$ by
\[
\Xi_{t}(x):=\frac{12}{t^{3}}\begin{pmatrix}\left(1-\gamma t+\frac{2}{3}\gamma^{2}t^{2}\right)\mathbb{I}_{d}-\frac{t^{2}}{6}(\textup{D}F(x)+\textup{D}F(x)^{\dagger}) & -\left(\frac{t}{2}-\gamma\frac{t^{2}}{2}\right)\mathbb{I}_{d}\\
-\left(\frac{t}{2}-\gamma\frac{t^{2}}{2}\right)\mathbb{I}_{d} & \frac{t^{2}}{3}\mathbb{I}_{d}
\end{pmatrix}\;,
\]
so that we can write
\begin{equation}
\Xi_{t}(x)=\frac{12}{t^{3}}\left[\mathbb{U}+t\mathbb{O}(t,\,x)\right]\;\;\;\;\text{where }\mathbb{U}=\begin{pmatrix}\mathbb{I}_{d} & \mathbb{O}_{d}\\
\mathbb{O}_{d} & \mathbb{O}_{d}
\end{pmatrix}\;.\label{eq:Xi_t}
\end{equation}
A direct computation shows that
\[
{\widehat{\Sigma}}_{t}(x) \,\Xi_{t}(x)=\mathbb{I}_{2d}+t\mathbb{O}(t,\,x)\;,
\]
and therefore by (\ref{eq:difSig}) and (\ref{eq:Xi_t}), we get
\[
{\Sigma}_{t}(x)\,\Xi_{t}(x)=\mathbb{I}_{2d}+t\mathbb{O}(t,\,x)\;.
\]
Since $\left[\mathbb{I}_{2d}+t\mathbb{O}(t,\,x)\right]^{-1}=\mathbb{I}_{2d}+t\mathbb{O}(t,\,x)$
for small enough $t\ge0$, we get from (\ref{eq:Xi_t}) that
\[
\Sigma_{t}(x)^{-1}=\Xi_{t}(x)\big[\mathbb{I}_{2d}+t\mathbb{O}(t,\,x)\big]^{-1}=\frac{12}{t^{3}}\big[\mathbb{U}+t\mathbb{O}(t,\,x)\big]
\]
and hence
\[
\Sigma_{t}(x)^{-1/2}=\frac{\sqrt{12}}{t^{3/2}}\big[\mathbb{U}+t\mathbb{O}(t,\,x)\big]\;.
\]
This completes the proof.
\end{proof}

\subsection{Total variation distance between Gaussian processes}

In this subsection, we provide estimates on the total variation distances
between $(Z_{t}^{\epsilon}(x))_{t\ge0}$ and $(Z_{t}^{\epsilon}(y))_{t\ge0}$ in the long
time asymptotics. For $v\in\mathbb{R}^{d}$ and a $d\times d$ symmetric
positive definite matrix $\mathbb{K}$, denote by $\mathcal{N}(v,\,\mathbb{K})$
the Gaussian distribution centered at $v$ with covariance matrix $\mathbb{K}$.
Before stating and proving Lemma~\ref{lem:tvZ_long}, we recall
serveral lemma from \cite{cut-off_overdamped}.
\begin{lem}
\label{lem:cut-off App 1}Let $x,y\in\mathbb{R}^{2d}$ and let $\mathbf{\mathbb{S}},\mathbf{\mathbb{T}}\in\mathbb{R}^{2d\times2d}$ be
two symmetric positive definite matrices. Let $c\neq0$ be a fixed constant. Then, we have 

\begin{enumerate}
    \item $\mathrm{d}_{\mathrm{TV}}(\mathcal{N}(cx,c^{2}\mathbb{S}),\mathcal{N}(cy,c^{2}\mathbb{T}))=\mathrm{d}_{\mathrm{TV}}(\mathcal{N}(x,\mathbb{S}),\mathcal{N}(y,\mathbb{T}))$.
    \item $\mathrm{d}_{\mathrm{TV}}(\mathcal{N}(x,\mathbb{S}),\mathcal{N}(y,\mathbb{T}))=\mathrm{d}_{\mathrm{TV}}(\mathcal{N}(x-y,\mathbb{S}),\mathcal{N}(0,\mathbb{T}))$.
    \item $\mathrm{d}_{\mathrm{TV}}(\mathcal{N}(x,\mathbb{S}),\mathcal{N}(y,\mathbb{S}))=\mathrm{d}_{\mathrm{TV}}(\mathcal{N}(\mathbb{S}^{-1/2}x,\mathbb{I}_{2d}),\mathcal{N}(\mathbb{S}^{-1/2}y,\mathbb{I}_{2d}))$.
    \item $\mathrm{d}_{\mathrm{TV}}(\mathcal{N}(0,\mathbb{S}),\mathcal{N}(0,\mathbb{T}))=\mathrm{d}_{\mathrm{TV}}(\mathcal{N}(0,\mathbb{T}^{-1/2}\mathbb{ST}^{-1/2}),\mathcal{N}(0,\mathbb{I}_{2d}))$.
\end{enumerate}
\end{lem}

\begin{proof}
We refer to \cite[Lemma A.1]{cut-off_overdamped} for the proof. 
\end{proof}
\begin{lem}
\label{lem:cut-off App 2}For $x\in\mathbb{R}^{2d}$,
\[
\mathrm{d}_{\mathrm{TV}}(\mathcal{N}(x,\mathbb{I}_{2d}),\mathcal{N}(0,\mathbb{I}_{2d}))=\sqrt{\frac{2}{\pi}}\int_{0}^{|x|/2}\mathrm{e}^{-\frac{s^{2}}{2}}\mathrm{d}s\le\frac{1}{\sqrt{2\pi}}|x|.
\]
\end{lem}

\begin{proof}
We refer to \cite[Lemma A.2]{cut-off_overdamped} for the proof.
\end{proof}

Using these properties, let us now derive suitable upper-bounds on $\mathrm{d}_{\mathrm{TV}}\left(Z_{t}^{\epsilon}(x),\,Z_{t}^{\epsilon}(y)\right)$ which will be used later on.

\begin{lem}
\label{lem:tvZ_long}For all $r>0$, there exist constants $c_{0}(r)>0$
and $C(r)>0$ such that
\[
\sup_{x,\,y\in\mathcal{D}_{r}}\mathrm{d}_{\mathrm{TV}}\left(Z_{t}^{\epsilon}(x),\,Z_{t}^{\epsilon}(y)\right)\le\frac{C(r)}{\epsilon^{1/2}}e^{-c_{0}t}\;.
\]
\end{lem}

\begin{proof}
Fix $r>0$ and $x,\,y\in\mathcal{D}_{r}$. Recall from the definition
(\ref{eq:Z_t^eps}) that we have
\begin{equation}
Z_{t}^{\epsilon}(x)\sim\mathcal{N}(X_{t}(x),\,2\epsilon\Sigma_{t}(x))\;.\label{eq:distZ}
\end{equation}
Hence, by the triangle inequality, we get
\begin{align}
\mathrm{d}_{\mathrm{TV}}\left(Z_{t}^{\epsilon}(x),\,Z_{t}^{\epsilon}(y)\right)\le\; & \mathrm{d}_{\mathrm{TV}}\left(\mathcal{N}(X_{t}(x),\,2\epsilon\Sigma_{t}(x)),\,\mathcal{N}(X_{t}(x),\,2\epsilon\Sigma)\right)\nonumber \\
 & +\mathrm{d}_{\mathrm{TV}}\left(\mathcal{N}(X_{t}(y),\,2\epsilon\Sigma),\,\mathcal{N}(X_{t}(y),\,2\epsilon\Sigma_{t}(y))\right)\nonumber \\
 & +\mathrm{d}_{\mathrm{TV}}\left(\mathcal{N}(X_{t}(x),\,2\epsilon\Sigma),\,\mathcal{N}(X_{t}(y),\,2\epsilon\Sigma)\right)\;.\label{eq:tvZ}
\end{align}
By Lemma~\ref{lem:cut-off App 1}-(1) and (2), the first distance in
the right-hand side equals to
\[
\mathrm{d}_{\mathrm{TV}}\left(\mathcal{N}(0,\,2\epsilon\Sigma_{t}(x)),\,\mathcal{N}(0,\,2\epsilon\Sigma)\right)=\mathrm{d}_{\mathrm{TV}}\left(\mathcal{N}(0,\,\Sigma_{t}(x)),\,\mathcal{N}(0,\,\Sigma)\right)\;.
\]
By \cite[Theorem 1.1 and display (2)]{DeMeRe}, we have
\begin{align}
\mathrm{d}_{\mathrm{TV}}\left(\mathcal{N}(0,\,\Sigma_{t}(x)),\,\mathcal{N}(0,\,\Sigma)\right)
&\le\frac{3}{2}\left\Vert \Sigma^{-1/2}\Sigma_{t}(x)\Sigma^{-1/2}-\mathbb{I}_{2d}\right\Vert _{\textup{F}}  \nonumber \\
&\le\frac{3}{2}\left\Vert \Sigma^{-1/2}\right\Vert _{\textup{F}}^{2}\left\Vert \Sigma_{t}(x)-\Sigma\right\Vert _{\textup{F}}\label{eq:tvZ0}
\end{align}
where the second inequality follows from the submultiplicativity of
Frobenius norm, i.e., (\ref{eq:submul}). Hence, by Lemma \ref{lem:long_Sigma},
we get
\begin{equation}
\mathrm{d}_{\mathrm{TV}}\left(\mathcal{N}(X_{t}(x),\,2\epsilon\Sigma_{t}(x)),\,\mathcal{N}(X_{t}(x),\,2\epsilon\Sigma)\right)\le C(r)e^{-\alpha(r)t}\;.\label{eq:tvZ1}
\end{equation}
Note that the second distance at the right-hand side of (\ref{eq:tvZ})
has the same bound. Now we turn to the last distance in (\ref{eq:tvZ}).
By Lemma~\ref{lem:cut-off App 1}-(1), (2) and (3), we get
\begin{align*}
\mathrm{d}_{\mathrm{TV}}\left(\mathcal{N}(X_{t}(x),\,2\epsilon\Sigma),\,\mathcal{N}(X_{t}(y),\,2\epsilon\Sigma)\right) & =\mathrm{d}_{\mathrm{TV}}\left(\mathcal{N}\left(\Sigma^{-1/2}\frac{X_{t}(x)-X_{t}(y)}{\sqrt{2\epsilon}},\,\mathbb{I}_{2d}\right),\,\mathcal{N}(0,\,\mathbb{I}_{2d})\right)\;.
\end{align*}
Hence, by Lemma~\ref{lem:cut-off App 2}, we have
\[
\mathrm{d}_{\mathrm{TV}}\left(\mathcal{N}(X_{t}(x),\,2\epsilon\Sigma),\,\mathcal{N}(X_{t}(y),\,2\epsilon\Sigma)\right)\le\frac{\left\Vert \Sigma^{-1/2}\right\Vert }{\sqrt{4\pi\epsilon}}|X_{t}(x)-X_{t}(y)|\;.
\]
Therefore, by Theorem \ref{thm:stability}, there exists a constant
$C(r)>0$ such that
\begin{equation}
\mathrm{d}_{\mathrm{TV}}\left(\mathcal{N}(X_{t}(x),\,2\epsilon\Sigma),\,\mathcal{N}(X_{t}(y),\,2\epsilon\Sigma)\right)\le\frac{C(r)}{\epsilon^{1/2}}e^{-\lambda t}\;.\label{eq:tvZ2}
\end{equation}
Summing up (\ref{eq:tvZ1}) and (\ref{eq:tvZ2}) complete the proof.
\end{proof}
Next we provide a bound which we shall use for the study of the small-time
asymptotics.
\begin{lem}
\label{lem:tvZ_short}For all $r>0$, there exist constants $c_{1}(r),\,c_{2}(r)>0$
such that
\[
\sup_{u,\,u'\in\mathcal{D}_{r}}\mathrm{d}_{\mathrm{TV}}\left(Z_{t}^{\epsilon}(u),\,Z_{t}^{\epsilon}(u')\right)\leq c_{1}(r)\left[\left(\frac{e^{c_{2}(r)t}}{t^{3/2}\epsilon^{1/2}}+1\right)|u-u'|+t\right]
\]
for all small enough $t>0$.
\end{lem}

\begin{proof}
We fix $r>0$ and $u,\,u'\in\mathcal{D}_{r}$. By (\ref{eq:distZ})
and the triangle inequality
\begin{align}
  \mathrm{d}_{\mathrm{TV}}\left(Z_{t}^{\epsilon}(u),\,Z_{t}^{\epsilon}(u')\right)\nonumber 
  \le &\;\mathrm{d}_{\mathrm{TV}}\left(\mathcal{N}(X_{t}(u),\,2\epsilon\Sigma_{t}(u)),\,\mathcal{N}(X_{t}(u'),\,2\epsilon\Sigma_{t}(u))\right) \\
  &+\mathrm{d}_{\mathrm{TV}}\left(\mathcal{N}(X_{t}(u'),\,2\epsilon\Sigma_{t}(u)),\,\mathcal{N}(X_{t}(u'),\,2\epsilon\Sigma_{t}(u'))\right)\;.\label{eq:dtv0}
\end{align}
By Lemma~\ref{lem:cut-off App 1}-(2), (3), the first term at the
right-hand side can be written as
\[
\mathrm{d}_{\mathrm{TV}}\left(\mathcal{N}\left(\Sigma_{t}(u)^{-1/2}\frac{X_{t}(u)-X_{t}(u')}{\sqrt{2\epsilon}},\,\mathbb{I}_{2d}\right),\,\mathcal{N}(0,\,\mathbb{I}_{2d})\right)\;.
\]
By Corollary \ref{cor:contraction}, Lemma \ref{lem:short_Sigma}, and Lemma~\ref{lem:cut-off App 2} this is bounded from above by
\begin{equation}
\frac{1}{\sqrt{4\pi\epsilon}}\left\Vert \Sigma_{t}(u)^{-1/2}\right\Vert \left|X_{t}(u)-X_{t}(u')\right|\le c_{1}(r)\frac{e^{c_{2}(r)t}}{t^{3/2}\epsilon^{1/2}}|u-u'|\label{eq:dtv2}
\end{equation}
for some constants $c_{1}(r),\,c_{2}(r)>0$.

On the other hand, by Lemma~\ref{lem:cut-off App 1}-(1), (3), the
second term at the right-hand side of (\ref{eq:dtv0}) equals to
\[
\mathrm{d}_{\mathrm{TV}}\left(\mathcal{N}(0,\,\Sigma_{t}(u)),\,\mathcal{N}(0,\,\Sigma_{t}(u'))\right)\;.
\]
By the same argument with (\ref{eq:tvZ0}) and Lemma \ref{lem:short_Sigma},
this distance is bounded from above by
\begin{equation}
\frac{3}{2}\left\Vert \Sigma_{t}(u)^{-1/2}\right\Vert _{\textup{F}}^{2}\left\Vert \Sigma_{t}(u')-\Sigma_{t}(u)\right\Vert _{\textup{F}}\le\frac{C(r)}{t^{3}}\left\Vert \Sigma_{t}(u')-\Sigma_{t}(u)\right\Vert _{\textup{F}}\;.\label{eq:dtv1}
\end{equation}
By Lemma \ref{lem_shor_Sigma0} and mean-value theorem along with
explicit formular (\ref{eq:Sigma_hat}) for $\widehat{\Sigma}_{t}$,
we get
\begin{align*}
\left\Vert \Sigma_{t}(u')-\Sigma_{t}(u)\right\Vert _{\textup{F}} & \le C(r)t^{4}+\left\Vert \widehat{\Sigma}_{t}(u')-\widehat{\Sigma}_{t}(u)\right\Vert _{\textup{F}}\\
 & =C(r)t^{4}+Ct^{3}\left\Vert \textup{D}F(u')-\textup{D}F(u)\right\Vert _{\textup{F}}\le C(r)\left(t^{4}+t^{3}|u'-u|\right)\;.
\end{align*}
Inserting this to (\ref{eq:dtv1}), we can conclude that the second
term at the right-hand side of (\ref{eq:dtv0}) is bounded by $C(r)\left(t+|u'-u|\right)$.
This bound along with (\ref{eq:dtv0}) and (\ref{eq:dtv2}) complete
the proof.
\end{proof}

\section{\label{sec7}Estimate of Total Variation Distance}

In this section, we estimate the total variation distance $\mathrm{d}_{\mathrm{TV}}\left(X_{t}^{\epsilon}(x),\,\mu^{\epsilon}\right)$
by gathering the results obtained in the previous section. From now
on, we suppose that both Assumptions~\ref{ass:main} and~\ref{ass:DF}
hold.

\subsection{Long-time behavior on zero-noise dynamics}

\label{sec:lem_asym_uld}

Let us start this section with the proof of Lemma \ref{lem:asym_uld} which follows the steps of~\cite[Lemma B.2]{cut-off_overdamped},
since it is one of the crucial tools in the sharp estimate of the
total variation distance.
\begin{proof}[Proof of Lemma \ref{lem:asym_uld}]
By the Hartman-Grobman theorem, there exist neighborhoods $\mathcal{U}$, $\mathcal{V}\subset \mathbb{R}^{2d}$
of the origin in $\mathbb{R}^{2d}$ and a diffeomorphism $\phi:\mathcal{U}\rightarrow\mathcal{V}$
that maps the path $(X_{t}(x))_{t\ge0}$ to the linearized path $(e^{\mathbb{A}t}\phi(x))_{t\ge0}$
for all $x\in\mathcal{U}$ where $\mathbb{A}=-\textup{D}G(0)$) (cf.
(\ref{eq:uld0_2})). 

We first prove the lemma when $x\in\mathcal{U}$. Denote by $-\lambda_{1},\,\dots,\,-\lambda_{N}\in\mathbb{C}$
the eigenvalues of $\mathbb{A}$ so that $\textup{Re}(\lambda_{i})>0$
for all $i\in\llbracket1,\,N\rrbracket$ by Remark \ref{rem:approx}-(3).
Denote by $\{w_{j,\,k}:j\in[1,\,N],\,k\in[1,\,N_{j}]\}$ the Jordan
basis of $\mathbb{A}$, i.e.,
\[
\mathbb{A}w_{j,\,k}=-\lambda_{j}w_{j,\,k}+w_{j,\,k+1}\;\;\;\;;\;j\in[1,\,N],\,k\in[1,\,N_{j}]\;,
\]
where $w_{j,\,N_{j}+1}=0$. Expanding $x$ by using this Jordan basis
in a way that
\begin{equation}
x=\sum_{j=1}^{N}\sum_{k=1}^{N_{j}}x_{j,\,k}w_{j,\,k}\;.\label{eq:expx}
\end{equation}
Here, by the Putzer spectral method (cf. \cite[Proof of Lemma B.1.]{cut-off_overdamped}),
we can write
\[
e^{\mathbb{A}t}x=\sum_{j=1}^{N}\sum_{k=1}^{N_{j}}\sum_{i=1}^{k}\frac{t^{k-i}e^{-\lambda_{j}t}}{(k-i)!}x_{j,\,i}w_{j,\,k}\;.
\]
Then, define
\begin{align*}
\eta & =\eta(x):=\min\left\{ \textup{Re}(\lambda_{j}):x_{j,\,k}\neq0\text{ for some }k\in\llbracket1,\,N_{j}\rrbracket\right\} >0\;,\\
\nu & =\ensuremath{\nu(x):=\max\left\{ N_{j}-k:x_{j,\,k}\neq0\text{ for some }k\in\llbracket1,\,N_{j}\rrbracket\right\} \ge0}\;\text{and}\\
I & =I(x):=\{j\in[1,\,N]:\textrm{Re}(\lambda_{j})=\eta\text{ and }x_{j,\,N_{j}-\nu}\neq0\}\;,
\end{align*}
so that we can write
\begin{equation}
\frac{\mathrm{e}^{\eta t}}{t^{\nu}}e^{\mathbb{A}t}x=\sum_{j\in I}\frac{e^{-(\lambda_{j}-\eta)t}}{\nu!}x_{j,\,N_{j}-\nu}w_{j,\,N_{j}}+\varepsilon_{t}\;,\label{eq:exp-ATx}
\end{equation}
where $\varepsilon_{t}\rightarrow0$ as $t\rightarrow\infty$. Since
$\lambda_{j}-\eta$ is pure imaginary for $j\in I$, by writing
\[
\lambda_{j}-\eta=-i\theta_{j}\;\;\;\text{and}\;\;\;v_{j}=\frac{x_{j,\,N_{j}-\nu}}{\nu!}w_{j,\,N_{j}}\;,
\]
we can deduce from (\ref{eq:exp-ATx}) that
\begin{equation}
\lim_{t\rightarrow\infty}\left|\frac{\mathrm{e}^{\eta t}}{t^{\nu}}e^{\mathbb{A}t}x-\sum_{k\in I}\mathrm{e}^{i\theta_{k}t}v_{k}\right|=0\;.\label{eq:approx2}
\end{equation}
In \cite[Lemma B.2.]{cut-off_overdamped}, using the Hartman-Grobman
theorem mentioned above, it is proven that we can replace $e^{\mathbb{A}t}x$
at (\ref{eq:approx2}) with $X_{t}(x)$. Therefore, we can derive
(\ref{eq:approx1}) with $\tau(x)=0$ in the case $x\in\mathcal{U}$.

For general $x\in\mathbb{R}^{2d}$, define $\tau(x)>0$ as
\[
\tau(x)=\inf\{t>0:\text{For all\;} s\ge t, X_{t}(x)\in\mathcal{U}\}
\]
so that $X_{\tau(x)}(x)\in\mathcal{U}$. Note that $\tau(x)<\infty$
follows from Theorem \ref{thm:stability}. Since $X_{t+\tau(x)}(x)=X_{t}(X_{\tau(x)}(x))$,
we also have (\ref{eq:approx1}) with the constants defined at $X_{\tau(x)}(x)\in\mathcal{U}$
in the previous step.
\end{proof}
\begin{rem}
\label{rem:const}In some special cases, the constants obtained in
the previous lemma can be simplified.
\begin{enumerate}
\item For $x\in\mathcal{U}$ with $x_{j,\,k}\neq0$ for all $j,\,k$ (which
happens almost everywhere in $\mathcal{U}$ with respect to the Lebesgue
measure), $\eta$ is the smallest real part of eigenvalue, and $\nu+1$
is the largest size of Jordan block associated with an eigenvalue
of smallest real part.
\item When $\mathbb{A}$ is diagonalizable, then $\nu=0$.
\item When all eigenvalues of $\mathbb{A}$ are positive real, we have $\theta_{k}=0$
for all $k\in I$ and all vectors $v_{k}$ are real vectors.
\end{enumerate}
\end{rem}

\subsection{Strategy to estimate total-variation distance}

Let $x\in\mathbb{R}^{2d}$ and recall the constants and vectors defined
in Lemma \ref{lem:asym_uld}. The mixing time for the process $(X_{t}^{\epsilon}(x))_{t\ge0}$
is defined as
\[
t_{\mathrm{mix}}^{\epsilon}(x):=\frac{1}{2\eta}\log\left(\frac{1}{2\epsilon}\right)+\frac{\nu}{\eta}\log\log\left(\frac{1}{2\epsilon}\right)+\tau(x)\;.
\]
For $t\ge\tau(x)$ define $v(t,\,x)\in\mathbb{R}^{2d}$ and $D^{\epsilon}(t,\,x)\in[0,\,1]$
as
\begin{align}
v(t,\,x) & :=\frac{(t-\tau(x))^{\nu}}{\mathrm{e}^{\eta(t-\tau(x))}}\Sigma^{-1/2}\sum_{k=1}^{m}\mathrm{e}^{i\theta_{k}(t-\tau(x))}v_{k}\;,\label{eq:def_v(t)}\\
D^{\epsilon}(t,\,x) & :=\mathrm{d}_{\mathrm{TV}}\left(\mathcal{N}\left(\frac{1}{\sqrt{2\epsilon}}v(t,\,x),\,\mathbb{I}_{2d}\right),\,\mathcal{N}(0,\,\mathbb{I}_{2d})\right)\;.\label{eq:def_D(t)}
\end{align}
Our main result regarding the estimate of the total variation distance
to the equilibrium is the following theorem.
\begin{thm}
\label{thm:tv}For all $x\in\mathbb{R}^{2d}$ and $w\in\mathbb{R}$,
it holds that
\[
\lim_{\epsilon\rightarrow0}\left|\mathrm{d}_{\mathrm{TV}}\left(X_{t_{\mathrm{mix}}^{\epsilon}(x)+w}^{\epsilon}(x),\,\mu^{\epsilon}\right)-D^{\epsilon}(t_{\textup{mix}}^{\epsilon}(x)+w,\,x)\right|=0\;.
\]
\end{thm}
Assuming Theorem~\ref{thm:tv} holds, we are now able to conclude the proofs of Theorems \ref{thm:cut-off} and \ref{thm:profilecut-off} which are similar to the proofs provided in~\cite[Corollaries 2.7, 2.9]{cut-off_overdamped} but for the sake of completeness we expose below the main ideas.
\begin{proof}[Proof of Theorems \ref{thm:cut-off} and \ref{thm:profilecut-off}] 

By Lemma~\ref{lem:cut-off App 2},
$$D^{\epsilon}(t,\,x)=\sqrt{\frac{2}{\pi}}\int_0^{|v(t,\,x)|/\sqrt{2\epsilon}}\mathrm{e}^{-x^2/2}\mathrm{d}x\;.$$
Furthermore, by construction of $t_{\textup{mix}}^{\epsilon}(x)$,
$$\lim_{\epsilon\rightarrow0}\frac{1}{\sqrt{2\epsilon}}\frac{(t_{\textup{mix}}^{\epsilon}(x)+w-\tau(x))^{\nu}}{\mathrm{e}^{\eta(t_{\textup{mix}}^{\epsilon}(x)+w-\tau(x))}}=\frac{1}{(2\eta(x))^{\nu(x)}}\mathrm{e}^{-\eta(x) w}\;.$$
Similar arguments as in~\cite[Corollary 2.7]{cut-off_overdamped} then ensure the existence of constants $C_1(x),\,C_2(x)>~0$ such that
$$C_1(x)\mathrm{e}^{-\eta(x) w}\leq\liminf_{\epsilon\rightarrow0}\frac{|v(t_{\textup{mix}}^{\epsilon}(x)+w,\,x)|}{\sqrt{2\epsilon}}\leq\limsup_{\epsilon\rightarrow0}\frac{|v(t_{\textup{mix}}^{\epsilon}(x)+w,\,x)|}{\sqrt{2\epsilon}}\leq C_2(x)\mathrm{e}^{-\eta(x)w}\;.$$
As a result,
$$\lim_{w\rightarrow\infty}\liminf_{\epsilon\rightarrow0}D^{\epsilon}(t_{\textup{mix}}^{\epsilon}(x)-w,\,x)=1\,,$$
and
$$\lim_{w\rightarrow\infty}\limsup_{\epsilon\rightarrow0}D^{\epsilon}(t_{\textup{mix}}^{\epsilon}(x)+w,\,x)=0\;.$$
Combining the behavior of $D^{\epsilon}(t_{\textup{mix}}^{\epsilon}(x)+w,\,x)$ when $\epsilon\rightarrow0$ along with Theorem~\ref{thm:tv} then immediately yields Theorem~\ref{thm:cut-off}. The same argument from~\cite[Corollary 2.9]{cut-off_overdamped} also yields Theorem~\ref{thm:profilecut-off}.
\end{proof}

The basic strategy to prove Theorem \ref{thm:tv} is based on the
coupling argument developed in \cite{cut-off_overdamped}, but due
to the degeneracy and lack of global contraction, we need to considerably
refine the argument. Our scheme of proof allows us to tackle the degeneracy
in~\eqref{eq:uld} and could be extended to other processes with
similar degeneracy. We explain now the structure of this proof.

The first step is to control the total variation distance between
the process $(X_{t}^{\epsilon}(x))_{t\geq0}$ and its Gaussian approximation
$(Z_{t}^{\epsilon}(x))_{t\geq0}$ using the moment estimates obtained
in Proposition \ref{prop:main_apprx}.
\begin{prop}
\label{prop:tvm1}Let $(t^{\epsilon})_{\epsilon>0}$ be a sequence
of positive real numbers such that
\begin{equation}
\alpha_{0}\log\frac{1}{\epsilon}\le t^{\epsilon}\le\frac{1}{\epsilon^{\theta}}\label{eq:tvm1}
\end{equation}
for some constant $\alpha_{0}>0$, where $\theta$ is the constant
defined in Notation \ref{not:T(x)}. Then, for any $r>0$, we have
\begin{equation}
\limsup_{\epsilon\rightarrow0}\sup_{x\in\mathcal{D}_{r}}\mathrm{d}_{\mathrm{TV}}\left(X_{t^{\epsilon}}^{\epsilon}(x),\,Z_{t^{\epsilon}}^{\epsilon}(x)\right)=0\;.\label{eq:tvm1-0}
\end{equation}
\end{prop}

This proposition is proven in Section \ref{sec73}. The next step
is to estimate the distribution of the Gaussian process $(Z_{t}^{\epsilon}(x))_{t\geq0}$.
\begin{prop}
\label{prop:tvm2}Let $x\in\mathbb{R}^{2d}$ and let $(t^{\epsilon})_{\epsilon>0}$
be a sequence of positive reals such that
\begin{equation}
t^{\epsilon}\ge t_{\mathrm{mix}}^{\epsilon}(x)-c\label{eq:tvm2}
\end{equation}
for some constant $c>0$ for all small enough $\epsilon$. Then, we
have
\[
\lim_{\epsilon\rightarrow0}\left|\mathrm{d}_{\mathrm{TV}}\left(Z_{t^{\epsilon}}^{\epsilon}(x),\,\mathcal{N}(0,2\epsilon\Sigma)\right)-D^{\epsilon}(t^{\epsilon},\,x)\right|=0\;.
\]
\end{prop}

The remaining part of this article is devoted to the proof of Propositions
\ref{prop:tvm1}, \ref{prop:tvm2}, and Theorems~\ref{thm:asymstat}
and~\ref{thm:tv}. Let us start with the proof of Theorem \ref{thm:tv}
assuming that Propositions \ref{prop:tvm1}, \ref{prop:tvm2}, and
Theorem \ref{thm:asymstat} hold.
\begin{proof}[Proof of Theorem \ref{thm:tv}]
By the triangle inequality, we have
\begin{align*}
 &  \bigg\vert \, \big\vert\mathrm{d}_{\mathrm{TV}}\left(X_{t_{\mathrm{mix}}^{\epsilon}(x)+w}^{\epsilon}(x),\,\mu^{\epsilon}\right)-D^{\epsilon}(t_{\textup{mix}}^{\epsilon}(x)+w,\,x)
 \big\vert\\
 & \;\;\;-\big\vert\mathrm{d}_{\mathrm{TV}}\left(Z_{t_{\mathrm{mix}}^{\epsilon}(x)+w}^{\epsilon}(x),\,\mathcal{N}(0,2\epsilon\Sigma)\right)
 -D^{\epsilon}(t_{\textup{mix}}^{\epsilon}(x)+w,\,x)\big\vert \bigg\vert \\
 & \le\mathrm{d}_{\mathrm{TV}}\left(X_{t_{\mathrm{mix}}^{\epsilon}(x)+w}^{\epsilon}(x),\,Z_{t_{\mathrm{mix}}^{\epsilon}(x)+w}^{\epsilon}(x)\right)+\mathrm{d}_{\mathrm{TV}}\left(\mathcal{N}(0,2\epsilon\Sigma),\,\mu^{\epsilon}\right)\;.
\end{align*}
Therefore, the proof is completed by Propositions \ref{prop:tvm1}
and \ref{prop:tvm2} with $t^{\epsilon}=t_{\mathrm{mix}}^{\epsilon}(x)+w$
(which satisfies (\ref{eq:tvm1}) and (\ref{eq:tvm2})) and Theorem
\ref{thm:asymstat}.
\end{proof}

\subsection{\label{sec73}Proof of Proposition \ref{prop:tvm1}}

We start by proving the following lemma which provides an upper-bound
on the total variation distance between the random variables $X_{t}^{\epsilon}(x)$
and $Z_{t}^{\epsilon}(x)$ for small times $t>0$
\begin{lem}
\label{lem:pinsker}For all $x\in\mathbb{R}^{2d}$, we have that
\[
\mathrm{d}_{\mathrm{TV}}\left(X_{t}^{\epsilon}(x),Z_{t}^{\epsilon}(x)\right)^{2}\le\frac{1}{2\epsilon}\int_{0}^{t}\mathbb{E}\left[\left|F(q_{s}^{\epsilon}(x))-F(q_{s}(x))-\textup{D}F(q_{s}(x))\left(q_{s}^{\epsilon}(x)-q_{s}(x)\right)\right|^{2}\right]\textup{d}s\;.
\]
\end{lem}

\begin{proof}
Write $Z_{t}^{\epsilon}=(\widetilde{q}_{t}^{\epsilon},\,\widetilde{p}_{t}^{\epsilon})$
so that we can rewrite the SDE describing the process $(Z_{t}^{\epsilon})_{t\geq0}$
as
\begin{equation}
\begin{cases}
\textup{d}\widetilde{q}_{t}^{\epsilon}=\widetilde{p}_{t}^{\epsilon}\textup{d}t\\
\textup{d}\widetilde{p}_{t}^{\epsilon}=-(F(q_{t})+\textup{D}F(q_{t})(\widetilde{q}_{t}^{\epsilon}-q_{t})+\gamma\widetilde{p}_{t}^{\epsilon})\textup{d}t+\sqrt{2\epsilon}\,\textup{d}B_{t}\;.
\end{cases}\label{eq:sdeZ}
\end{equation}
Now, denote by $\mathbb{P}_{X}$ and $\mathbb{P}_{Z}$ the law of
the processes $(X_{s}^{\epsilon}(x))_{s\in[0,\,t]}$ and $(Z_{s}^{\epsilon}(x))_{s\in[0,\,t]}$,
respectively. Then, by Pinsker's inequality, we have
\begin{equation}
\mathrm{d}_{\mathrm{TV}}\left(X_{t}^{\epsilon}(x),\,Z_{t}^{\epsilon}(x)\right)^{2}\le-2\mathbb{E}_{\mathbb{P}_{X}}\left[\log\frac{\textup{d}\mathbb{P}_{Z}}{\textup{d}\mathbb{P}_{X}}\right]\label{eq:pinsker}\;.
\end{equation}
By the Girsanov theorem, (\ref{eq:uld}), and (\ref{eq:sdeZ}), we
have
\begin{align*}
\log\frac{\textup{d}\mathbb{P}_{Z}}{\textup{d}\mathbb{P}_{X}}= \;& \frac{1}{\sqrt{2\epsilon}}\int_{0}^{t}\left\langle F(q_{s}^{\epsilon}(x))-F(q_{s}(x))-\textup{D}F(q_{s}(x))\left(q_{s}^{\epsilon}(x)-q_{s}(x)\right),\,\textup{d}B_{s}\right\rangle \\
 & -\frac{1}{4\epsilon}\int_{0}^{t}\left|F(q_{s}^{\epsilon}(x))-F(q_{s}(x))-\textup{D}F(q_{s}(x))\left(q_{s}^{\epsilon}(x)-q_{s}(x)\right)\right|^{2}\textup{d}s\;.
\end{align*}
Inserting this to (\ref{eq:pinsker}), we can complete the proof.
\end{proof}
Our strategy is to first couple the processes $(X_{t}^{\epsilon}(x))_{t\geq0}$
and $(Z_{t}^{\epsilon}(x))_{t\geq0}$ for small times as explained
in Proposition~\ref{prop:tv1}, similarly to what was done in~\cite{cut-off_overdamped}.
To that end, we need to introduce the following small time scales.
Recall the constants $\alpha_{1},\,\alpha_{2}>0$ from Proposition
\ref{prop:main_apprx} and the constant $\alpha_{0}>0$ from the statement
of Proposition \ref{prop:tvm1}. Let us take $\sigma$ small enough
so that
\begin{equation}
0<\sigma<\min\left(\frac{2\alpha_{1}\alpha_{0}}{3},\,\frac{2\alpha_{2}}{3}\right)\;.\label{eq:consig}
\end{equation}
Additionally, let us take $\nu$ small enough from Proposition
\ref{prop:main_apprx} such that
\begin{equation}\label{eq:condition nu}
    0<\nu<\min\left(\frac{\sigma}{2},\alpha_1\alpha_0-\frac{3\sigma}{2}\right).
\end{equation}
Let us now define the time scale $\delta^{\epsilon}$ by $\delta^{\epsilon}:=\epsilon^{\sigma}$.
The following is the first step in the approximation of $X_{t^{\epsilon}}^{\epsilon}(x)$
by $Z_{t^{\epsilon}}^{\epsilon}(x)$ in the total variation distance.
\begin{prop}
\label{prop:tv1}Let $(t^{\epsilon})_{\epsilon>0}$ be a sequence
of positive reals satisfying (\ref{eq:tvm1}). Then, for all $r>0$,
we have
\[
\limsup_{\epsilon\rightarrow0}\sup_{x\in\mathcal{D}_{r}}\mathrm{d}_{\mathrm{TV}}\left(X_{t^{\epsilon}}^{\epsilon}(x),\,Z_{\delta^{\epsilon}}^{\epsilon}(X_{t^{\epsilon}-\delta^{\epsilon}}^{\epsilon}(x))\right)=0.
\]
\end{prop}

\begin{proof}
Let $r_{0}>0$ be small enough so that we have for all $x\in\mathcal{D}_{r_{0}}$,
\begin{equation}
T(x)=0\;\;\;\text{and\;\;\;}H(x)<\frac{1}{16(d+2)\kappa_{0}\rho}\label{eq:cond_r}\;,
\end{equation}
where $T(x)$ is defined in Notation \ref{not:T(x)}, and the constants
$\kappa_{0}$ and $\rho$ are the ones appearing in Lemma \ref{lem_Hquad}
and Assumption \ref{ass:DF}, respectively.

Since $X_{t^{\epsilon}}^{\epsilon}(x)=X_{\delta^{\epsilon}}^{\epsilon}(X_{t^{\epsilon}-\delta^{\epsilon}}^{\epsilon}(x))$,
we have
\begin{align}
 & \mathrm{d}_{\mathrm{TV}}\left(X_{t^{\epsilon}}^{\epsilon}(x),Z_{\delta^{\epsilon}}^{\epsilon}(X_{t^{\epsilon}-\delta^{\epsilon}}^{\epsilon}(x))\right)\nonumber \\
 & \leq\int_{\mathbb{R}^{d}}\mathrm{d}_{\mathrm{TV}}\left(X_{\delta^{\epsilon}}^{\epsilon}(u),Z_{\delta^{\epsilon}}^{\epsilon}(u)\right)\mathbb{P}(X_{t^{\epsilon}-\delta^{\epsilon}}^{\epsilon}(x)\in\mathrm{d}u)\nonumber \\
 & \leq\int_{\mathcal{D}_{r_{0}}}\mathrm{d}_{\mathrm{TV}}\left(X_{\delta^{\epsilon}}^{\epsilon}(u),Z_{\delta^{\epsilon}}^{\epsilon}(u)\right)\mathbb{P}(X_{t^{\epsilon}-\delta^{\epsilon}}^{\epsilon}(x)\in\mathrm{d}u)
 +\mathbb{P}(|X_{t^{\epsilon}-\delta^{\epsilon}}^{\epsilon}(x)|>r_{0})\;.\label{eq:dec_tv1}
\end{align}
By Chebychev inequality and Proposition \ref{prop:mom_est}, we have
\begin{align}
\mathbb{P}(|X_{t^{\epsilon}-\delta^{\epsilon}}^{\epsilon}(x)|>r_{0}) & \leq\frac{\mathbb{E}\left[|X_{t^{\epsilon}-\delta^{\epsilon}}^{\epsilon}(x)|^{2}\right]}{r_{0}^{2}}\leq\frac{\kappa_{0}}{r_{0}^{2}}\left[H(x)\mathrm{e}^{-\lambda(t^{\epsilon}-\delta^{\epsilon})}+\frac{d\epsilon}{\lambda}\right]\;.\label{eq:tv1-1}
\end{align}

Next we fix $u\in\mathcal{D}_{r_{0}}$. Then, by Lemma \ref{lem:pinsker},
we have
\begin{equation}
\mathrm{d}_{\mathrm{TV}}\left(X_{\delta^{\epsilon}}^{\epsilon}(u),\,Z_{\delta^{\epsilon}}^{\epsilon}(u)\right)^{2}\le\frac{1}{\epsilon}\left[A^{\epsilon}(1)+A^{\epsilon}(2)\right]\label{eq:tv2}\;,
\end{equation}
where
\begin{align*}
 & A^{\epsilon}(1)=\int_{0}^{\delta^{\epsilon}}\mathbb{E}\left[\left|F(q_{s}^{\epsilon}(u))-F(q_{s}(u))\right|^{2}\right]ds\;,\\
 & A^{\epsilon}(2)=\left\Vert \textup{D}F(q_{s}(u))\right\Vert \int_{0}^{\delta^{\epsilon}}\mathbb{E}\left[|X_{s}^{\epsilon}(u)-X_{s}(u)|^{2}\right]ds\;.
\end{align*}
Let us first look $A^{\epsilon}(1)$. By the Cauchy-Schwarz inequality
and Assumption \ref{ass:DF}, we have
\begin{align*}
\big|F(q_{s}^{\epsilon}(u))-F(q_{s}(u))\big|^{2} & =\big|\int_{0}^{1}\textup{D}F(q_{s}(u)+\zeta(q_{s}^{\epsilon}(u)-q_{s}(u)))\,(q_{s}^{\epsilon}(u)-q_{s}(u))\mathrm{d}\zeta\bigg|^{2}\\
 & \leq\big|q_{s}^{\epsilon}(u)-q_{s}(u)\big|^{2}\int_{0}^{1}|\textup{D}F(q_{s}(u)+\zeta(q_{s}^{\epsilon}(u)-q_{s}(u)))|^{2}\mathrm{d}\zeta\\
 & \leq C\,\big|q_{s}^{\epsilon}(u)-q_{s}(u)\big|^{2}\int_{0}^{1}\mathrm{e}^{2\rho|q_{s}(u)+\zeta(q_{s}^{\epsilon}(u)-q_{s}(u))|^{2}}\mathrm{d}\zeta\\
 & \leq C\big|q_{s}^{\epsilon}(u)-q_{s}(u)\big|^{2}\mathrm{e}^{4\rho|q_{s}(u)|^{2}}\mathrm{e}^{4\rho|q_{s}^{\epsilon}(u)|^{2}}\;.
\end{align*}
By Theorem \ref{thm:stability}, we have for $u=(q,p)$ $|q_{s}(u)|^{2}\le\kappa\left(|u|^{2}+U(q)\right)$
for all $s\ge0$. Inserting this bound in the exponential above and
noticing that
\begin{equation}
\big|q_{s}^{\epsilon}(u)-q_{s}(u)\big|^{2}\le\big|X_{s}^{\epsilon}(u)-X_{s}(u)\big|^{2}\le2\left[\big|X_{s}^{\epsilon}(u)-Z_{s}^{\epsilon}(u)\big|^{2}+2\epsilon\big|Y_{s}(u)\big|^{2}\right]\label{eq:difq}\;,
\end{equation}
and then applying Cauchy-Schwarz inequality, we get
\[
A^{\epsilon}(1)\le C\mathrm{e}^{4\rho\kappa\left(|u|^{2}+U(q)\right))}\int_{0}^{\delta^{\epsilon}}\mathbb{E}\left[\mathrm{e}^{8\rho|X_{s}^{\epsilon}(u)|^{2}}\right]^{1/2}\left(\mathbb{E}\left[\big|X_{s}^{\epsilon}(u)-Z_{s}^{\epsilon}(u)\big|^{4}\right]^{1/2}+2\epsilon\mathbb{E}\left[\big|Y_{s}\big|^{4}\right]^{1/2}\right)\mathrm{d}s\;.
\]
Since $u\in\mathcal{D}_{r_{0}}$, by (\ref{eq:cond_r}), for all small
enough $\epsilon>0$,
\[
8\rho<\frac{1}{2(d+2)\kappa_{0}\left(H(u)+\frac{d\epsilon}{\lambda}\right)}\le\frac{1}{2(d+2)\kappa_{0}\left(H(u)\mathrm{e}^{-\lambda s}+\frac{d\epsilon}{\lambda}\right)}\;,
\]
and therefore by Corollary \ref{cor_expmom}, Proposition \ref{prop:main_apprx}
and Lemma \ref{lem:mom_Yt}, we get
\begin{equation}
A^{\epsilon}(1)\le C(n,\,r_{0})\epsilon^{1-2\nu}\delta^{\epsilon}\;.\label{eq:tv5}
\end{equation}
On the other hand, by Theorem \ref{thm:stability}, the second inequality
of (\ref{eq:difq}), Proposition \ref{prop:main_apprx} and Lemma
\ref{lem:mom_Yt}, we also get
\begin{equation}
A^{\epsilon}(2)\le C(n,\,r_{0})\epsilon^{1-2\nu}\delta^{\epsilon}\;.\label{eq:tv6}
\end{equation}
Inserting (\ref{eq:tv1-1}), (\ref{eq:tv2}), (\ref{eq:tv5}) and
(\ref{eq:tv6}) to (\ref{eq:dec_tv1}) yields
\[
\mathrm{d}_{\mathrm{TV}}\left(X_{t^{\epsilon}}^{\epsilon}(x),Z_{\delta^{\epsilon}}^{\epsilon}(X_{t^{\epsilon}-\delta^{\epsilon}}^{\epsilon}(x))\right)\le C(n,\,r_{0})(\delta^{\epsilon}\epsilon^{-2\nu})^{1/2}+\frac{\kappa_{0}}{r_{0}^{2}}\left[H(x)\mathrm{e}^{-\lambda(t^{\epsilon}-\delta^{\epsilon})}+\frac{d\epsilon}{\lambda}\right]\;.
\]
This completes the proof since $\delta^{\epsilon}\epsilon^{-2\nu}=\epsilon^{\sigma-2\nu}$ and $\sigma>2\nu$ by~\eqref{eq:condition nu}. Besides, $\mathrm{e}^{-\lambda(t^{\epsilon}-\delta^{\epsilon})}$ converge
to $0$ as $\epsilon\rightarrow0$.
\end{proof}
Next we control the total variation distance between $Z_{\delta^{\epsilon}}^{\epsilon}(X_{t^{\epsilon}-\delta^{\epsilon}}^{\epsilon}(x))$
and $Z_{t^{\epsilon}}^{\epsilon}(x)$.
\begin{prop}
\label{prop:tv2} Let $(t^{\epsilon})_{\epsilon>0}$ be a sequence
of positive reals satisfying (\ref{eq:tvm1}). Then, for all $r>0$,
it holds that
\[
\limsup_{\epsilon\rightarrow0}\sup_{x\in\mathcal{D}_{r}}\mathrm{d}_{\mathrm{TV}}\left(Z_{\delta^{\epsilon}}^{\epsilon}(X_{t^{\epsilon}-\delta^{\epsilon}}^{\epsilon}(x)),\,Z_{t^{\epsilon}}^{\epsilon}(x)\right)=0\;.
\]
\end{prop}

\begin{proof}
Since $Z_{t^{\epsilon}}^{\epsilon}(x)=Z_{\delta^{\epsilon}}^{\epsilon}(Z_{t^{\epsilon}-\delta^{\epsilon}}^{\epsilon}(x))$, we have
\begin{align}
 & \mathrm{d}_{\mathrm{TV}}\left(Z_{\delta^{\epsilon}}^{\epsilon}(X_{t^{\epsilon}-\delta^{\epsilon}}^{\epsilon}(x)),\,Z_{t^{\epsilon}}^{\epsilon}(x)\right)\nonumber \\
&\leq  \int_{\mathbb{R}^{d}}\mathrm{d}_{\mathrm{TV}}\left(Z_{\delta^{\epsilon}}^{\epsilon}(u),\,Z_{\delta^{\epsilon}}^{\epsilon}(u')\right)\mathbb{P}(X_{t^{\epsilon}-\delta^{\epsilon}}^{\epsilon}(x)\in\mathrm{d}u,\,Z_{t^{\epsilon}-\delta^{\epsilon}}^{\epsilon}(x)\in\mathrm{d}u')\nonumber \\
&\leq \int_{u,\,u'\in\mathcal{D}_{r}}\mathrm{d}_{\mathrm{TV}}\left(Z_{\delta^{\epsilon}}^{\epsilon}(u),\,Z_{\delta^{\epsilon}}^{\epsilon}(u')\right)\mathbb{P}(X_{t^{\epsilon}-\delta^{\epsilon}}^{\epsilon}(x)\in\mathrm{d}u,\,Z_{t^{\epsilon}-\delta^{\epsilon}}^{\epsilon}(x)\in\mathrm{d}u')\nonumber \\
 &\qquad +\mathbb{P}(|X_{t^{\epsilon}-\delta^{\epsilon}}^{\epsilon}(x)|>r)+\mathbb{P}(|Z_{t^{\epsilon}-\delta^{\epsilon}}^{\epsilon}(x)|>r)\;.\label{eq:02}
\end{align}
By the Markov inquality and the Cauchy-Schwarz inequality, we have
\begin{align*}
\mathbb{P}(|X_{t^{\epsilon}-\delta^{\epsilon}}^{\epsilon}(x)|>r) & \le\frac{\mathbb{E}(|X_{t^{\epsilon}-\delta^{\epsilon}}^{\epsilon}(x)|^{2})}{r^{2}}\le\frac{2}{r^{2}}\left[\mathbb{E}(|X_{t^{\epsilon}-\delta^{\epsilon}}^{\epsilon}(x)-Z_{t^{\epsilon}-\delta^{\epsilon}}^{\epsilon}(x)|^{2})+\mathbb{E}(|Z_{t^{\epsilon}-\delta^{\epsilon}}^{\epsilon}(x)|^{2})\right]\;,\\
\mathbb{P}(|Z_{t^{\epsilon}-\delta^{\epsilon}}^{\epsilon}(x)|>r) & \le\frac{1}{r^{2}}\mathbb{E}(|Z_{t^{\epsilon}-\delta^{\epsilon}}^{\epsilon}(x)|^{2})\;.
\end{align*}
Since by the Cauchy-Schwarz inequality we have
\[
\mathbb{E}(|Z_{t^{\epsilon}-\delta^{\epsilon}}^{\epsilon}(x)|^{2})\le2|X_{t^{\epsilon}-\delta^{\epsilon}}(x)|^{2}+4\epsilon\mathbb{E}(|Y_{t^{\epsilon}-\delta^{\epsilon}}(x)|^{2})\;,
\]
we get from Theorem \ref{thm:stability}, Proposition \ref{prop:main_apprx}
and Lemma \ref{lem:mom_Yt} that
\begin{equation}
\mathbb{P}(|X_{t^{\epsilon}-\delta^{\epsilon}}^{\epsilon}(x)|>r)+\mathbb{P}(|Z_{t^{\epsilon}-\delta^{\epsilon}}^{\epsilon}(x)|>r)\le\frac{C(r)}{r^{2}}\left[\epsilon+\textup{e}^{-\lambda(t^{\epsilon}-\delta^{\epsilon})}\right]\;.\label{eq:03}
\end{equation}

Next we compute the integral at the right-hand side of (\ref{eq:02}).
By Lemma \ref{lem:tvZ_short} and Proposition \ref{prop:main_apprx},
this integral is bounded by
\begin{align*}
 & \int_{u,\,u'\in\mathcal{D}_{r}}\left[\frac{C(r)}{(\delta^{\epsilon})^{3/2}\epsilon^{1/2}}|u-u'|+\delta^{\epsilon}\right]\mathbb{P}(X_{t^{\epsilon}-\delta^{\epsilon}}^{\epsilon}(x)\in\mathrm{d}u,\,Z_{t^{\epsilon}-\delta^{\epsilon}}^{\epsilon}(x)\in\mathrm{d}u')\\
& \le \frac{C(r)}{(\delta^{\epsilon})^{3/2}\epsilon^{1/2}}\mathbb{E}\left[|X_{t^{\epsilon}-\delta^{\epsilon}}^{\epsilon}(x)-Z_{t^{\epsilon}-\delta^{\epsilon}}^{\epsilon}(x)|\right]+\delta^{\epsilon}
 \le \frac{C(r)}{(\delta^{\epsilon})^{3/2}}\left[\frac{\textup{e}^{-\alpha_{1}(t^{\epsilon}-\delta^{\epsilon})}}{\epsilon^\nu}+\epsilon^{\alpha_{2}}\right]+\delta^{\epsilon}\;.
\end{align*}
By this computation, (\ref{eq:02}), and (\ref{eq:03}), we can conclude
that
\[
\mathrm{d}_{\mathrm{TV}}\left(Z_{\delta^{\epsilon}}^{\epsilon}(X_{t^{\epsilon}-\delta^{\epsilon}}^{\epsilon}(x)),\,Z_{t^{\epsilon}}^{\epsilon}(x)\right)\le C(r)\left[\epsilon+e^{-\lambda(t^{\epsilon}-\delta^{\epsilon})}+\frac{1}{(\delta^{\epsilon})^{3/2}}\left(\frac{\textup{e}^{-\alpha_{1}(t^{\epsilon}-\delta^{\epsilon})}}{\epsilon^\nu}+\epsilon^{\alpha_{2}}\right)\right]+\delta^{\epsilon}\;.
\]
The right-hand side tends to $0$ by (\ref{eq:tvm1}),
 (\ref{eq:consig}) and~\eqref{eq:condition nu}.
\end{proof}
Now we are ready to prove Proposition \ref{prop:tvm1}.
\begin{proof}[Proof of Proposition \ref{prop:tvm1}]
By the triangle inequality, we have
\[
\mathrm{d}_{\mathrm{TV}}\left(X_{t^{\epsilon}}^{\epsilon},\,Z_{t^{\epsilon}}^{\epsilon}\right)\le\mathrm{d}_{\mathrm{TV}}\left(X_{t^{\epsilon}}^{\epsilon},\,Z_{\delta^{\epsilon}}^{\epsilon}(X_{t^{\epsilon}-\delta^{\epsilon}}^{\epsilon}(x))\right)+\mathrm{d}_{\mathrm{TV}}\left(Z_{\delta^{\epsilon}}^{\epsilon}(X_{t^{\epsilon}-\delta^{\epsilon}}^{\epsilon}(x)),\,Z_{t^{\epsilon}}^{\epsilon}\right)
\]
and therefore the proof is completed by Propositions \ref{prop:tv1}
and \ref{prop:tv2}.
\end{proof}

\subsection{\label{sec74}Proof of Proposition \ref{prop:tvm2}}
\begin{proof}[Proof of Proposition \ref{prop:tvm2}]
By (\ref{eq:distZ}) and Lemma~\ref{lem:cut-off App 1} (1),
we can write
\begin{align*}
\mathrm{d}_{\mathrm{TV}}\left(Z_{t^{\epsilon}}^{\epsilon}(x),\,\mathcal{N}(0,2\epsilon\Sigma)\right) & =\mathrm{d}_{\mathrm{TV}}\left(\mathcal{N}(X_{t^{\epsilon}}(x),\,2\epsilon\Sigma_{t^{\epsilon}}(x)),\,\mathcal{N}(0,2\epsilon\Sigma)\right)\\
 & =\mathrm{d}_{\mathrm{TV}}\left(\mathcal{N}\left(\frac{X_{t^{\epsilon}}(x)}{\sqrt{2\epsilon}},\,\Sigma_{t^{\epsilon}}(x)\right),\,\mathcal{N}(0,\Sigma)\right)\;.
\end{align*}
Therefore, by the triangle inequality and Lemma~\ref{lem:cut-off App 1} (2),
\begin{align*}
 & \left|\,\mathrm{d}_{\mathrm{TV}}\left(Z_{t^{\epsilon}}^{\epsilon}(x),\mathcal{N}(0,2\epsilon\Sigma)\right)-\mathrm{d}_{\mathrm{TV}}\left(\mathcal{N}\left(\frac{X_{t^{\epsilon}}(x)}{\sqrt{2\epsilon}},\,\Sigma\right),\,\mathcal{N}(0,\Sigma)\right)\right|\\
& \le  \mathrm{d}_{\mathrm{TV}}\left(\mathcal{N}\left(\frac{X_{t^{\epsilon}}(x)}{\sqrt{2\epsilon}},\,\Sigma_{t^{\epsilon}}(x)\right),\,\mathcal{N}\left(\frac{X_{t^{\epsilon}}(x)}{\sqrt{2\epsilon}},\,\Sigma\right)\right)=\mathrm{d}_{\mathrm{TV}}\left(\mathcal{N}(0,\Sigma_{t^{\epsilon}}(x)),\,\mathcal{N}(0,\Sigma)\right)\;.
\end{align*}
By Lemma \ref{lem:long_Sigma} and the fact that $t^{\epsilon}\rightarrow\infty$
(by (\ref{prop:tvm2})), the right-hand side converges to $0$, and
hence, it suffices to prove that
\begin{equation}
\limsup_{\epsilon\rightarrow0}\left|\mathrm{d}_{\mathrm{TV}}\left(\mathcal{N}\left(\frac{X_{t^{\epsilon}}(x)}{\sqrt{2\epsilon}},\,\Sigma\right),\,\mathcal{N}(0,\Sigma)\right)-D^{\epsilon}(t^{\epsilon},\,x)\right|=0\;.\label{eq:pup1}
\end{equation}
By Lemma~\ref{lem:cut-off App 1} (3),
\[
\mathrm{d}_{\mathrm{TV}}\left(\mathcal{N}\left(\frac{X_{t^{\epsilon}}(x)}{\sqrt{2\epsilon}},\,\Sigma\right),\,\mathcal{N}(0,\Sigma)\right)=\mathrm{d}_{\mathrm{TV}}\left(\mathcal{N}\left(\Sigma^{-1/2}\frac{X_{t^{\epsilon}}(x)}{\sqrt{2\epsilon}},\,\mathbb{I}_{2d}\right),\,\mathcal{N}(0,\,\mathbb{I}_{2d})\right)\;.
\]
Hence, recalling the definitions (\ref{eq:def_v(t)}) of $v(t,\,x)$
and (\ref{eq:def_D(t)}) of $D^{\epsilon}(t,\,x)$, and applying the
triangle inequality along with Lemmas~\ref{lem:cut-off App 1} (2) and~\ref{lem:cut-off App 2},
we can bound the absolute value at (\ref{eq:pup1}) from above by
\begin{align*}
 & \left|\,\mathrm{d}_{\mathrm{TV}}\left(\mathcal{N}\left(\Sigma^{-1/2}\frac{X_{t^{\epsilon}}(x)}{\sqrt{2\epsilon}},\,\mathbb{I}_{2d}\right),\,\mathcal{N}(0,\,\mathbb{I}_{2d})\right)-\mathrm{d}_{\mathrm{TV}}\left(\mathcal{N}\left(\frac{v(t^{\epsilon},\,x)}{\sqrt{2\epsilon}},\,\mathbb{I}_{2d}\right),\,\mathcal{N}(0,\,\mathbb{I}_{2d})\right)\right|\\
& \le  \mathrm{d}_{\mathrm{TV}}\left(\mathcal{N}\left(\frac{1}{\sqrt{2\epsilon}}\left(\Sigma^{-1/2}X_{t^{\epsilon}}(x)-v(t^{\epsilon},\,x)\right),\,\mathbb{I}_{2d}\right),\,\mathcal{N}\left(0,\,\mathbb{I}_{2d}\right)\right)\\
&\le  \frac{1}{\sqrt{4\pi\epsilon}}\left|\Sigma^{-1/2}X_{t^{\epsilon}}(x)-v(t^{\epsilon},\,x)\right|\\
&=  \frac{1}{\sqrt{4\pi\epsilon}}\left|\frac{(t^{\epsilon}-\tau(x))^{\nu}}{\mathrm{e}^{\eta(t^{\epsilon}-\tau(x))}}\Sigma^{-1/2}\left(\frac{\mathrm{e}^{\eta(t^{\epsilon}-\tau(x))}}{(t^{\epsilon}-\tau(x))^{\nu}}X_{t^{\epsilon}}-\sum_{k=1}^{m}\mathrm{e}^{i\theta_{k}(t^{\epsilon}-\tau(x))}v_{k}\right)\right|\;.
\end{align*}
The last line converges to $0$ as $\epsilon\rightarrow0$ by
  Lemma \ref{lem:asym_uld} and since by (\ref{eq:tvm2}) we have
\[
\limsup_{\epsilon\rightarrow\infty}\frac{(t^{\epsilon})^{\nu}}{\epsilon^{1/2}\mathrm{e}^{\eta t^{\epsilon}}}<\infty\;.
\]

\end{proof}

\section{\label{sec8}Proof of Theorem \ref{thm:asymstat}}

It only remains to prove Theorem \ref{thm:asymstat} to conclude this
work. The proof of this theorem is also based on Propositions \ref{prop:tvm1}
and \ref{prop:tvm2}. In addition to these propositions, we shall
need the following a priori variance bound on the stationary density
$\mu_{\epsilon}$.
\begin{lem}
\label{lem:control tail mu_epsilon} There exists a constant $c>0$
independent of $\epsilon$ such that
\[
\int_{\mathbb{R}^{2d}}|x|^{2}\mu_{\epsilon}(\mathrm{d}x)\leq c\epsilon\;.
\]
\end{lem}

\begin{proof}
This proof is similar to the proof provided in \cite{cut-off_overdamped}
and is inspired from \cite[p.122, Section 5]{priola2012exponential}.
Let $L>0$. Then, for all $t\geq0$, by the stationarity of $\mu_{\epsilon}$
and Jensen's inequality we get
\begin{equation}
\int_{\mathbb{R}^{2d}}(|x|^{2}\land L)\mu_{\epsilon}(\mathrm{d}x)=\int_{\mathbb{R}^{2d}}\mathbb{E}\left[|X_{t}^{\epsilon}(x)|^{2}\land L\right]\mu_{\epsilon}(\mathrm{d}x)\leq\int_{\mathbb{R}^{2d}}\left(\mathbb{E}\left[|X_{t}^{\epsilon}(x)|^{2}\right]\land L\right)\mu_{\epsilon}(\mathrm{d}x)\;.\label{eq:bdsq1}
\end{equation}
We bound the last expectation using Proposition \ref{prop:mom_est}.
Then letting $t\rightarrow\infty$ and applying the dominated convergence
theorem, we get
\[
\int_{\mathbb{R}^{2d}}(|x|^{2}\land L)\mu_{\epsilon}(\mathrm{d}x)\le\int_{\mathbb{R}^{2d}}\left(c_{3}\epsilon\land L\right)\mu_{\epsilon}(\mathrm{d}x)\;.
\]
Finally, letting $L\rightarrow\infty$ and applying the monotone convergence
theorem completes the proof.
\end{proof}
The main ingredients in the proof of Theorem \ref{thm:asymstat} are
Propositions \ref{prop:tvm1}, \ref{prop:tvm2}, and the following
lemma providing a rough a priori bound $\frac{1}{\epsilon^{\theta}}$
on the mixing time.
\begin{lem}
\label{lem:roughmix}Let $t_{\epsilon}=\frac{1}{\epsilon^{\theta}}$
where $\theta$ is the constant defined in Notation \ref{not:T(x)}.
Then, for all $x\in\mathbb{R}^{d}$, we have that
\[
\lim_{\epsilon\rightarrow0}\mathrm{d}_{\mathrm{TV}}\left(X_{t^{\epsilon}}^{\epsilon}(x),\,\mu^{\epsilon}\right)=0\;.
\]
\end{lem}

\begin{proof}
We fix $x\in\mathbb{R}^{2d}$. Denote by $p_{t}^{\epsilon}(\cdot,\,\cdot)$
the transition kernel of the process $(X_{t}^{\epsilon})_{t\geq0}$.
Then, we have
\begin{align*}
\mathrm{d}_{\mathrm{TV}}\left(X_{t^{\epsilon}}^{\epsilon}(x),\,\mu^{\epsilon}\right) & =\frac{1}{2}\int_{\mathbb{R}^{2d}}\left|p_{t^{\epsilon}}^{\epsilon}(x,\,z)-\mu^{\epsilon}(z)\right|\textup{d}z\\
 & =\frac{1}{2}\int_{\mathbb{R}^{2d}}\left|\int_{\mathbb{R}^{2d}}\left(p_{t^{\epsilon}}^{\epsilon}(x,\,z)-p_{t^{\epsilon}}^{\epsilon}(y,\,z)\right)\mu^{\epsilon}(y)\,\textup{d}y\right|\textup{d}z\\
 & \le\frac{1}{2}\int_{\mathbb{R}^{2d}}\int_{\mathbb{R}^{2d}}\left|p_{t^{\epsilon}}^{\epsilon}(x,\,z)-p_{t^{\epsilon}}^{\epsilon}(y,\,z)\right|\mu^{\epsilon}(y)\,\textup{d}y\textup{d}z\\
 & =\int_{\mathbb{R}^{2d}}\mathrm{d}_{\mathrm{TV}}\left(X_{t^{\epsilon}}^{\epsilon}(x),\,X_{t^{\epsilon}}^{\epsilon}(y)\right)\mu^{\epsilon}(\mathrm{d}y)\;.
\end{align*}
Take $R$ large enough so that $x\in\mathcal{D}_{R}$. By Lemma \ref{lem:control tail mu_epsilon}
and the Markov inequality, there exists $c>0$ such that
\[
\mu^{\epsilon}(\mathcal{D}_{R}^{c})\leq\frac{\int_{\mathbb{R}^{2d}}|x|^{2}\mu_{\epsilon}(\mathrm{d}x)}{R^{2}}\leq\frac{c\epsilon}{R^{2}}
\]
and therefore we get
\[
\mathrm{d}_{\mathrm{TV}}\left(X_{t^{\epsilon}}^{\epsilon}(x),\,\mu^{\epsilon}\right)\le\int_{\mathcal{D}_{R}}\mathrm{d}_{\mathrm{TV}}\left(X_{t^{\epsilon}}^{\epsilon}(x),\,X_{t^{\epsilon}}^{\epsilon}(y)\right)\mu^{\epsilon}(\mathrm{d}y)+\frac{c\epsilon}{R^{2}}\;.
\]
By Proposition \ref{prop:tvm1},
\[
\limsup_{\epsilon\rightarrow0}\int_{\mathcal{D}_{R}}\left|\mathrm{d}_{\mathrm{TV}}\left(X_{t^{\epsilon}}^{\epsilon}(x),\,X_{t^{\epsilon}}^{\epsilon}(y)\right)-\mathrm{d}_{\mathrm{TV}}\left(Z_{t^{\epsilon}}^{\epsilon}(x),\,Z_{t^{\epsilon}}^{\epsilon}(y)\right)\right|\mu^{\epsilon}(\mathrm{d}y)=0
\]
since both $x$ and $y$ at the integrand belong to the set $\mathcal{D}_{R}$.
On the other hand, by Lemma \ref{lem:tvZ_long}, we have that (note
that $x\in\mathcal{D}_{R}$)
\[
\lim_{\epsilon\rightarrow0}\int_{\mathcal{D}_{R}}\mathrm{d}_{\mathrm{TV}}\left(Z_{t^{\epsilon}}^{\epsilon}(x),\,Z_{t^{\epsilon}}^{\epsilon}(y)\right)\mu^{\epsilon}(\mathrm{d}y)=0
\]
and the proof is completed.
\end{proof}
Now we prove Theorem \ref{thm:asymstat}.
\begin{proof}[Proof of Theorem \ref{thm:asymstat}]
Let $t^{\epsilon}=\frac{1}{\epsilon^{\theta}}$ where $\theta$ is
the constant defined in Notation \ref{not:T(x)}. Then, by the triangle
inequality, for any $x\in\mathbb{R}^{2d}$
\[
\mathrm{d}_{\mathrm{TV}}\left(\mathcal{N}(0,2\epsilon\Sigma),\,\mu^{\epsilon}\right)\le\mathrm{d}_{\mathrm{TV}}\left(\mathcal{N}(0,2\epsilon\Sigma),\,Z_{t^{\epsilon}}^{\epsilon}(x)\right)+\mathrm{d}_{\mathrm{TV}}\left(Z_{t^{\epsilon}}^{\epsilon}(x),\,X_{t^{\epsilon}}^{\epsilon}(x)\right)+\mathrm{d}_{\mathrm{TV}}\left(X_{t^{\epsilon}}^{\epsilon}(x),\,\mu^{\epsilon}\right)\;.
\]
By Proposition \ref{prop:tvm1} and Lemma \ref{lem:roughmix}, the
second and the third term at the right-hand side converge to $0$
as $\epsilon\rightarrow0$. On the other hand, by the triangle inequality,
the first term at the right-hand side is bounded from above by
\[
\left|\mathrm{d}_{\mathrm{TV}}\left(\mathcal{N}(0,2\epsilon\Sigma),\,Z_{t^{\epsilon}}^{\epsilon}(x)\right)-D^{\epsilon}(t^{\epsilon})\right|+D^{\epsilon}(t^{\epsilon})\;.
\]
By Proposition \ref{prop:tvm1}, the first term tends to $0$ as $\epsilon\rightarrow0$,
while it is easy to check that $D^{\epsilon}(t^{\epsilon})\rightarrow0$
for $t^{\epsilon}=\frac{1}{\epsilon^{\theta}}$ as $\epsilon\rightarrow0$.
This completes the proof.
\end{proof}
 





 \begin{acks}[Acknowledgments]
 S. Lee, M. Ramil, and I. Seo were supported by the National Research Foundation of Korea (NRF) grant funded by the Korean government (MEST) No. 2023R1A2C100517311 and Samsung Science and Technology Foundation (Project Number SSTF-BA1901-03). S. Lee and I. Seo were also supported by NRF grant funded by the Korean government No. 2016K2A9A2A1300381525. I. Seo was supported by the NRF grant funded by the Korean government No. 2019R1A6A1A10073437 and a Seoul National University Research Grant in 2023. I. Seo thanks KIAS for their support as a visiting scholar from March 2024 to February 2025. 
 \end{acks}

\begin{funding}
M. Ramil and S. Lee were supported
by the Samsung Science and Technology Foundation (Project Number SSTF-BA1901-03).
I. Seo was supported by the Samsung Science and Technology Foundation
(Project Number SSTF-BA1901-03) and National Research Foundation(NRF)
of Korea grant funded by the Korea government(MSIT) (No. 2022R1F1A106366811,
2022R1A5A600084012, and 2023R1A2C100517311). The authors are also supported by BK21 SNU Mathematical Sciences Division. The auhors are grateful to the two referees for their careful reading and valuable suggestions.
\end{funding}

\bibliographystyle{imsart-number} 
\bibliography{bibliography}       





\end{document}